\documentclass[a4paper, 11pt]{amsart}

\usepackage[margin=1in,marginparwidth=0.8in, marginparsep=0.1in]{geometry}

\usepackage{graphicx}
\usepackage{marginnote}

\usepackage[british,UKenglish,USenglish,american]{babel}

\usepackage{newlfont}
\usepackage{amssymb}
\usepackage{amsmath}
\usepackage{amsfonts}
\usepackage{latexsym}
\usepackage{amsthm}
\usepackage{mathrsfs}
\usepackage{stmaryrd}
\usepackage[all,cmtip]{xy}
\usepackage{caption}
\usepackage{enumerate}

\UseRawInputEncoding

\usepackage{amsmath,amssymb,amscd,xypic}
\usepackage{amsfonts}
\usepackage[leqno]{amsmath}
\usepackage{mathrsfs} 
\usepackage{enumerate}

\usepackage[usenames,dvipsnames]{xcolor}

\usepackage{layout}
\usepackage{fullpage}

\usepackage[backref=page]{hyperref}

\hypersetup{
 colorlinks,
 citecolor=Green,
 linkcolor=blue,
 urlcolor=Blue}

\theoremstyle{plain}                    
\newtheorem{teo}{Theorem}[subsection]     
\newtheorem{theoremalpha}{Theorem}

\newtheorem{prop}[teo]{Proposition}

\newtheorem{cor}[teo]{Corollary}       
\newtheorem{lem}[teo]{Lemma}            
\theoremstyle{definition}               
\newtheorem{notations}[teo]{}
\newtheorem{defin}[teo]{Definition}

\theoremstyle{remark}
\newtheorem{rmk}[teo]{Remark}

\newenvironment{sis}{\left\{\begin{aligned}}{\end{aligned}\right.}

\newcommand{\bbN}{{\mathbb N}}
\newcommand{\bbC}{{\mathbb C}}
\newcommand{\bbQ}{{\mathbb Q}}
\newcommand{\bbZ}{{\mathbb Z}}
\newcommand{\bbP}{{\mathbb P}}
\newcommand{\bbG}{{\mathbb G}}

\newcommand{\cB}{{\mathcal B}}
\newcommand{\cC}{{\mathcal C}}

\renewcommand{\cL}{{\mathcal L}}

\newcommand{\cM}{{\mathcal M}}
\newcommand{\cN}{{\mathcal N}}
\newcommand{\cO}{{\mathcal O}}
\newcommand{\cP}{{\mathcal P}}

\newcommand{\cS}{{\mathcal S}}

\newcommand{\cX}{{\mathcal X}}

\newcommand{\un}{\underline}
\newcommand{\ov}{\overline}
\newcommand{\wt}{\widetilde}

\renewcommand{\rm}{\mathrm}
\newcommand{\scr}{\mathscr}

\numberwithin{equation}{subsection}

\newcommand{\oo}{\mathcal{O}}
\newcommand{\mt}{\mathcal}

\DeclareMathOperator{\Pic}{Pic}
\DeclareMathOperator{\Aut}{Aut}
\DeclareMathOperator{\Spec}{Spec}

\DeclareMathOperator{\RPic}{RPic}
\DeclareMathOperator{\NS}{NS}
\DeclareMathOperator{\Hom}{Hom}

\DeclareMathOperator{\coker}{coker}
\renewcommand{\Im}{\text{Im}}
\DeclareMathOperator{\End}{End}
\DeclareMathOperator{\id}{id}
\DeclareMathOperator{\res}{res}
\DeclareMathOperator{\red}{red}
\DeclareMathOperator{\dt}{dt}
\DeclareMathOperator{\tf}{tf}
\DeclareMathOperator{\rk}{rk}
\DeclareMathOperator{\glue}{glue}
\DeclareMathOperator{\Gm}{\mathbb{G}_{m}}

\DeclareMathOperator{\Ga}{\mathbb{G}_{\mathrm a}}
\DeclareMathOperator{\PGL}{PGL}
\DeclareMathOperator{\GL}{GL}
\DeclareMathOperator{\SL}{SL}

\DeclareMathOperator{\PO}{PO}
\DeclareMathOperator{\Gr}{Gr}
\DeclareMathOperator{\inst}{inst}

\newcommand{\Mg}{\mathcal M_{g,n}}
\newcommand{\Cg}{\mathcal C_{g,n}}

\newcommand{\bg}[1]{\mathrm{Bun}_{#1,g,n}}

\newcommand{\JT}{\mathrm{J}_{T,g,n}}
\newcommand{\JGm}{\mathrm{J}_{\Gm,g,n}}
\DeclareMathOperator{\Bun}{\mathrm{Bun}}

\DeclareMathOperator{\Bil}{Bil}
\DeclareMathOperator{\Quad}{Quad}
\DeclareMathOperator{\Sym}{Sym}
\newcommand{\ev}{\text{ev}}

\renewcommand{\ss}{\mathrm{ss}}
\newcommand{\ab}{\mathrm{ab}}

\newcommand{\ad}{\mathrm{ad}}
\renewcommand{\sc}{\mathrm{sc}}

\newcommand{\g}{\mathfrak g}

\newcommand{\roo}{\mathrm{roots}}
\newcommand{\coroo}{\mathrm{coroots}}
\newcommand{\wei}{\mathrm{weights}}
\newcommand{\cowei}{\mathrm{coweights}}

\newcommand{\cc}[3]{c^{#1,#2}_{#3}}
\newcommand{\ee}[2]{\Gamma^{#1}_{#2}}

\renewcommand\arraystretch{1.5}

\title{The Picard group of the universal moduli stack of principal bundles on pointed smooth curves.}

\author{Roberto Fringuelli}
\address{Roberto Fringuelli, Dipartimento di Matematica, Universit\`a di Roma ``Tor Vergata'', Via della Ricerca Scientifica 1, I-00133 Roma, Italy}
\email{fringuel@mat.uniroma2.it}

\author{Filippo Viviani}
\address{Filippo Viviani,
Dipartimento di Matematica e Fisica,
Universit\`a Roma Tre,
Largo San Leonardo Murialdo,
I-00146 Roma,  Italy }
\email{viviani@mat.uniroma3.it}

\begin{document}

\begin{abstract}
For any smooth connected linear algebraic group $G$ over an algebraically closed field $k$, we describe the Picard group of the universal moduli stack of principal $G$-bundles over pointed  smooth $k$-projective curves.
\end{abstract}



\maketitle

\tableofcontents

\section{Introduction}
The moduli stack $\Bun_G(C)$ of (principal) $G$-bundles, where $G$ is a complex reductive group,  over a connected, smooth and projective complex curve $C$ has been deeply studied because of its relation to the Wess-Zumino-Witten (=WZW) model associated to $G$. Such models form a special class of rational conformal field theories, see \cite{Bea94}, \cite{Sor96} and \cite{Bea96} for nice surveys. In the WZW-model associated to a simply connected group $G$, the spaces of conformal blocks can be interpreted  as  spaces of  generalized theta functions, that is spaces of global sections of suitable line bundles (e.g., powers of determinant line bundles) on $\Bun_G(C)$, see \cite{BL, KNR94, Fal94, Pau96, LS97}. This interpretation lead to a rigorous mathematical proof of the Verlinde formula and the factorization rules \cite{Ver88} for the WZW-model  of a simply connected  group $G$, see \cite{TUY, Fal94, Tel96}.  
For some (partial, as far as we understand) results on the Verlinde formula for the WZW-model of a non simply connected group $G$, see \cite{Pan94, Bea97, FS99, AMW00, Opr11, KM13}.

The above application to conformal field theory leads naturally to the study of the Picard group of $\Bun_G(C)$ (and of the, closely related,  good moduli space of semistable $G$-bundles),  a question which  makes sense over an arbitrary (say algebraically closed) field $k$. Thanks to the effort of many mathematicians, we have by now a complete understanding of the Picard group of every connected component $\Bun^{\delta}_G(C)$ of $\Bun_G(C)$ (where $\delta\in \pi_1(G)$) over an arbitrary field $k=\ov k$: the case of $\SL_r$ dates back to Drezet-Narasimhan in the late eighties \cite{DN}, where they generalize earlier work of Seshadri \cite{Ramanan1973}; the case of a simply connected, almost-simple $G$ is dealt with in \cite{KN97, LS97, SO99, Fa03, BK05} using the uniformization of $G$-bundles on a curve $C$ in terms of the affine Grassmannian $\Gr_G$ of $G$; the case of a semisimple almost-simple $G$ is dealt with in \cite{Laszlo, BLS98, Tel98}; the case of an arbitrary reductive group was finally established by Biswas-Hoffmann \cite{BH10}.

The aim of this paper is to determine, for an arbitrary connected and smooth linear algebraic group $G$ (not necessarily reductive) over an \emph{algebraically closed field $k$} (of arbitrary characteristics),  the Picard group of the \emph{universal moduli stack of $G$-bundles} $\bg{G}$ over $n$-pointed curves of genus $g$, which parametrizes  $G$-bundles over families of (connected, smooth and projective) $k$-curves of genus $g\geq 0$ endowed with $n\geq 0$ pairwise disjoint ordered sections. The stack $\bg{G}$ comes equipped with a forgetful morphism $\Phi_G:\bg{G}\to \Mg$ onto the algebraic stack \footnote{Note that $\Mg$ is a Deligne-Mumford stack if and only if  $2g-2+n> 0$, which however we do not assume in this paper.}   $\Mg$ of $n$-pointed curves of genus $g$. 

The stack $\bg{G}$ is an algebraic stack, locally of finite type and smooth over $\Mg$ (see Theorem \ref{beh}) and  its connected  components (which are integral and smooth over $k$) are in functorial bijection with the fundamental group $\pi_1(G)$ (see Theorem \ref{concomp} and Corollary \ref{C:regint}). We will denote the connected components and the restriction of the forgetful morphism by 
$$\Phi_G^{\delta}: \bg{G}^{\delta}\to \Mg, \quad \text{ for any } \delta\in \pi_1(G). $$
Note that the algebraic stack $\bg{G}^{\delta}$ is not, in general, of finite type over $\Mg$ (or, equivalently, over $k$): in Proposition \ref{P:qc=a}, we prove that this happens if and only if the reductive quotient $G^{\red}$ of $G$, i.e., the quotient of $G$ by its unipotent radical, is an algebraic torus. On the other hand, if $G$ is reductive but not a torus, $\bg{G}^{\delta}$ admits an exhaustive chain of $k$-finite type open substacks $\left\{\bg{G}^{[d],\leq m}\right\}_{m\geq 0}$ (called the instability exhaustion) of increasing codimension (see Proposition \ref{P:inst-cover}). 
We also prove that every $k$-finite type open substack of $\bg{G}$ is a quotient stack with the notable exception of $(g,n)=(1,0)$ (see Proposition \ref{P:qstack-cover}).

Our first main result says that we can reduce the computation of $\Pic(\bg{G}^{\delta})$ to the case of a reductive group. More precisely, let $G$ be a connected and smooth linear algebraic group over $k=\ov k$ and let $\red: G\twoheadrightarrow G^{\red}$ be its reductive quotient, i.e., the quotient of $G$ by its unipotent radical. 
 Since $\red$ induces an isomorphism $\pi_1(\red):\pi_1(G)\xrightarrow{\cong} \pi_1(G^{\red})$ at the level of fundamental groups, we get a morphism 
$$
\red_\#:\bg{G}^{\delta}\to \bg{G^{\red}}^{\delta} \quad \text{ for any } \delta\in \pi_1(G)\cong \pi_1(G^{\red})
$$
which is smooth, surjective and of finite type (see Corollary \ref{C:red-ft}).

\begin{theoremalpha}\label{T:thmA}(see Theorem \ref{T:BunG-nr})
For any $ \delta\in \pi_1(G)\cong \pi_1(G^{\red})$, the pull-back homomorphism
	$$
	\red_\#^*:\Pic(\bg{G^{\red}}^{\delta})\xrightarrow{\cong} \Pic(\bg{G}^{\delta}) 
	$$
	is an isomorphism. 
\end{theoremalpha}

The above Theorem also holds true for the stack $\Bun_G(C/S)$, parametrizing $G$-bundles on a family of curves $C\to S$  over an integral regular quotient stack (e.g., algebraic space or scheme) over $k$. In particular, if $C$ is a (connected smooth and projective) curve over $k$,  then we have an isomorphism 
$$
\red_\#^*: \Pic(\mathrm{Bun}_{G^{\red}}^\delta(C/k)) \xrightarrow{\cong}\Pic(\mathrm{Bun}_G^\delta(C/k)) \quad \text{ for any } \delta\in \pi_1(G)\cong \pi_1(G^{\red}). 
$$
Combining this  with the computation of  $\Pic(\mathrm{Bun}_{G^{\red}}^\delta(C/k))$ by Biswas-Hoffmann \cite{BH10}, we get  a  description of the Picard group of the moduli stack of $G$-bundles over a fixed $k$-curve $C$, for any  connected smooth linear algebraic group $G$.

Hence, from now on, we will focus on the Picard group of $\bg{G}^{\delta}$ for a reductive group $G$. 
Note that the Picard group of $\Mg$ is well-known up to torsion. If $\mathrm{char}(k)\neq 2$ is completely known (see \cite{FV} and references therein). The pull-back morphism 
$$(\Phi_G^{\delta})^*:\Pic(\Mg)\to \Pic(\bg{G}^{\delta})$$
is injective since $\Phi_G^{\delta}$  is fpqc and cohomologically flat in degree zero (that we baptize Stein) by Proposition \ref{P:SteinG}. Therefore, we can focus our attention onto the relative Picard group 
$$
\RPic(\bg{G}^{\delta}):=\Pic(\bg{G}^{\delta})/(\Phi_G^{\delta})^*(\Pic(\Mg)).
$$
It turns out that the structure of $\RPic(\bg{G}^{\delta})$ in genus $g\geq 1$ is completely different from the one in genus $g=0$. Let us first consider the case $g\geq 1$. 

A first source of line bundles on $\bg{G}$ comes from the determinant of cohomology $d_{\pi}(-)$ and the Deligne pairing $\langle -,-\rangle_{\pi}$ of line bundles on the universal curve $\pi:\mathcal C_{G,g,n}\to\bg{G}$.
To be more precise,  any character $\chi:G\to\Gm\in \Lambda^*(G):=\Hom(G, \Gm)$ gives rise to a morphism of stacks 
$$
\chi_\#:\bg{G}\to\bg{\Gm}
$$
and, by  pulling back via $\chi_{\#}$ the universal $\Gm$-bundle (i.e., line bundle) on the universal curve over $\bg{\Gm}$, we get a line bundle $\mathcal L_{\chi}$ on  $\mathcal C_{G,g,n}$. Then, using these line bundles $\mathcal L_{\chi}$  and the sections $\sigma_1,\ldots,\sigma_n$ of $\pi$, we define the following line bundles, that we call \emph{tautological line bundles}, on $\bg{G}$ (and hence, by restriction, also on $\bg{G}^{\delta}$)
\begin{equation*}\label{E:tautINT}
\begin{aligned}
& \mathscr L(\chi,\zeta):=d_{\pi}\big(\mt L_\chi(\zeta_1\cdot\sigma_1+\ldots+\zeta_n\cdot\sigma_n)\big),\\
& \langle (\chi,\zeta),(\chi',\zeta')\rangle:=\langle \mt L_\chi(\zeta_1\cdot\sigma_1+\ldots+\zeta_n\cdot\sigma_n), \mt L_{\chi'}(\zeta_1'\cdot\sigma_1+\ldots+\zeta_n'\cdot\sigma_n)\rangle_{\pi},\\
\end{aligned}
\end{equation*}
for $\chi, \chi'\in \text{Hom}(G,\Gm)$ and $\zeta=(\zeta_1,\ldots,\zeta_n), \zeta'=(\zeta'_1,\ldots,\zeta'_n)\in \bbZ^n$. See \S\ref{tau}  for more details.

Our next main result says that if $G=T$ is a torus, then $\RPic(\bg{T}^{\delta})$ is generated by tautological line bundles and the following Theorem also clarifies the dependence relations among the  tautological line bundles.

\begin{theoremalpha}\label{T:thmB} (see Theorem \ref{BunT})
Assume that $g\geq 1$. Let $T$ be an algebraic torus and let $d\in \pi_1(T)$.
              
                The relative Picard group $\RPic(\bg{T}^d)$ is a free abelian group of finite rank generated by the tautological line bundles and sitting in the following functorial exact sequence
		\begin{equation*}\label{seqT-INT}
		0\to\Sym^2 \Lambda^*(T)\oplus\left(\Lambda^*(T)\otimes\bbZ^n\right)\xrightarrow{\tau_T+\sigma_T}\RPic\left(\bg{T}^d\right)\xrightarrow{\rho_T} \Lambda^*(T)\to 0 \quad \text{ if } g\geq 2,
		\end{equation*}
		\begin{equation*}\label{seqT2-INT}
		0\to\Sym^2 \Lambda^*(T) \oplus\left(\Lambda^*(T)\otimes\bbZ^n\right)\xrightarrow{\tau_T+\sigma_T}\RPic\left(\mathrm{Bun}_{T,1,n}^d\right)\xrightarrow{\rho_T} \frac{\Lambda^*(T)}{2\Lambda^*(T)} \to 0 \:\: \text{ if } g=1,
		\end{equation*}
		where $\tau_T(=\tau_{T,g,n})$ (called \emph{transgression map}) and $\sigma_T(=\sigma_{T,g,n})$ are defined by 
		$$
		\begin{array}{cclll}
		\tau_T(\chi\cdot\chi')&=&\langle(\chi,0),(\chi',0)\rangle, &\text{ for any } \chi, \chi'\in\Lambda^*(T),\\
		\sigma_T(\chi\otimes \zeta)&=&\langle(\chi,0),(0,\zeta)\rangle, &\text{ for any }\chi\in\Lambda^*(T) \text{ and }\zeta\in\bbZ^n,\\
		\end{array}
		$$
		and $\rho_T(=\rho_{T,g,n})$ is the unique homomorphism such that
		$$ 
		\rho_T(\mathscr L(\chi,\zeta))=
		\begin{cases}
		\chi \in \Lambda^*(T) & \text{ if } g\geq 2,\\
		[\chi] \in \frac{\Lambda^*(T)}{2\Lambda^*(T)}& \text{ if } g=1,\\
		\end{cases} \quad \text{ for any } \chi\in\Lambda^*(T) \text{ and }\zeta\in\bbZ^n.
		$$
\end{theoremalpha}
The above Theorem was proved in  \cite[Thm. A]{MV} (see also \cite[Notation 1.5]{MV}) for $T=\Gm$, $g\geq 2$ and $n=0$, under the assumption that $\text{char}(k)=0$, 
and in \cite[Thm. 2.9]{PTT15} for $T=\Gm$, $g=1$, $n=0$ and $d>0$, under the assumption that $\text{char}(k)$ does not divide $d$.

The crucial ingredients in proving  Theorem \ref{T:thmB} are  the study of the restriction homomorphism  from $\RPic(\bg{T}^d)$ to the Picard group of the geometric fibers of $\Phi_T^d$ computed in \cite{BH10} (which we determine in  Propositions \ref{P:restr} and \ref{P:restrPic}) and the computation of the Picard group of the rigidification $\bg{T}^0\fatslash T$  of $\bg{T}^d$ by the torus $T$ (which is an abelian stack over $\Mg$ if  $n\geq 1$, see \S\ref{S:RPicT}), which uses the weak Franchetta  conjecture for the universal family $\Cg\to \Mg$ (see  \S\ref{S:Franchetta}). Note that the extra relations in genus one come from the relative Serre duality applied to the Deligne pairing (see Remark \ref{R:taut-dual}) and the fact that the relative dualizing sheaf is trivial in genus one. 

Now consider the case of an arbitrary reductive group $G$. Note that any character of $G$ factors through its maximal abelian quotient $\ab:G\twoheadrightarrow G^{\ab}$, i.e., the quotient of $G$ by its derived subgroup. 
Hence, the tautological line bundles on $\bg{G}^{\delta}$ are all pull-backs of line bundles via the morphism (induced by $\ab$)
$$\ab_\#:\bg{G}^{\delta}\to \bg{G^{\ab}}^{\delta^{\ab}}$$ 
where  $\delta^{\ab}:=\pi_1(\ab)(\delta)\in \pi_1(G^{\ab})$. Moreover, Theorem \ref{T:thmB} implies that the subgroup of $\RPic(\bg{G}^{\delta})$ generated by the tautological line bundles coincides with the pull-back of  $\RPic(\bg{G^{\ab}}^{\delta^{\ab}})$ via $\ab_\#$.

The next result says that, for an arbitrary reductive group $G$, the relative Picard group of $\bg{G}^{\delta}$ is generated by the image of the pull-back $\ab_\#^*$ together with the image of a functorial transgression map $\tau_G$ (which coincides with the transgression map $\tau_T$ in Theorem \ref{T:thmB} if $G=T$ is a torus).

\begin{theoremalpha}\label{T:thmC}(see Theorem \ref{T:BunG})
Assume that $g\geq 1$. 	Let $G$ be a reductive group and let $\ab:G\to G^{\ab}$ be its maximal abelian quotient. Choose a maximal torus $\iota: T_G\hookrightarrow G$ and let $\scr W_G$ be the Weyl group of $G$. Fix $\delta \in \pi_1(G)$ and denote by  $\delta^{\ab}$  its image in $\pi_1(G^{\ab})$.
	\begin{enumerate}
\item\label{T:thmC1} There exists a unique injective homomorphism  (called \emph{transgression map}\footnote{This is the algebraic analogue of the topological trasgression map $H^4(B G,\bbZ) \to H^2(\bg{G}^{\delta},\bbZ)$, see \cite[\S 1]{TW09}})
\begin{equation}\label{E:trasgrINT}
\tau_G(=\tau_{G,g,n}):(\Sym^2 \Lambda^*(T_G))^{\mathscr W_G}\hookrightarrow \RPic\Big(\bg{G}^{\delta}\Big),
\end{equation}
such that, for any lift $d\in \pi_1(T_G)$ of $\delta\in \pi_1(G)$, the composition of $\tau_G$ with 
$$\iota_\#^*:\RPic(\bg{G}^{\delta})\to \RPic(\bg{T_G}^d)$$ 
is equal to the $\scr W_G$-invariant part of the homomorphism $\tau_{T_G}: \Sym^2 \Lambda^*(T_G)\to \RPic(\bg{T_G}^d)$ defined in Theorem \ref{T:thmB}. 

\item \label{T:thmC2} There is a push-out diagram of injective homomorphisms of abelian groups
\begin{equation}\label{E:amalgINT}
\xymatrix{
\Sym^2 \Lambda^*(G^{\ab})\ar@{^{(}->}[rr]^{\Sym^2 \Lambda^*_\ab}\ar@{^{(}->}[d]^{\tau_{G^\ab}}&& (\Sym^2\Lambda^*(T_G))^{\mathscr W_G}\ar@{^{(}->}[d]^{\tau_G}\\
\RPic\left(\bg{G^{\ab}}^{\delta^{\ab}}\right)\ar@{^{(}->}[rr]^{\ab_\#^*} &&\RPic\Big(\bg{G}^{\delta}\Big)
}
\end{equation}
where $\Sym^2 \Lambda^*_\ab$ is the homomorphism induced by the morphism of tori $T_G\xrightarrow{\iota} G\xrightarrow{\ab} G^{\ab}$. 
\end{enumerate}
	
Furthermore, the transgression homomorphism \eqref{E:trasgrINT} and the diagram \eqref{E:amalgINT} are contravariant with respect to  homomorphisms of reductive groups $\phi:H\to G$ such that $\phi(T_H)\subseteq T_G$.
\end{theoremalpha}

The above theorem was proved for $G$ simply-connected almost-simple and $n>0$ in \cite[Thm. 17]{Fa03}\footnote{Indeed Faltings shows in loc. cit. that the Picard group of $\Bun_G(C/S)$ for any family of curves $C\to S$ over a connected Noetherian base-scheme $S$ and endowed with a section, with $G$ simply-connected almost-simple, is isomorphic to $(\Sym^2\Lambda^*(T_G))^{\mathscr W_G}\cong \bbZ$ via a ``central charge" homomorphism, which coincides with the inverse of our transgression map $\tau_G$.} and for  $G=GL_r$, $g\geq 2$ and $n=0$ in \cite[Thm. A]{Fri18}, under the assumption $\mathrm{char}(k)=0$.

Let us comment  on the strategy of the proof of Theorem \ref{T:thmC}. The uniqueness of the transgression map $\tau_G$ follows from the injectivity of the pull-back homomorphism 
$$\iota_\#^*:\RPic(\bg{G}^{[d]}) \hookrightarrow \RPic(\bg{T_G}^d)$$
 induced by the morphism $\iota_\#:\bg{T_G}^d\to \bg{G}^{[d]}$ for any $d\in \pi_1(T_G)$ (see Corollary \ref{C:inj}). In order to prove the existence of the transgression map $\tau_G$, we proceed (roughly) as follows (see \S\ref{SS:trasgr}): we first  define, using the known properties of Chow groups of flag bundles (see \S\ref{SS:Ch-Flag}), a homomorphism 
 $$
c_1(\tau_G)_{\bbQ}^{\leq m}:(\Sym^2\Lambda^*(T_G))^{\scr W_G} \to A^1\left(\bg{G}^{[d],\leq m}\right)_{\mathbb Q} \quad \text{for any }m\geq 0,
$$ 
such that the composition with $\iota_\#^*$ (for $m\gg 0$) is the $\scr W_G$-invariant part of the rational first Chern class of $\tau_{T_G}$ (see  Proposition \ref{P:Q-transfr}); then we show that $c_1(\tau_G)_{\bbQ}^{\leq m}$ (for $m\gg 0$) is the rational first Chern class of the restriction of a homomorphism $\tau_G$ as in \eqref{E:trasgrINT} by using  a descent argument applied to the finite type, smooth and with geometrically integral fibers morphism (see Theorem \ref{T:Holla})
$$j_\#:\bg{B_G}^d\to \bg{G}^{[d]},$$ 
where $j:B_G\hookrightarrow G$ is the Borel subgroup of $G$ containing the torus $T_G$.  Finally, the proof of Theorem \ref{T:thmC}\eqref{T:thmC2} is based on a canonical identification of the push-out of the diagram \eqref{E:amalgINT}
with the subgroup of invariants $\RPic(\bg{T_G}^d)^{\scr W_G}$, for a suitable action of the Weyl group $\scr W_G$ on $\RPic(\bg{T_G}^d)$ (see \S\ref{SS:WGinv}).

Finally, we consider the genus zero case, which is very different from the positive genus case. 
In order to state the result, let us introduce some notation. 
Let $\ss:G\twoheadrightarrow G^{\ss}$ be the semisimplification of $G$, i.e., the quotient of $G$ by its radical, and let $\sc:G^{\sc}\twoheadrightarrow G^{\ss}$ be the universal cover.
The choice of a maximal torus $T_G\subset G$ determines maximal tori   $T_{G^{\ss}}\subset G^{\ss}$ and $T_{G^{\sc}}\subset G^{\sc}$ in such a way that $\ss(T_G)=T_{G^{\ss}}$ and $\sc(T_{G^{\sc}})=T_{G^{\ss}}$. We get the following morphisms of tori
\begin{equation}\label{E:toriINT}
\sc:T_{G^{\sc}}\xrightarrow{\ov \sc} T_G \xrightarrow{\ss} T_{G^{\ss}}. 
\end{equation}

The following Theorem says that, in the genus zero case, the relative Picard group is completely determined by a suitable weight function.

\begin{theoremalpha}\label{T:thmD} (see Theorems \ref{BunTg0} and \ref{T:BunGg0})
Assume that $g=0$.  Let $G$ be a reductive group and choose a maximal torus $\iota: T_G\hookrightarrow G$. Fix $\delta \in \pi_1(G)$ and choose a lift  $d\in \pi_1(T_G)$ of $\delta$, i.e., such that $[d]=\delta$. 
Consider the homomorphism (called the \emph{weight function} with respect to the representative $d$)
                 \begin{equation}\label{weiG-INT}
		w_G^d:\RPic\left({\mathrm{Bun}}_{G,0,n}^{\delta}\right)\xrightarrow{\iota_\#^*}\RPic\left(\mathrm{Bun}_{T_G,0,n}^d\right)\xrightarrow{w_{T_G}^d} \Lambda^*(T_G),
		\end{equation}
		where the function $w_{T_G}^d$ (which is the weight function for the torus $T_G$) is defined by (using that $\RPic\left(\mathrm{Bun}_{T_G,0,n}^d\right)$ is generated by tautological line bundles)
\begin{equation*}
\begin{sis}
&w_T^d(\scr L(\chi, \zeta))=[(d, \chi)+|\zeta|+1]\chi,  \\
&  w_T^d(\langle (\chi, \zeta), (\chi', \zeta') \rangle)= [(d, \chi')+|\zeta'|]\chi+[(d, \chi)+|\zeta|]\chi'.\\
\end{sis}
\end{equation*}
\begin{enumerate}
\item \label{T:thmD2a} There exists a representative $d\in \pi_1(T_G)$ of $\delta\in \pi_1(G)$ such that $w_G^d$ is injective.

\item \label{T:thmD2b}  
Let $\Omega_d^*(T_G)\subset \Lambda^*(T_G)$ be the subgroup of those characters, whose composition with  $\Lambda^*(\ov\sc):\Lambda^*(T_G)\to \Lambda^*(T_{G^{\sc}})$ is equal to $b(d^\ss,-)$ for some element $b\in(\Sym^2\Lambda^*(T_{G^{\sc}}))^{\scr W_G}$, where $d^{\ss}:=\pi_1(\ss)(d)\in \pi_1(T_{G^{\ss}})=\Lambda(T_{G^{\ss}})\subseteq \Lambda(T_{G^{\sc}})$.  Then the image of $w_G^d$ is 
		  $$
		  \Im(w_G^d)=
		  \begin{cases}
		  \Omega_d^*(T_G)&\text{if }n\geq 1,\\
		  \{\chi\in  \Omega_d^*(T_G)\: :  (\chi,d)\in 2\mathbb Z\}& \text{if }n=0.
		  \end{cases}
		  $$
\end{enumerate}
Furthermore, the homomorphism $w_G^d$ is functorial   for all the homomorphisms of reductive groups $\phi:H\to G$ such that $\phi(T_H)\subseteq T_G$.
\end{theoremalpha}

The above Theorem is proved for $T=\Gm$ and $n=0$ in  \cite[Prop. 2.6]{PTT15}. As we will see, the case $G$ reductive and $n=3$ is a direct consequence of a result in \cite{BH10}.

\vspace{0.1cm}
In a sequel \cite{FV2} to this paper,  we compute the Picard group of the rigidification $\bg{G}^{\delta}\fatslash Z(G)$ of $\bg{G}^{\delta}$ by the center $Z(G)$ of $G$ and the divisor class group of  the good moduli space of semistable $G$-bundles (extending the work of \cite{MV} for $\Gm$ and of \cite{Fri18} for $\GL_r$). This is closely related to  the computation of the subgroup associated to the $Z(G)$-gerbe  $\bg{G}^{\delta}\to \bg{G}^{\delta}\fatslash Z(G)$ in the  Brauer group of the codomain (extending the work of \cite{FPbr} for $\GL_r$). Moreover, we also give alternative presentations of the Picard group of $\bg{G}^{\delta}$ and we study the restriction homomorphism onto the Picard group of the moduli stack of principal $G$-bundles over a fixed smooth curve (as computed in \cite{BH10}).


\vspace{0.1cm}

Finally, let us now discuss some potential applications of the main results of this paper (and of its sequel \cite{FV2}), which we will come back to in the near future.
First of all, taking a line bundle $\cL$ on $\bg{G}^{\delta}$ (for a reductive group $G$) which is relatively positive in the sense of \cite[(8.5)]{Tel00}, and, using the vanishing \cite[Thm. 8.8]{Tel00} of the relative higher cohomology groups of $\cL$, it follows that  $(\Phi_G^{\delta})_*(\cL)$ is a vector bundle on $\Mg$. For $G$ simply connected with Lie algebra $\mathfrak g$, the relative Picard group of $\bg{G}=\bg{G}^{0}$ is freely generated by a line bundle $\mathcal N$, satisfying a suitable positivity condition. In this case, $(\Phi_G^{\delta})_*(\cL)$ is the (well-known) vector bundle of conformal blocks attached to the triple $(\mathfrak g,l,\underline{0})$, where $l$ is the unique integer such that $\mathcal N^l=\mathcal L$ and $\underline{0}$ is the $n$-tuple of dominant integral weights corresponding to $(0,\ldots,0)$. However, for $G$ non simply connected, the above vector bundles do not seem to have been studied in full generality, and, in particular, it would be interesting to determine their ranks  (i.e., to establish a general Verlinde formula), to compute their Chern classes and to extend them to the stack  $\ov{\mathcal M}_{g,n}$ of stable pointed curves. Second, it would be natural to study a compactification of the universal stack  of $G$-bundles over $\ov{\mathcal M}_{g,n}$ (see \cite{Bal, Cas, Sch05, Sol} for some recent progress), and to determine its Picard group as was done for $\Gm$ in \cite{MV} and for $\GL_r$ in \cite{Fri18}. This could have potential applications to the study of the birational geometry of $\bg{G}^{\delta}$, as in \cite{BFV12, CMKV17} for $G=\Gm$.

\subsection*{Notations}

\begin{notations}
We denote by $k=\ov k$ an algebraically closed field of arbitrary characteristic. All the schemes and algebraic stacks that we will appear in this paper will be locally of finite type over $k$
(hence locally Noetherian). 
\end{notations}

\begin{notations}\label{N:famcur}

A \emph{curve} is a connected, smooth and projective scheme of dimension one over $k$.  The genus of a curve $C$ is $g(C):=\dim H^0(C,\omega_C)$.

A \emph{family of curves} $\pi: C\to S$ is a proper and flat morphism of stacks whose geometric fibers are curves. If all the geometric fibers of $\pi$ have the same genus $g$, then we say that $\pi:C\to S$ is a family of curves of genus $g$ (or a family of curves with relative genus $g$) and we set $g(C/S):=g$. Note that any family of curves $\pi:C\to S$ with $S$ connected is a family of genus $g$ curves for some $g\geq 0$.

\begin{rmk}\label{R:proj-fam}
A family of genus $g$ curves $\pi:C\to S$  is projective (i.e., $\pi$ is a projective morphism) if either $g=g(C/S)\neq 1$  or $\pi$ has a section $\sigma:S\to C$. 

Indeed, if  $g(C/S)\neq 1$, then the relative dualizing line bundle $\omega_{\pi}$ is $\pi$-relatively ample if $g(C/S)\geq 2$ or $\pi$-relatively antiample if $g(C/S)=0$. On the other hand, if $\pi$ has a section $\sigma$, then $\Im(\sigma)$ is a relative Cartier divisor which is $\pi$-relatively ample. 
\end{rmk}
Note that the assumptions in the above Remark are really needed since there are examples of families of genus one curves (without sections) that are not projective, see \cite[XIII, 3.2]{Ray} and \cite{Zom}.
\end{notations}

\begin{notations}
Given two integers $g,n\geq 0$, we will denote by $\Mg$ the stack (over $k$) whose fiber over a scheme $S$ is the groupoid of families $(\pi:\cC\to S,\un\sigma=\{\sigma_1,\ldots, \sigma_n\})$ of  \emph{$n$-pointed curves of genus $g$} over $S$, i.e., $\pi:\cC\to S$ is a family of curves of  genus $g$ and $\{\sigma_1,\ldots, \sigma_n\}$ are (ordered) sections of $\pi$ that are fiberwise disjoint.

It is well known that the stack $\Mg$ is an irreducible algebraic stack, smooth and separated over $k$, and of dimension $3g-3+n$. 
Moreover,  $\Mg$ is a DM(=Deligne-Mumford) stack if and only if $3g-3+n>0$.

We will denote by $(\pi_{g,n}=\pi:\Cg\to \Mg, \un \sigma)$ the universal $n$-pointed curve over $\Mg$.

%
\end{notations}

\begin{notations}
A linear algebraic group over $k$ is a group scheme of finite type over $k$ that can be realized as a closed algebraic subgroup of $\GL_n$, or equivalently it is an affine group scheme of finite type over $k$. 
We will be dealing almost always with linear algebraic groups that are smooth (which is always the case if $\mathrm{char}(k)=0$) and connected. 

Given a linear  algebraic group $G$,  a principal $G$-bundle over an algebraic  stack $S$ is a $G$-torsor over $S$, where $G$ acts on the right. 
\end{notations}

\section{Preliminaries}\label{S:prel}

\subsection{Reductive groups}\label{red-grps}

In this subsection, we will collect some result on the structure of reductive groups, that will be used in what follows. An excellent introduction to reductive groups can be found in \cite{Mil}.

Let us first recall that a \emph{reductive group} (over $k$) is a smooth and connected linear algebraic group (over $k$) which does not contain non-trivial connected normal unipotent algebraic subgroups. To any reductive group $G$, we can associate in a canonical way two semisimple groups and two (algebraic) tori:
	\begin{itemize}
		\item  the derived subgroup $\mathscr D(G):=[G,G]$;
		\item  the abelianization $G^{\mathrm{ab}}:=G/\mathscr D(G)$;
		\item  the radical subgroup $\mathscr R(G)$, which is equal (since $G$ is reductive) to the connected component of identity of the center $\scr Z(G)$;
		\item  the semisimplification $G^{\rm{ss}}:=G/\mathscr R(G)$.
	\end{itemize}

The above four reductive groups associated to $G$ fit in a cross-like diagram:
\begin{equation}\label{E:cross}
\xymatrix{
& \scr D(G) \ar@{^{(}->}[d] \ar@{->>}[dr] & \\
\scr R(G) \ar@{^{(}->}[r]  \ar@{->>}[dr] & G \ar@{->>}[d]^{\ab}  \ar@{->>}[r]_{\ss}  & G^{\rm{ss}} \\
& G^{\ab} & \\
}
\end{equation}
where the horizontal and vertical lines are short exact sequences of reductive groups, the upper right diagonal arrow is a central isogeny of semisimple groups and the lower left diagonal arrow is a central isogeny of tori. 


Since the two semisimple groups $\scr D(G)$ and $G^{\ss}$ are isogenous, they share the same simply connected cover, that we will denote by $G^{\sc}$, and the same adjoint quotient, that we will denote by $G^{\ad}$.  Hence we have the following tower of central isogenies of semisimple groups
\begin{equation}\label{E:tower-ss}
G^{\sc}\twoheadrightarrow \scr D(G) \twoheadrightarrow G^{\ss} \twoheadrightarrow G^{\ad}.
\end{equation}
The Lie algebra $\g$ of $G$ splits as 
\begin{equation}\label{E:Lie-split}
\g=\g^{\ab}\oplus \g^{\ss},
\end{equation}
where $\g^{\ab}$ is the abelian Lie algebra of the tori $\scr R(G)$ and $G^{\ab}$, whose dimension is called the abelian rank of $G$, and $\g^{\ss}$ is the semisimple Lie algebra of each of the  semisimple groups in \eqref{E:tower-ss}, whose rank is called the semisimple rank of $G$. The semisimple Lie algebra $\g^{\ss}$ decomposes as a direct sum of simple Lie algebras of classical type (i.e., type $A_n$, $B_n$, $C_n$, $D_n$, $E_6$, $E_7$, $E_8$, $F_4$ or  $G_2$). If $G$ is a semisimple group such that its Lie algebra $\g=\g^{ss}$ is simple, then $G$ is said to be almost-simple.

Recall now that all maximal tori of $G$ are conjugate and let us fix one such maximal torus, that we call $T_G$. We will denote by $B_G$  the  Borel subgroup of $G$ that contains $T_G$ and by $\scr N(T_G)$ the normalizer of $T_G$  in $G$, so that 
\begin{equation}\label{E:Weil} 
\scr W_G:=\mathscr N(T_G)/T_G
\end{equation}
is the Weyl group of $G$.

The maximal torus $T_G$ induces compatible maximal tori of every semisimple group appearing in \eqref{E:tower-ss}, that we will call, respectively, $T_{G^{\sc}}$,  $T_{\scr D(G)}$, $T_{G^{\ss}}$ and $T_{G^{\ad}}$. These tori fit into the following commutative diagram:
\begin{equation}\label{E:crosstori}
\xymatrix{
T_{G^{\sc}} \ar@{->>}[dr] & & & \\
& T_{\scr D(G)} \ar@{^{(}->}[d] \ar@{->>}[dr] & & \\
\scr R(G) \ar@{^{(}->}[r]  \ar@{->>}[dr] & T_G \ar@{->>}[d]  \ar@{->>}[r]  & T_{G^{\rm{ss}}} \ar@{->>}[dr]  & \\
& G^{\ab} & & T_{G^{\ad}}\\
}
\end{equation}
where the horizontal and vertical lines are short exact sequences of tori, and the diagonal arrows are (central) isogenies of tori. Using the canonical realization \eqref{E:Weil} of the Weyl group  (and the similar ones for the semisimple groups in \eqref{E:tower-ss}), diagram \eqref{E:crosstori} induces canonical isomorphisms of Weyl groups 
\begin{equation}\label{E:iso-Weyl}
\scr W_{G^{\sc}}\cong \scr W_{\scr D(G)}\cong \scr W_G \cong \scr W_{G^{\ss}}\cong \scr W_{G^{\ad}}. 
\end{equation} 

Recall now that  a torus $T$ determines two canonical lattices (i.e.,  free abelian groups) of rank equal to the dimension of $T$:
\begin{itemize}
		\item the character lattice $\Lambda^*(T):=\text{Hom}(T,\Gm)$,
		\item the cocharacter lattice $\Lambda(T):=\text{Hom}(\Gm,T)$.
\end{itemize}
The above lattices are in canonical duality via the pairing given by composition 
\begin{equation}\label{E:perpair}
(-,-):\text{Hom}(\Gm,T)\times \text{Hom}(T,\Gm)\longrightarrow \text{Hom}(\Gm,\Gm)=\bbZ.
\end{equation}

By taking the cocharacter lattices of the tori in the diagram \eqref{E:crosstori}, we get the following $\scr W_G$-equivariant commutative diagram of lattices 
\begin{equation}\label{E:tori-cocar}
\xymatrix{
\Lambda(T_{G^{\sc}}) \ar@{^{(}->}[dr] & & & \\
& \Lambda(T_{\scr D(G)}) \ar@{^{(}->}[d] \ar@{^{(}->}[dr] & & \\
\Lambda(\scr R(G)) \ar@{^{(}->}[r]  \ar@{^{(}->}[dr] & \Lambda(T_G) \ar@{->>}[d]^{\Lambda_\ab}  \ar@{->>}[r]_{\Lambda_\ss}  & \Lambda(T_{G^{\rm{ss}}}) \ar@{^{(}->}[dr]  & \\
& \Lambda(G^{\ab}) & & \Lambda(T_{G^{\ad}}),\\
}
\end{equation}
where the horizontal and vertical lines are short exact sequences and  the diagonal arrows are finite index inclusions. Note that there are natural identifications 
$$
\Lambda(T_{G^{\sc}})\cong \Lambda_{\coroo}(\g^{\ss}) \quad \text{ and } \quad \Lambda(T_{G^{\ad}})\cong \Lambda_{\cowei}(\g^{\ss}), 
$$
where $\Lambda_{\coroo}(\g^{\ss}) $ (resp., $\Lambda_{\cowei}(\g^{\ss})$) is the lattice of coroots (resp., of coweights) of the semisimple Lie algebra $\g^{\ss}$.

	
In a similar way, if we take the character lattices of the tori in the diagram \eqref{E:crosstori}, we get the following $\scr W_G$-equivariant commutative diagram of lattices 
\begin{equation}\label{E:tori-car}
\xymatrix{
\Lambda^*(T_{G^{\sc}})  & & & \\
& \Lambda^*(T_{\scr D(G)})  \ar@{_{(}->}[ul]& & \\
\Lambda^*(\scr R(G)) & \Lambda^*(T_G) \ar@{->>}[u]  \ar@{->>}[l]  & \Lambda^*(T_{G^{\rm{ss}}}) \ar@{_{(}->}[ul]   \ar@{_{(}->}[l]^{\Lambda^*_\ss} & \\
& \Lambda^*(G^{\ab}) \ar@{_{(}->}[ul] \ar@{^{(}->}[u]_{\Lambda^*_\ab}& & \Lambda^*(T_{G^{\ad}}) \ar@{_{(}->}[ul],\\
}
\end{equation}
where the horizontal and vertical lines are short exact sequences and  the diagonal arrows are finite index inclusions. Note that there are natural identifications 
$$
\Lambda^*(T_{G^{\sc}})\cong \Lambda_{\wei}(\g^{\ss}) \quad \text{ and } \Lambda(T_{G^{\ad}})\cong \Lambda_{\roo}(\g^{\ss}), 
$$
where $\Lambda_{\wei}(\g^{\ss}) $ (resp., $\Lambda_{\roo}(\g^{\ss})$) is the lattice of weights (resp., of roots) of the semisimple Lie algebra $\g^{\ss}$.

	
Note that the two diagrams \eqref{E:tori-cocar} and \eqref{E:tori-car}, together with the root system of the semisimple Lie algebra $\g^{\ss}$, are equivalent to the root data of the reductive group $G$ (see \cite[\S 19]{Mil}), and hence it completely determines the reductive group $G$. 

The fundamental group $\pi_1(G)$ of $G$ is canonically isomorphic to $\Lambda(T_G)/\Lambda(T_{G^{\sc}})$ and it fits into the following short exact sequence of finitely generated abelian groups 
\begin{equation}\label{E:seq-pi1}
 \frac{\Lambda(T_{\scr D(G)})}{\Lambda(T_{G^{\sc}})}\hookrightarrow \pi_1(G)=\frac{\Lambda(T_G)}{\Lambda(T_{G^{\sc}})}\twoheadrightarrow \Lambda(G^{\ab}),
\end{equation}
where the first term is  the torsion subgroup of $\pi_1(G)$ and the last term is  the torsion-free quotient of $\pi_1(G)$. 


We end this subsection with the following Lemma that will be useful later on.

\begin{lem}\label{L:inv-char}
Let $G$ be a reductive  group with maximal torus $T_G\subset G$  and consider the natural action of the Weyl group  $\scr W_G$ on $\Lambda^*(T_G)$. Then we have an isomorphism 
$$\Lambda_{\ab}^*:\Lambda^*(G^{\ab})\xrightarrow{\cong} \Lambda^*(T_G)^{\scr W_G}.$$
\end{lem}
\begin{proof}
By taking the $\scr W_G$-invariants in the short exact $\scr W_G$-equivariant sequence in the central column of \eqref{E:tori-car}  and using that $\scr W_G\cong \scr W_{\scr D(G)}$ acts trivially on 
$\Lambda^*(G^{\ab})$, we get an exact sequence
$$
0\to \Lambda^*(G^{\ab}) \xrightarrow{\Lambda_{\ab}^*} \Lambda^*(T_G)^{\scr W_G} \to \Lambda^*(T_{\scr D(G)})^{\scr W_{\scr D(G)}}
$$
Hence, it is enough to show that $\Lambda^*(T_{\scr D(G)})^{\scr W_{\scr D(G)}}=0$. With this aim, take an element $\lambda\in \Lambda^*(T_{\scr D(G)})^{\scr W_{{\scr D}(G)}}\subset \Lambda^*(T_{G^{\sc}})\cong \Lambda_{\wei}(\g^{\ss})$ and write it as rational linear combination of fundamental weights of $\g^{\ss}$, i.e., $\lambda=\sum_{i=1}^r b_i\epsilon_i$, with $b_i\in\mathbb Q$. For any $i$, using standard properties of the reflection $s_{\alpha_i}\in \scr W_{\scr D(G)}$ (resp., $s_{-\alpha_i}$) associated to the simple root $\alpha_i$ (resp., $-\alpha_i$) and the invariance of $\lambda$ under the action of the Weyl group $\scr W_{\scr D(G)}$, we have that 
	$$
	b_i=(\lambda,\alpha_i^\vee)=(s_{-\alpha_i}(\lambda),\alpha_i^\vee)=(\lambda,s_{\alpha_i}(\alpha_i^\vee))=(\lambda,-\alpha_i^\vee))=-b_i,
	$$
which implies $\lambda=0$, as required. 
\end{proof}

\subsection{Integral bilinear symmetric forms and integral quadratic forms}\label{Sec:int-forms} 

In this subsection, we review some facts on integral bilinear symmetric forms and integral quadratic forms.

Let $\Lambda$ be a lattice of rank $r$, i.e., $\Lambda\cong \bbZ^r$.
We put integral bilinear (symmetric) forms and quadratic forms on $\Lambda$ (or $\Lambda$-integral)
\begin{equation}\label{D:intform}
\begin{aligned}
& \Bil \Lambda:=\left\{B:\Lambda\times\Lambda\to \bbZ \quad \text{ such that } B \: \text{ is bilinear}\right\},\\
& \Bil^s\Lambda:=\left\{B:\Lambda\times\Lambda\to \bbZ \quad \text{ such that } B \: \text{ is bilinear and symmetric}\right\},\\
& \Quad\Lambda:=\left\{Q:\Lambda\to \bbZ \quad  \text{ such that } Q \: \text{ is quadratic} \right\},
\end{aligned}
\end{equation}
where, by definition, a quadratic form $Q$ satisfies: $Q(a\cdot x)=a^2Q(x)$ for any  $a\in \bbZ, x\in \Lambda$ and $(x,y)\mapsto Q(x+y)-Q(x)-Q(y)$ is a bilinear form on $\Lambda$.

Integral symmetric bilinear forms and quadratic forms on $\Lambda$ are related by the following maps
\begin{equation}\label{E:bs-quad}
\xymatrix{ 
\Bil^s\Lambda \ar@/^/[rr]^q && \Quad\Lambda \ar@/^/[ll]^b\\
B \ar@{|->}[rr] && Q_B(x):=B(x,x) \\
B_Q(x,y):=Q(x+y)-Q(x)-Q(y) &&  Q \ar@{|->}[ll] 
}
\end{equation}
The maps $q$ and $b$ are injective with image given by even symmetric bilinear or quadratic forms
\begin{equation}\label{E:evforms}
\begin{aligned}
\Im(b)=(\Bil^s\Lambda)^{\ev}:=\{B\in \Bil^s\Lambda: \: B(x,x) \: \text{ is even for any } x\in \Lambda\},\\
\Im(q)=(\Quad\Lambda)^{\ev}:=\{Q\in \Quad\Lambda: \: Q(x+y)-Q(x)-Q(y) \: \text{ is even for any } x, y\in \Lambda\}.\\
\end{aligned}
\end{equation}
Note that both the compositions $q\circ b$ and $b\circ q$ are multiplication by $2$.

Now we reinterpret the above constructions in terms of the dual lattice $\Lambda^*:=\Hom(\Lambda, \bbZ)$. 
On the tensor product $\Lambda^*\otimes \Lambda^*$ there is an involution $i$ defined by $i(\chi\otimes \mu)=\mu\otimes \chi$. Consider the following lattices 
\begin{equation}\label{D:dual-latt}
 (\Lambda^*\otimes \Lambda^*)^s:=(\Lambda^*\otimes \Lambda^*)^i\subset \Lambda^*\otimes \Lambda^* \quad \text{ and } \quad \Sym^2 \Lambda^*:=\frac{\Lambda^*\otimes \Lambda^*}{\langle \chi\otimes \mu-\mu\otimes \chi\rangle}.
\end{equation}
We will denote the elements of $\Sym^2\Lambda^*$ in the following way: $\chi\cdot \mu:=[\chi\otimes \mu]$. 

The lattices in \eqref{D:dual-latt} are isomorphic to the lattices of  integral (symmetric) bilinear forms and quadratic forms on $\Lambda$, respectively, via the following isomorphisms:
\begin{equation}\label{E:iso-b}
\xymatrix{ 
(\Lambda^*\otimes \Lambda^*)^s \ar[r]^{\cong} \ar@{^{(}->}[d] & \Bil^s\Lambda\ar@{^{(}->}[d] \\
\Lambda^*\otimes \Lambda^* \ar[r]^{\cong} & \Bil \Lambda\\
\chi\otimes \mu \ar@{|->}[r] & (\chi\otimes \mu)(x,y):=\chi(x)\mu(y),
}
\text{ and } \quad 
\begin{aligned}
\\
\Sym^2 \Lambda^* & \stackrel{\cong}{\longrightarrow} \Quad\Lambda,\\
\chi\cdot\mu& \mapsto (\chi\cdot \mu)(x):=\chi(x)\mu(x).
\end{aligned}
\end{equation}
In terms of the isomorphisms \eqref{E:iso-b}, the maps in \eqref{E:bs-quad} take the following form
\begin{equation}\label{E:bs-quad2}
\xymatrix{ 
(\Lambda^*\otimes \Lambda^*)^s \ar@/^/[rr]^q && \Sym^2\Lambda^* \ar@/^/[ll]^b\\
\sum_k a_k \chi_k\otimes \mu_k \ar@{|->}[rr] &&\sum_k a_k \chi_k\cdot \mu_k \\
\chi\otimes \mu+\mu\otimes \chi &&  \chi\cdot \mu \ar@{|->}[ll] 
}
\end{equation} 
If we fix a basis $\{\chi_i\}_{i=1}^r$ of $\Lambda^*$, then a basis of $(\Lambda^*\otimes \Lambda^*)^s$ is given by $\left\{\{\chi_i\otimes \chi_i\}_i\cup\{\chi_i\otimes \chi_j+\chi_j\otimes \chi_i\}_{i<j}\right\}$, while a basis of $\Sym^2 \Lambda^*$ is given by $\{\chi_i\cdot \chi_j\}_{i\leq j}$.  Using the above basis, it follows that 
\begin{equation*}\label{E:coker}
\coker(b)=(\bbZ/2\bbZ)^r \quad \text{ and }  \quad \coker(q)=(\bbZ/2\bbZ)^{\binom{r}{2}}.
\end{equation*}

We end this subsection with some results (that we will need later on) on $\scr W_{G^{\sc}}$-invariant   quadratic forms on the lattice of cocharacters $\Lambda(T_{G^{\sc}})$, where $T_{G^{\sc}}$ is a (fixed) maximal torus   of a \emph{simply connected semisimple group}  $G^{\sc}$. As in \S\ref{red-grps}, we will denote by $G^{\sc}\twoheadrightarrow G^{\ad}$ the adjoint quotient of $G^{\sc}$, i.e., the quotient of $G^{\sc}$ by its finite center. 
 
From what we said above, there are canonical isomorphisms 
\begin{equation}\label{E:bilquad}
b:\Sym^2\Lambda^*(T_{G^{\sc}})=\Quad \Lambda(T_{G^{\sc}})\xrightarrow{\cong} (\Bil^s\Lambda(T_{G^{\sc}}))^\ev.
\end{equation}
Note that this isomorphism is equivariant with respect to the natural action of the  Weyl group $\scr W_{G^{\sc}}$ on both sides. 

We are interested in the invariant subgroup $(\Sym^2 \Lambda^*(T_{G^{\sc}}))^{\scr W_{G^{\sc}}}$, which therefore parametrizes $\scr W_{G^{\sc}}$-invariant even symmetric bilinear forms  (or, equivalently, $\scr W_{G^{\sc}}$-invariant quadratic forms) on  $\Lambda(T_{G^{\sc}})\cong \Lambda_{\coroo}(\g^{\ss})$.  Let us first compute its rank. 

\begin{lem}\label{L:Sym2inv}
\noindent 
\begin{enumerate}[(i)]
\item \label{L:Sym2inv1} If $G^{\sc}$ is almost simple, then $(\Sym^2\Lambda^*(T_{G^{\sc}}))^{\scr W_{G^{\sc}}}$  is freely generated by an even symmetric bilinear  form $B_{G^{\sc}}$ (called the \emph{basic inner product} of $G^{\sc}$), which is non-degenerate and it satisfies $B_{G^{\sc}}(\alpha^\vee, \alpha^\vee)=2$ for all short coroots $\alpha^{\vee}$.
\item \label{L:Sym2inv2} If $G^{\sc}={G_1^{\sc}}\times \ldots \times {G_s^{\sc}}$ is the decomposition of $G^{\sc}$ into almost simple factors, then 
$$ (\Sym^ 2\Lambda^*(T_{G^{\sc}}))^{\scr W_{G^{\sc}}}=(\Sym^ 2\Lambda^*(T_{{G_1^{\sc}}}))^{\scr W_{{G_1^{\sc}}}}\oplus \ldots \oplus (\Sym^ 2\Lambda^*(T_{{G_s^{\sc}}}))^{\scr W_{{G_s^{\sc}}}}.$$
\end{enumerate}
\end{lem}
\begin{proof}
Part \eqref{L:Sym2inv1} is the content of \cite[Lemma 1.7.5]{Del97} and part \eqref{L:Sym2inv2} can be proved as in \cite[(1.8.3)]{Del97}.
\end{proof}

Observe that we have a $\scr W_{G^{\sc}}$-equivariant inclusion of lattices $\Lambda(T_{G^{\sc}})\subset \Lambda(T_{G^\ad})$, see \eqref{E:tori-cocar}. Hence, we can extend any (resp., $\scr W_{G^{\sc}}$-invariant) symmetric bilinear form $B:\Lambda(T_{G^{\sc}})\times \Lambda(T_{G^{\sc}})\to\mathbb Z$ to a unique rational  (resp., $\scr W_{G^{\sc}}$-invariant)  symmetric bilinear form $B':\Lambda(T_{G^\ad})\times \Lambda(T_{G^\ad})\to \mathbb Q$, called its rational extension.

\begin{lem}
If $B\in(\Sym^2\Lambda^*(T_{G^{\sc}}))^{\scr W_{G^{\sc}}}$, then its rational extension $B'$ on $\Lambda(T_{G^\ad})\times \Lambda(T_{G^\ad})$ is integral in $\Lambda(T_{G^\ad})\times \Lambda(T_{G^{\sc}})$ and $\Lambda(T_{G^{\sc}})\times \Lambda(T_{G^\ad})$, i.e., $B'(\Lambda(T_{G^\ad})\times \Lambda(T_{G^{\sc}})+(\Lambda(T_{G^{\sc}})\times \Lambda(T_{G^\ad}))\subseteq \bbZ$.
\end{lem}
\begin{proof}See \cite[Lemma 4.3.4]{BH10}.
\end{proof}

The above Lemma allows us to define the \emph{contraction homomorphism} associated to any element $d\in \Lambda(T_{G^{\ad}})$:
\begin{equation}\label{E:ctrc-hom}
\begin{array}{cccl}
(d,-):&(\Sym^2\Lambda^*(T_{G^{\sc}}))^{\scr W_{G^{\sc}}}&\to& \Lambda^*(T_{G^{\sc}})\\
&B&\mapsto& B(d,-):=B'(d,-)
\end{array}
\end{equation}
where $B'$ is the rational extension of $B$. 

We will need the following Lemma, which tells us when the contraction homomorphism $(d,-)$  is injective.

\begin{lem}\label{L:inj-cont}
Let $G^{\ad}=G_1^{\ad}\times \ldots \times G_s^{\ad}$ be the decomposition of $G^{\ad}$ into almost simple factors (which are then automatically semisimple adjoint groups) and choose maximal tori in such a way that $T_{G^{\ad}}=T_{G_1^{\ad}}\times \ldots \times T_{G_s^{\ad}}$. 
\begin{enumerate}[(i)]
\item \label{L:inj-cont1} The contraction homomorphism $(d,-)$ associated to an element $d\in \Lambda(T_{G^{\ab}})$ is injective if and only if $d$ satisfies the following condition
\begin{equation*}
d=d_1+\ldots+d_s\in \Lambda(T_{G^{\ad}})= \Lambda(T_{G_1^{\ad}})\oplus \ldots \oplus \Lambda(T_{G_s^{\ad}}) \text{ with } d_i\neq 0 \text{ for every } 1\leq i \leq s. \tag{*}
\end{equation*}
\item \label{L:inj-cont2} For every $\delta\in \pi_1(G^{\ad})=\frac{\Lambda(T_{G^{\ad}})}{\Lambda(T_{G^{\sc}})}$, there exists a representative $d\in \Lambda(T_{G^{\ad}})$ of $\delta$, i.e., $[d]=\delta$, that satisfies condition (*).
\end{enumerate}
\end{lem}
\begin{proof}
It follows easily from Lemma \ref{L:Sym2inv},  see the proof of \cite[Lemma 4.3.6]{BH10}.
\end{proof}

\subsection{Picard groups of algebraic stacks}\label{pic-stacks}

In this subsection, we will collect some facts about the Picard group of algebraic  stacks (locally of finite type over $k$).

The first result describes the behavior of the Picard group under restriction to open substacks. 

\begin{lem}\label{restr}
Let $\mathcal X$ be a regular algebraic  stack  and $\mathcal U\subset\mathcal X$ be an open substack. Then the restriction morphism 
	$$
	\text{Pic}(\mathcal X)\longrightarrow \text{Pic}(\mathcal U)
	$$
is surjective and it is an isomorphism if  the codimension of the complement $\mathcal X\backslash\mathcal U$ is at least two.
\end{lem}
\begin{proof}
See \cite[Lemma 7.3]{BH12}.
\end{proof}

The next result (which is a generalization of \cite[Lemma 5.2]{AV}) will be used several times in what follows in order to  descend a line bundle along a flat morphism of finite type.

\begin{prop}\label{P:fiber}
Let $\mathcal X$ be a regular algebraic  stack,  flat and of finite type over a regular integral stack $\cS$ that is generically a scheme. Assume that the fibers of $f:\mathcal X\to \cS$ are  integral. Then we have an exact sequence of Picard groups
	$$
	\text{Pic}(\cS)\xrightarrow{f^*} \text{Pic}(\mathcal X)\xrightarrow{\res_{\eta}} \text{Pic}(\cX_{\eta})\to 0,
	$$
	where $\eta:=\text{Spec}(k(\cS))$ is the generic point of $\cS$ and $\res_{\eta}$ is the restriction to the generic fiber $\cX_{\eta}:=\cX\times_{\cS} \eta$. 
\end{prop}
\begin{proof}It can be proven using the same arguments in the proof of \cite[Lemma 5.2]{AV}, which deals with the special case $\cS=\Spec R$ with $R$ a unique factorization domain (in which case $\Pic(\cS)=0$).

The proof of the surjectivity of $\res_{\eta}$  is the same as in loc. cit., using the fact that $\cX$ is regular. 

The inclusion $\Im(f^*)\subseteq \ker(\res_{\eta})$ is obvious. Let us sketch a proof of the inclusion $\Im(f^*)\supseteq \ker(\res_{\eta})$, adapting the argument of loc. cit..
Let $L$ be a line bundle on $\mathcal X$ which is trivial on the generic fiber $\cX_{\eta}$. Choose a nowhere vanishing section of $L|_{\cX_{\eta}}$; this will extend to a  section $s$ over some open substack $\mathcal U\subset \mathcal X$ containing the generic fiber. Let $D$ be the Cartier divisor defined by $s$. In particular, we have that $\oo_{\cX}(D)\cong L$.  Write $D=\sum_i n_i D_i$ with $n_i\in \bbZ$ and $D_i$ prime divisors. Since $f$ is equidimensional and dominant (being flat over an integral base) and  the image $E_i:=f(D_i)$ does not contain the generic point $\eta$ (by construction), we conclude that $E_i$ is a prime divisor of $\cS$. Since  the fibers of $f$ are integral, we must have that $D_i=f^{-1}(E_i)$. Each prime divisor $E_i$ is Cartier because $\cS$ is regular and, by what proved above, $f^*\cO_{\cS}(\sum_i n_i E_i)=\cO_{\cX}(\sum_i n_iD_i)=L$, which concludes the proof. 
\end{proof}

The following condition on a morphism of stacks will play an important role in what follows. Recall that a morphism $\pi:\mt X\to \mt S$ of algebraic  stacks is \emph{fpqc}, if it is faithfully flat and, for any point $x\in \mt X$, there exists an open neighborhood $\mt U\subset \mt X$ such that the image $\pi(\mt U)\subset \mt S$ is open and the restriction $\pi_{|\mt U}:\mt U \to \pi(\mt U)$ is quasi-compact (see \cite[Sec. 2.3.2]{Vis05}). For example, any morphism $\pi:\mt X\to \mt S$ that is faithfully flat and locally of finite presentation (i.e., it is fppf) is fpqc. 

\begin{defin}\label{D:Stein}
Let $\pi:\mt X\to\mt S$ be a morphism of algebraic  stacks. We will say that $\pi$ is \textbf{Stein} if it is fpqc  and the natural homomorphism $(\pi_{\mt T})^\#:\oo_{\mt T}\to(\pi_{\mt T})_*\oo_{\mt X_{\mt T}} $ is an isomorphism for any arbitrary base change $\mt T\to\mt S$, where $\pi_{\mt T}:\mt X_{\mt T}=\mt X\times_{\mt S} \mt T\to \mt T$ is the base change of $\pi$  through the  morphism $\mt T\to\mt S$.
\end{defin}

For a Stein morphism $\pi:\mt X\to\mt S$, we have that $\pi_*(\Gm)=\Gm$ and hence the pull-back map  $\pi^*:\Pic(\mt S)\to \Pic(\mt X)$ on Picard groups is injective.  
We will need the following  Lemma on the pull-back of the relative Picard group
\begin{equation}\label{E:RelPic}
\RPic(\mt X/ \mt S):=\Pic(\mt X)/\pi^*(\Pic(\mt S)),
\end{equation}
along an fpqc morphism. 

\begin{lem}\label{base-change} 
Let $\pi:\mt X\to \mt S$ be a Stein morphism of algebraic  stacks. Then, for any fpqc morphism $\mt S'\to \mt S$, the pull-back homomorphism
	$$
	\RPic(\mt X/\mt S)\to\RPic(\mt X_{\mt S'}/\mt S') 
	$$
of relative Picard groups is injective.
\end{lem}
\begin{proof}
The Leray spectral sequences for the fppf sheaf $\Gm$ with respect to the morphisms $\pi$ and $\pi_{\mt S'}$ give the following commutative diagram of groups:
	$$\xymatrix{
		H^0(\mt S,R^1\pi_{*}\Gm)\ar[r]& H^0(\mt S',R^1(\pi_{\mt S'})_*\Gm)\\
		\Pic(\mt X)/\Pic(\mt S)\ar[r]\ar@{^{(}->}[u]& \Pic(\mt X_{\mt S'})/\Pic(\mt S').\ar@{^{(}->}[u]
	}
	$$
	with injective vertical arrows. 
 The top horizontal arrow  is injective because $\mt S'\to\mt S$ is fpqc and $R^1\pi_{*}\Gm$ is a sheaf for the fpqc topology.  From the above commutative diagram, the bottom horizontal arrow is also injective.
 \end{proof}

Finally, we need to recall some facts about the Picard group of a \textbf{quotient stack}  over $k$, which in this text will always mean an algebraic stack of finite type over $k$ of the form $[X/G]$ with $X$ an algebraic space of finite type over $k$ and $G$ a (smooth) linear algebraic group over $k$. 

Recall that for a quotient stack $\cX=[X/G]$ the Picard group $\Pic(\cX)$ coincides with the equivariant Picard group $\Pic^G(X)$ (i.e., the group of $G$-linearized line bundles on $X$) while the first operational Chow group $A^1(\cX)$ coincides with the first equivariant operational Chow group $A_G^1(X)$ (as defined in \cite[Sec. 2.6]{EG98}).

\begin{prop}\label{Equi}(Edidin-Graham \cite{EG98})
Let $\mt X=\left[ X/G \right]$ be a quotient stack over $k$.
\begin{enumerate}[(i)]
\item \label{Equi1} If $\cX$ is locally factorial (e.g., if it is smooth), then there exists an algebraic space $Y$ of finite type over $k$ (called  an \emph{equivariant approximation} of $[X/G]$) together with a smooth, surjective, finite type morphism $f:Y\to \cX$ that induces an isomorphism
$$
f^*:\Pic(\cX)\xrightarrow{\cong} \Pic(Y).
$$
\item \label{Equi2} If $\cX$ is smooth, the first Chern class 
$$c_1:\Pic(\cX)\to A^1(\cX)$$
is an isomorphism. 
\item \label{Equi3} Assume $p:\mathcal C\to \mathcal X$ is either a representable morphism of locally factorial algebraic stacks  or a smooth morphism of regular stacks, we have the following commutative diagram of groups
$$
\xymatrix{
\Pic(\mathcal X)\ar[d]_{p^*}\ar[r]_{f^*}^\cong&\Pic(Y)\ar[d]\\
\Pic(\mathcal C)\ar[r]^(0.4)\cong&\Pic(\mathcal C\times_{\mathcal X} Y)
}
$$
where $f:Y\to\mathcal X$ is any equivariant approximation of $\mathcal X$.
\end{enumerate}
\end{prop}
\begin{proof}
Part \eqref{Equi2} follows from \cite[Cor. 1]{EG98}.

Part \eqref{Equi1}: pick a representation $V$ of $G$ such that $G$ acts freely on an open subset $U$ of $V$ whose complement has codimension at least  $2$ and set $Y:=X\times U/G$, which is an algebraic space (of finite type over $k$) since $G$ acts freely on $X\times U$. Consider the morphism 
$$f:Y=X\times U/G\xrightarrow{i} [X\times V/G]\xrightarrow{h} [X/G]=\cX,$$
where $i$ is induced by the inclusion $U\hookrightarrow V$ and $h$ is induced by the first projection  $X\times V\to X$. From \cite[Lemma 2]{EG98}, we deduce that the pull-back map $f^*$ induces an isomorphism 
$$f^*:\Pic(\cX)=\Pic^G(X)\xrightarrow[\cong]{h^*} \Pic^G(X\times V)\xrightarrow[\cong]{i^*}  \Pic^G(X\times U)=\Pic(Y).
$$

Part \eqref{Equi3}: Assume first that $p$ is a representable morphism of locally factorial stacks. Then there exists a morphism of $G$-schemes $\pi: C\to X$, such that $\pi/G:[C/G]\to [X/G]$ is exactly the morphism $p$. Then the assertion follows by the construction of $Y$ in the proof of part \eqref{Equi1} and the observation that $\mathcal C\times_{\mathcal X} Y\cong C\times U/G$.

Assume now that $p$ is smooth and $\cC$ and $\mt X$ are regular. By construction $f$ is the composition of an open immersion $i:Y\hookrightarrow \mt W$ of regular stacks whose complement has codimension at least $2$ and a vector bundle $h:\mt W\to \mt X$. So the same happens to $\widetilde f=\widetilde h\circ \widetilde i:\cC\times_{\mt X}Y\to\cC$. Observe that $\widetilde h^*$, resp., $\widetilde i^*$, is an isomorphism because $\widetilde h$ is a vector bundle, resp., by Lemma \ref{restr}. Hence the same holds for $\widetilde f^*$, concluding the proof.
\end{proof}

\subsection{Chow groups of flag bundles}\label{SS:Ch-Flag}
In this section, we collect some facts about Chow groups of flags bundles (i.e bundles of flag varieties). As usual, $G$ is a reductive group, and we fix a  Borel subgroup $B=B_G\subset G$ and a maximal torus $T=T_G\subset B$.

The flag variety $G/B$ is a smooth projective variety of dimension $N:=\text{dim}G-\text{dim}B$. The quotient $G/T$ is an affine bundle on $G/B$. Let $E$ be a $G$-bundle over a scheme $C$. Since $G$ acts on $E$, by taking the quotient with respect to $B$ and $T$, we obtain the flag bundle $E/B\to C$ and its affine bundle $E/T$, respectively. They sit in the following cartesian diagram (on the left) 
\begin{equation}\label{E:diagr-Bred}
\xymatrix{
	\mathcal B T \ar[d]\ar@{}[dr]|\square& \ar[l] E/T\ar[d] & & A^*(\mathcal BT)\ar[r]& A^*(E/T)\\
	\mathcal B B \ar[d] \ar@{}[dr]|\square &\ar[l] E/B\ar[d] &\Longrightarrow & A^*(\mathcal BB)\ar[u]^{\cong}\ar[r] &A^*(E/B)\ar[u]_{\cong}\\
	\mathcal BG&\ar[l] C & & A^*(\mathcal BG) \ar[u]\ar[r] & A^*(C)\ar[u]
}
\end{equation}
which induces a commutative diagram (on the right) at the level of operational Chow rings $A^*(-)$.  The isomorphisms in the right diagram come from the fact that both the vertical maps in the upper square of the left diagram are vector bundles. Furthermore, it is well-known (see e.g., \cite[Lemma 2 and 3]{EG97}) that
$$\Sym^*\Lambda^*(T)\cong A^*(\mathcal BT) \cong A^*(\mathcal BB).$$
Hence, we have a well-defined commutative diagram of homomorphisms of graded rings:
\begin{equation}\label{E:diagr-Bred2}
\xymatrix{
\Sym^*\Lambda^*(T)\ar[rrd]_{\cc{G}{T}{E}}\ar[rr]^{\cc{G}{B}{E}}&& A^*(E/B)\ar[d]^{\cong}\\
& & A^*(E/T)
}
\end{equation}
where the vertical arrow is the isomorphism given by the pull-back along $E/T\to E/B$. The next result, due to Brion and Edidin-Graham, will be fundamental in the computation of the Picard group of our main object.

\begin{prop}\label{P:Chow-flag} Let $E\to C$ be a $G$-bundle over a scheme and let $p:E/B\to C$ be the associated flag bundle. Then for any  $v\in(\Sym^*\Lambda^*(T_G))^{\scr W_G}$, 
\begin{enumerate}[(i)]
\item \label{Chow-flag-i}we have the equality in $A^*(E/B)_{\mathbb Q}$
\begin{equation}\nonumber
p^*p_*\cc{G}{B}{E}\left(v\cdot\prod_{\alpha>0}\alpha\right)=|\scr W_G|\cdot \cc{G}{B}{E}(v),
\end{equation}
where $\{\alpha>0\}\subset \Lambda^*(T_G)$ is the set of positive roots with respect to the Borel subgroup $B$;
\item \label{Chow-flag-ii}  if $E\to C$ is locally trivial for the Zariski topology (e.g., if it admits a $B$-reduction), there exists a unique integral class $\ee{G}{E}(v)\in A^*(C)$, functorial with respect to base changes of $E\to C$, such that:
\begin{enumerate}
\item \label{Chow-flag-iia} we have the equality $p^*\ee{G}{E}(v)=\cc{G}{B}{E}(v)$ in $A^*(E/B)$,
\item\label{Chow-flag-iib}  we have the equality 
$|\scr W_G|\cdot \ee{G}{E}(v)=p_*\cc{G}{B}{E}\left(v\cdot\prod_{\alpha>0}\alpha\right)
$ in $A^*(C)$, 
\end{enumerate}
\end{enumerate}
\end{prop}
\begin{proof}
Point  \eqref{Chow-flag-i}:  see \cite[Proposition 1.2]{BrionTodd}. Point  \eqref{Chow-flag-ii}: the assertion  \eqref{Chow-flag-iia}  is exactly the content of \cite[Theorem 1]{EG97}. Consider the (integral) class in $A^*(E/B)$:
$$m:=|\scr W_G|\cdot p^*\ee{G}{E}(v)-p^*p_*\cc{G}{B}{E}\left(v\cdot\prod_{\alpha>0}\alpha\right)=|\scr W_G|\cdot \cc{G}{B}{E}(v)-p^*p_*\cc{G}{B}{E}\left(v\cdot\prod_{\alpha>0}\alpha\right).$$ 
We claim that $m=0$ in $A^1(E/B)$. Indeed, by diagram \eqref{E:diagr-Bred2}, the class $m$ is the pull-back of a class from $A^*(\mathcal BB)\cong\Sym^*\Lambda^*(T)$ which is torsion-free. Furthermore, by applying \eqref{Chow-flag-i}  to the flag bundle $\mathcal BB\to \mathcal BG$ (see Remark \ref{R:Brion-stack}), we get that $m$ must be a torsion class, hence is zero. By hypothesis, the bundle $p:E/B\to C$ is a Chow envelope. In particular, the pull-back $p^*$ of integral Chow groups is injective. Hence, $m=0$ which implies \eqref{Chow-flag-iib}.
\end{proof}

\begin{rmk}\label{R:Brion-stack}The point $(i)$ of the above proposition is still true if we assume that $C$ is an algebraic stack of finite type over $k$. It is essentially due to the fact that the computation in \cite[Proposition 1.2]{BrionTodd} uses Grothendieck-Riemann-Roch Theorem, which still holds in our setting because $p$ is representable, by diagram \eqref{E:diagr-Bred}.
\end{rmk}
Using the diagram (\ref{E:diagr-Bred2}), we get the analogous result for the $G/T$-bundle $q: E/T\to C$.
\begin{cor}\label{C:Chow-flag} With the assumptions of Proposition \ref{P:Chow-flag}, we have
$$
\frac{1}{|\scr W_G|}q^*p_*\cc{G}{B}{E}\left(v\cdot\prod_{\alpha>0}\alpha\right)=\cc{G}{T}{E}(v)\in A^*(E/T)_{\mathbb Q}\quad \text{ for any } v\in(\Sym^*\Lambda^*(T_G))^{\scr W_G}. 
$$
\end{cor}
\begin{rmk} Observe that, even if the class $|\scr W_G|^{-1}\prod_{\alpha>0}\alpha$ is not integral, its image in $A^N(G/B)$ defines an integral class, which is  indeed the Poincar\'e dual of the closed orbit $[B/B]\in G/B$ (see \cite{Dem74}). The map $\cc{G}{B}{G}:\Sym^N\Lambda^*(T_G)\to A^N(G/B)\cong\mathbb Z$ in degree $N$ is not surjective for some reductive groups $G$. The order of the cokernel is the so-called \emph{torsion index}. So, for an arbitrary reductive group $G$,  there may not exist an integral class in $\Sym^N \Lambda^*(T_G)$ whose image in $A^N(G/B)$ is the class $[B/B]\in G/B$.
\end{rmk}

\subsection{Weak Franchetta conjecture}\label{S:Franchetta}

The aim of this section is to recall the so called  ``weak Franchetta Conjecture", which  is the computation of the relative Picard of the universal curve $\pi: \mt C_{g,n}\to \mt M_{g,n}$ (over our fixed base field $k=\ov k$):
$$
\operatorname{RelPic}(\mathcal C_{g,n}):=\operatorname{Pic}(\mathcal C_{g,n})/\pi^*\operatorname{Pic}(\mathcal M_{g,n}).
$$

\begin{teo}[weak Franchetta Conjecture]\label{franchetta}
 The group $\operatorname{RelPic}(\mathcal C_{g,n})$ is generated by the relative dualizing line bundle $\omega_{\pi}$ and the line bundles 
$\{\oo(\sigma_1), \ldots, \oo(\sigma_n)\}$ associated to the universal sections $\sigma_1,\ldots,\sigma_n$, subject to the following relations:
\begin{itemize}
\item if $g=1$ then $\omega_{\pi}=0$;
\item if $g=0$ then $\oo(\sigma_1)=\ldots= \oo(\sigma_n)$ and $\omega_{\pi}=\oo(-2\sigma_1)$.
\end{itemize}
\end{teo}
The above result was proved for $g\geq 3$ by Arbarello-Cornalba \cite{AC87} if $\rm{char}(k)=0$ and by Schr\"oer \cite{Sch03} for an arbitrary field $k$. The extension to arbitrary pairs $g,n\geq 0$ can be found in \cite{FV}.

Note also that the result for $\operatorname{RelPic}(\mathcal C_{g,n})_{\bbQ}$ follows (under the assumption $2g-2+n>0$, i.e., when $\Mg$ is a DM stack)  from the computation of $\Pic(\ov \cM_{g,n})_{\bbQ}$ performed by Arbarello-Cornalba \cite{AC98} in characteristic zero and by Moriwaki \cite{Mor01} in positive characteristic.

\begin{rmk}\label{R:M10}
Note that the group $\operatorname{RelPic}(\mathcal C_{1,0})$ is trivial, which gives another proof of the fact that  $\pi:\cC_{1,0}\to \cM_{1,0}$ is not projective, see Remark \ref{R:proj-fam}.
\end{rmk} 

The above Theorem allows us to compute also the group of relative degree-$0$ line bundles on the universal family $\pi:\Cg\to \Mg$:
$$\operatorname{RelPic^0}(\mathcal C_{g,n}):=\{L\in \operatorname{RelPic}(\mathcal C_{g,n})| L \text{ has $\pi$-relative degree }0\}.
$$

\begin{cor}\label{C:Franch0}
The group $\operatorname{RelPic}^0(\mathcal C_{g,n})$ is
\begin{enumerate}[(i)]
\item freely generated by $\omega_{\pi}((2-2g)\sigma_1)$ and $\oo(\sigma_i-\sigma_{i+1})$ for $i=1,\ldots, n-1$, if $n\geq 1$ and $g\geq 2$;
\item freely generated by $\oo(\sigma_i-\sigma_{i+1})$ for $i=1,\ldots, n-1$, if $n\geq 1$ and $g=1$;
\item trivial if either $n=0$ or $g=0$. 
 \end{enumerate}
\end{cor}

\section{The universal moduli stack $\bg{G}$}\label{S:BunG}

In this section $G$ will be a  \emph{connected (smooth) linear algebraic group}  over $k=\ov k$, i.e., a connected and smooth affine group scheme of finite type over $k$. Further restrictions on $G$, like  reductiveness, will be specified when needed.

We denote by $\bg{G}$ the \emph{universal moduli stack of $G$-bundles over $n$-pointed  curves of genus $g$}. More precisely, for any scheme $S$, $\bg{G}(S)$ is the groupoid of triples $(C\to S,\un{\sigma}, E)$, where  $(\pi:\cC\to S,\un\sigma=\{\sigma_1,\ldots, \sigma_n\})$ is a family of  $n$-pointed curves of genus $g$ over $S$ and $E$ is a $G$-bundle on $C$.
We will denote by  $(\pi:\mathcal C_{G,g,n}\to\bg{G},\un \sigma, \mathcal E)$ the universal family of $G$-bundles.

By definition, we have a forgetful surjective morphism 
\begin{equation}\label{E:PhiG}
\begin{aligned}
\Phi_G(=\Phi_{G,g,n}):\bg{G}& \longrightarrow \Mg\\
(C\to S,\un{\sigma}, E)& \mapsto (C\to S,\un{\sigma})
\end{aligned}
\end{equation}
 onto the moduli stack $\Mg$ of $n$-pointed  curves of genus $g$. Note that the universal $n$-pointed curve $(\mathcal C_{G,g,n}\to\bg{G},\un \sigma)$ over $\bg{G}$ is the pull-back of the universal $n$-pointed curve $(\Cg\to\Mg,\un \sigma)$ over $\Mg$. 
 
 When $n=0$, we will remove the reference to the points from the notation (e.g. $\mathrm{Bun}_{G,g}$ instead of $\mathrm{Bun}_{G,g,0}$, $\mt M_g$ instead of $\Mg$, $(C\to S,E)$ instead of $(C\to S,\un{\sigma},E)$, etc).

For any family of curves $C\to S$, we denote by $\mathrm{Bun}_{G}(C/S)$ the  \emph{moduli stack of $G$-bundles on $C\to S$}. 
More precisely, for any $S$-scheme $T$, $\mathrm{Bun}_{G}(C/S)(T)$ is the groupoid of $G$-bundles on $C_T:=C\times_S T$. By definition, we have  a forgetful surjective morphism 
\begin{equation}\label{E:PhiG-rel}
\Phi_{G}(C/S):\mathrm{Bun}_G(C/S) \longrightarrow S
\end{equation}

The relation between the universal stacks $\bg{G}$ and the relative stacks $\mathrm{Bun}_{G}(C/S)$ goes as follows.
First of all, we have that 
\begin{equation}\label{E:reluniv1}
\mathrm{Bun}_{G}(\Cg/\Mg)=\bg{G}. 
\end{equation}
On the other hand, if the family $C\to S$ has constant relative genus $g=g(C/S)$ then we have that 
\begin{equation}\label{E:reluniv2}
\mathrm{Bun}_{G}(C/S)=S\times_{\mathcal M_{g}}\mathrm{Bun}_{G,g},
\end{equation}
with respect to the modular morphism $S\to \cM_g$ associated to the family $C\to S$.

The geometric properties of the universal stack $\bg{G}$ and of the relative stack $\mathrm{Bun}_{G}(C/S)$ are collected in the following

\begin{teo}[Behrend \cite{BK}, Wang \cite{W}] \label{beh}
Let $\pi:C\to S$ be a family of curves. 
\begin{enumerate}[(i)]
\item \label{beh1} $\mathrm{Bun}_G(C/S)$ is an algebraic stack  locally of finite presentation and smooth over $S$. 
\item \label{beh2} The relative diagonal of $\mathrm{Bun}_G(C/S)\to S$ is affine and finitely presented. 
\end{enumerate}
In particular, this is true for $\bg{G}$ over $\Mg$. 
\end{teo} 
\begin{proof}
Part \eqref{beh1}:  $\mathrm{Bun}_G(C/S)$ is an algebraic  stack locally of finite presentation over $S$ by \cite[Prop. 4.4.5]{BK} and smooth over $S$ by \cite[Prop. 4.5.1]{BK} (see also \cite[Prop. 6.0.18]{W}). 

Part \eqref{beh2}: since the properties of being affine and finitely presented are both \'etale local on the target, we can assume, up to replacing $S$ with an \'etale cover, that the family $\pi$ has a section. This implies that the family $\pi$ is projective (since the image of a section defines a relatively ample Cartier divisor), and hence 
the relative diagonal of $\mathrm{Bun}_G(C/S)\to S$ is affine and finitely presented by \cite[Cor. 3.2.2]{W}.

The corresponding statement for $\bg{G}$ follows from the relative case applied to the universal family $\pi:\Cg\to \Mg$.
\end{proof}

Any morphism of  connected  linear algebraic groups $\phi:G\to H$ determines a morphism of stacks over $\Mg$
\begin{equation}\label{E:fun1}
\begin{array}{lccc}
\phi_\#(=\phi_{\#,g,n}):&\bg{G}&\longrightarrow&\bg{H}\\
&\Big(C\to S,\un \sigma, E\Big)&\longmapsto &\Big(C\to S,\un \sigma, (E\times H)/G\Big)
\end{array}
\end{equation}
where the (right) action of $G$ on $E\times H$ is $(p,h).g:=(p.g,\phi(g)^{-1}h)$. And, similarly, given a family of curves $C\to S$, we can define the morphism of stacks 
\begin{equation}\label{E:fun2}
\begin{array}{lccc}
\phi_\#(=\phi_\#(C/S)):&\mathrm{Bun}_{G}(C/S)&\longrightarrow&\mathrm{Bun}_{H}(C/S)\\
&E&\longmapsto &(E\times H)/G.
\end{array}
\end{equation}

\begin{rmk}\label{R:locfin}
Since $\Phi_{G,g,n}$ (resp., $\Phi_G(C/S)$) is locally of finite type over $\Mg$ (resp., $S$) by Theorem \ref{beh}\eqref{beh1}, we deduce from \cite[\href{https://stacks.math.columbia.edu/tag/06U9}{Tag 06U9}]{stacks-project} that the morphisms \eqref{E:fun1} (resp., \eqref{E:fun2}) are locally of finite type.  
\end{rmk}

\begin{lem}\label{prod}Any commutative (resp., cartesian) diagram of connected  linear algebraic groups
$$
\xymatrix{
G_1\ar[d]^{\varphi_G}\ar[r]^{\phi_1}& H_1\ar[d]^{\varphi_H}\\
G_2\ar[r]^{\phi_2}&H_2,
}
$$
induces a commutative (resp., cartesian) diagram of moduli stacks
$$
\xymatrix{
\bg{G_1}\ar[d]^{(\varphi_G)_\#}\ar[r]^{(\phi_1)_\#}& \bg{H_1}\ar[d]^{(\varphi_H)_\#}\\
\bg{G_2}\ar[r]^{(\phi_2)_\#}&\bg{H_2}.
}
$$
The analogue  statement is true for the stack of principal bundles over a fixed family of curves $C\to S$. 
\end{lem}

\begin{proof}The proof is essentially the same as \cite[Lemma 2.2.1]{BH10}.
\end{proof}

\begin{rmk}\label{cart} Observe that if $G$ is trivial, then $\bg{G}=\mathcal M_{g,n}$. In particular we have the following cartesian diagrams
$$
\xymatrix@R-1pc{
G\times H\ar[dd]\ar[r]& H\ar[dd]& &\bg{G\times H}\ar[dd]\ar[r]& \bg{H}\ar[dd]\\
&&\Longrightarrow &&\\
G\ar[r]&1&&\bg{G}\ar[r]&\mathcal M_{g,n}.
}
$$
Similarly, for any family of curves $C\to S$, we have an isomorphism 
$$\mathrm{Bun}_{G\times H}(C/S)\cong \mathrm{Bun}_G(C/S)\times_S \mathrm{Bun}_H(C/S).$$
\end{rmk}

\subsection{The connected components of $\bg{G}$}\label{S:compBunG}

In this subsection we will recall the description of the connected components of $\bg{G}$, we determine their relative dimension over $\Mg$ and we study when these connected components are of finite type over $\Mg$ (and hence also over $k$). And similarly for the connected components of $\mathrm{Bun}_G(C/S)\to S$.

Recall that to any connected  (smooth) linear algebraic group $G$ over $k=\ov k$  it is possible to associate, in a functorial way, a finitely generated abelian group, denoted by $\pi_1(G)$ and called the \emph{fundamental group}\footnote{The name is justified by the fact that  if $k=\bbC$ then $\pi_1(G)$ coincides with the topological fundamental group of the complex Lie group $G(\bbC)$.} of $G$. For more details on the topic, we refer the reader to \cite[Sec. 10]{Mer98}, see also \cite{CoGPfr}, \cite{CoGPen} and \cite[Sec. 1.8]{BKG}).
Consider the exact sequence of connected linear algebraic groups 
\begin{equation}\label{E:red-quot}
1\to G_u\to G \xrightarrow{\red} G^{\red} \to 1,
\end{equation} 
where $G_u$ is the unipotent radical of $G$ (i.e., the largest connected normal subgroup of $G$ that is unipotent) and $G^{\red}:=G/G_u$ is the reductive canonical quotient of $G$.
From the definition of fundamental group (see \cite[Sec. 10]{Mer98}), it follows immediately that the morphism $\red$ induces an isomorphism\footnote{If $k=\bbC$, this isomorphism follows from the well-known fact that the complex Lie group $G_u(\bbC)$ is simply connected.}
 \begin{equation}\label{E:pired}
 \pi_1(\red):\pi_1(G)\xrightarrow{\cong} \pi_1(G^{\red}),
 \end{equation}
where  $\pi_1(G^{\red})$ can be computed from the root data of $G^{\red}$ as explained  in \S \ref{red-grps}.

\begin{teo}[Hoffmann \cite{Ho10a}]\label{concomp}
The connected components  of $ \bg{G}$ (and of $\mathrm{Bun}_G(C/S)$ for any family of curves $C\to S$ with $S$ connected) are in functorial bijection with the fundamental group $\pi_1(G)$ of $G$. 
\end{teo}
\begin{proof} 
Hoffmann proves in \cite[Theorem 5.8]{Ho10a} that, for a curve $C$ over $k=\ov k$,  the connected components of $\mathrm{Bun}_G(C/k)$ are in functorial bijection with $\pi_1(G)$. This implies our result using that $\Phi_G:\bg{G}\to \Mg$ and $\Phi_G(C/S):\mathrm{Bun}_G(C/S)\to S$ are smooth (and surjective) by Theorem \ref{beh}, and that $\Mg$ is connected. 
\end{proof}

For any $\delta\in\pi_1(G)$, we denote with 
\begin{equation}\label{E:PhiGcomp}
\Phi_{G,g,n}^{\delta}=\Phi_G^{\delta}:\bg{G}^{\delta}\to \Mg  \quad \Big(\text{ resp., } \Phi_{G}^{\delta}(C/S): \mathrm{Bun}_G^{\delta}(C/S) \to S \Big)
\end{equation} 
the corresponding connected component of $\bg{G}$ (resp., of  $\mathrm{Bun}_G(C/S)$ for a family of curves $C\to S$ with $S$ connected). 

The functoriality in the above Theorem \ref{concomp} means that for any morphism $\phi:G\to H$ of connected  linear algebraic groups over $k$, the induced morphisms   \eqref{E:fun1} and \eqref{E:fun2} respect the decomposition into connected components, i.e., 
for any $\delta\in \pi_1(G)$ and for any family of curves $C\to S$ with $S$ connected, we have that 
$$\phi_\#(\bg{G}^{\delta})\subseteq \bg{H}^{\pi_1(\phi)(\delta)} \quad \text{ and } \quad \phi_\#(\mathrm{Bun}_G(C/S))\subseteq \mathrm{Bun}_H^{\pi_1(\phi)(\delta)}(C/S),
$$
where $\pi_1(\phi):\pi_1(G)\to \pi_1(H)$ is the map induced by the morphism $\phi$.

\begin{cor}\label{C:regint}
For every $g,n\geq 0$ and $\delta\in \pi_1(G)$, the algebraic stack $\bg{G}^{\delta}$ is smooth over $k$ (hence regular) and integral.
\end{cor}
\begin{proof}
The algebraic stack $\bg{G}^{\delta}$ is smooth over $k$ because it is smooth over $\Mg$ by Theorem \ref{beh}\eqref{beh1}  and $\Mg$ is smooth over $k$. Moreover, being also connected by Theorem \ref{concomp},  $\bg{G}^{\delta}$ is integral.
\end{proof}

\vspace{0.1cm}

We now determine the relative dimension of each connected component $\bg{G}^{\delta}$ over $\Mg$. To this aim, consider the morphism 
\begin{equation}\label{E:pi1ad}
\pi_1(\det \circ \ad): \pi_1(G)\xrightarrow{\pi_1(\ad)} \pi_1(\GL(\g))\xrightarrow[\cong]{\pi_1(\det)} \pi_1(\Gm)=\bbZ,
\end{equation}
where $\ad=\ad_G: G\to \GL(\g)$ is the adjoint representation of $G$ and $\det: \GL(\g)\to \Gm$ is the determinant morphism. 

\begin{teo}\label{T:dim-comp}
The relative dimension of   $ \bg{G}^{\delta}\to \Mg$ (and of $\mathrm{Bun}^{\delta}_G(C/S)\to S$ for any family of curves $C\to S$ with $S$ connected) is equal to 
$$(g-1)\dim G-\pi_1(\det \circ \ad)(\delta).
$$
If $G$ is reductive, then $ \bg{G}$ is equidimensional over $\Mg$ (resp., $\mathrm{Bun}_G(C/S)$ is equidimensional over $S$) of relative dimension equal to 
$$(g-1)\dim G.$$
\end{teo}
\begin{proof}
Clearly, it is enough to prove the statement for $\mathrm{Bun}^{\delta}_G(C/k)$ where $C$ is a curve over $k=\ov k$. Fix a $G$-bundle $E$ on $C$. We denote by $\mathrm{ad}(E):=(E\times\g)/G$ the adjoint bundle of $E$, i.e., the vector bundle on $C$ induced by $E$ via the adjoint representation $\ad:G\to \GL(\g)$.  

It is well-known that the first order infinitesimal deformations of a $E\to C$ are parametrized by $H^1(C,\ad(E))$ while the infinitesimal automorphisms of $E\to C$ are parametrized by $H^0(C,\ad(E))$. Hence,  the dimension of $\mathrm{Bun}^{\delta}_G(C/k)$ at a point $E\to C$ is equal to 
$$\dim H^1(C,\ad(E))-\dim H^0(C,\ad(E))=-\chi(\ad(E))=-\pi_1(\det \circ \ad)(\delta)-(1-g)\dim G,
$$
where in the last equality we have applied Riemann-Roch theorem to the vector bundle $\ad(E)$ over $C$ which has rank $\dim G$ and degree equal to $\pi_1(\det \circ \ad)(\delta)$. 
For another proof which does not use deformation theory, see \cite[Sec. 8.1]{BK}.

The last statement  follows from the well-know fact that if $G$ is reductive then $\Im(\ad_G)\subseteq \SL(\g)$. For another proof, see \cite[Cor. 8.1.9]{BK}.
\end{proof}

\vspace{0.1cm} 

We now determine which connected components $\bg{G}^{\delta}$ are of finite type over $\Mg$ (and hence also over $k$) and, similarly,  which connected components  $\mathrm{Bun}_G(C/S)$ are of finite type over $S$. The answer turns out to depend solely on the group $G$ and not on the pair $(g,n)$ nor on the family $C\to S$ nor on the given connected component. 

First of all, we show that unipotent groups give rise to finite type stacks of bundles. To achieve this, we need the following

\begin{lem}\label{L:lin-filtr}
Let $U$ be a smooth connected normal unipotent subgroup in $G$. Then, it admits a linearly filtered filtration, i.e., a filtration
\begin{equation*}
\{1\}\subset U_r\subset\ldots\subset U_1\subset U_0=U
\end{equation*}
of normal smooth connected unipotent subgroups of $G$ such that the quotient $V_i:=U_i/U_{i+1}$ is a vector group, i.e., it is isomorphic to $\mathbb G_a^m$ for some $m\geq 1$, and the action by conjugation of $G$ restricted to $V_i$ is linear, i.e., factors through the natural action of $\GL_m$ on $\mathbb G_a^m$.
\end{lem}

\begin{proof}Since $k$ is algebraically closed,  $U$ is a smooth connected split unipotent group, see \cite[\S 15]{Bo}. By \cite[Thm B]{McN}, there exists a filtration
$$
\{1\}\subset U_r\subset\ldots\subset U_1\subset U_0=U
$$
by $G$-invariant normal subgroups of $U$ such that the quotient $V_i:=U_i/U_{i+1}$ is a vector group and the action by conjugation of $G$ restricted to $V_i$ is linear. Observe that the $G$-invariance is equivalent to say that $U_i$ is normal in $G$. The properties of being connected, smooth and unipotent are easy to check.
\end{proof}

We are now ready to prove 
\begin{prop}\label{P:red-ft} 
Consider an exact sequence of smooth connected linear algebraic groups
$$1\to U\to G\xrightarrow{\varphi} H\to 1,$$ 
with $U$ unipotent. Then the morphism  $\varphi_\#:\bg{G}\to\bg{H}$ (resp., the  $\varphi_\#:\mathrm{Bun}_{G}(C/S)\to\mathrm{Bun}_H(C/S)$ for any family of curves $C\to S$) is smooth, surjective and of finite type. 
\end{prop}

\begin{proof}
We present the proof just for the universal moduli stacks; the proof for the relative case $\mathrm{Bun}_G(C/S)$ follows immediately by pulling back $\varphi_\#$ along the morphism $S\to \cM_g$ associated to the family of curves $C\to S$ (up to restricting to the connected components of $S$). 
By Lemma \ref{L:lin-filtr} below, the group $U$ admits a linearly filtered filtration
\begin{equation}\label{E:uni-filtr}
\{1\}\subset U_r\subset\ldots\subset U_1\subset U_0=U
\end{equation}
We proceed by induction on the length of the filtration.

\fbox{$\text{Length}(U_\bullet)=0$} By assumption, $U\cong \mathbb G_a^m$ and the action of $G$ by conjugation on $U$ is linear. 

Let us first show the surjectivity of $\varphi_\#$. Let $(\pi:C\to T,\un{\sigma},F)\in \bg{H}(T)$. Since $U$ is abelian, there is a conjugation action of $H$ on $U$ and we may form the quotient
$$
U_H^F:=(F\times U)/H\to C,
$$
with respect to the diagonal action of $H$ on $F$ and $U$. By hypothesis $G$, and so $H$, acts linearly on $U$. So, $U_H^F$ is a vector bundle on $C$, hence $R^2(\pi_{\text{fppf}})_*(U_H^F)=0$. We may now apply \cite[Prop. 4.2.5]{BK} in order to infer that  $\varphi_\#$ is surjective. 

Consider now an object $(\pi:C\to T,\underline{\sigma},E)\in \bg{G}(T)$. By \cite[Proposition 4.2.4]{BK}, the fiber of $\varphi_\#$ over $\varphi_\#((\pi:C\to T,\underline{\sigma},E))\in \bg{H}(T)$ is isomorphic to the moduli stack $\mathrm{Bun}_{U_G^E}(C/T)$ of torsors under the (non-constant) underlying additive group scheme of the vector bundle
$$
U_G^E:=(E\times U)/G\to C
$$
where $G$ acts on $U$ by conjugation. By \cite[Cor. 8.1.3]{BK}, the stack $\mathrm{Bun}_{U_G^E}(C/T)$ is smooth and of finite type over $T$, which then implies the same property for $\varphi_\#$.

\fbox{$r:=\text{Length}(U_\bullet)>1$} By hypothesis, $U$ sits in the middle of an exact sequence $$1\to U_1\to U\to V\to 1,$$
where $U_1$ is a smooth connected unipotent normal subgroup of $G$ and $V$ is a vector group on which  $G$ acts linearly. In particular, we have that $\varphi_\#$ factors through
$$
\bg{G}\to \bg{G/U_1}\to \bg{H}.
$$
The second map is smooth, surjective and of finite type by the previous case. Observe that $U_1$ admits a linearly filtered filtration of length $r-1$ (restrict the filtration \eqref{E:uni-filtr} to $U_1$). By inductive hypothesis, the map $\bg{G}\to \bg{G/U_1}$ is smooth, surjective and of finite type and so is the map $\varphi_\#$.
\end{proof}

\begin{cor}\label{C:red-ft}
Let $G$ be a smooth connected linear algebraic group and let $\red: G\to G^{\red}$ be its reductive quotient. 
Then  $\red_\#:\bg{G}\to\bg{G^{\red}}$ (resp.,  $\red_\#:\mathrm{Bun}_{G}(C/S)\to\mathrm{Bun}_{G^{\red}}(C/S)$ for any family of curves $C\to S$) is smooth, surjective and of finite type.
\end{cor}
\begin{proof}
Apply Proposition \ref{P:red-ft} to the exact sequence \eqref{E:red-quot}. 
\end{proof}

We are now ready to show

\begin{prop}\label{P:qc=a}
For  a connected smooth  linear algebraic group $G$ over $k$, the  following conditions are equivalent:
\begin{enumerate}[(i)]
\item\label{qc=a-1} the reductive group $G^{\red}$ is a torus;
\item\label{qc=a-2a}  $\Phi_G^{\delta}: \bg{G}^\delta\to \Mg $ is of finite type  for any pair $g,n\geq 0$ and  for  any  $\delta\in\pi_1(G)$;
\item \label{qc=a-2b}$\Phi_G^{\delta}: \bg{G}^\delta\to \Mg $ is quasi-compact for some  pair $g,n\geq 0$ and  for some  $\delta\in\pi_1(G)$;
\item \label{qc=a-2c} $\bg{G}^\delta$ is  quasi-compact over  $k$, for some  pair $g,n\geq 0$ and  for some  $\delta\in\pi_1(G)$;
\item \label{qc=a-3a} $\Phi_G^{\delta}(C/S): \mathrm{Bun}^{\delta}_G(C/S)\to S$ is of finite type  for any family of curves $C\to S$ with $S$ connected and for any  $\delta\in\pi_1(G)$;
\item \label{qc=a-3b} $\Phi_G^{\delta}(C/S): \mathrm{Bun}^{\delta}_G(C/S)\to S$ is quasi-compact over $S$, for some family of curves $C\to S$ with $S$ connected and for some  $\delta\in\pi_1(G)$.
\end{enumerate}
\end{prop}
The above Proposition could be well-known to the experts, but we are not aware of any reference, so that we include a complete proof. 
\begin{proof} 
Let us split the proof in several steps.

\fbox{\eqref{qc=a-1} $\Rightarrow$ \eqref{qc=a-2a} and  \eqref{qc=a-3a}} By Proposition \ref{P:red-ft}, the  morphisms $\bg{G}^\delta\to \bg{G^{\red}}^{\delta}$ and $\mathrm{Bun}^{\delta}_G(C/S)\to \mathrm{Bun}^{\delta}_{G^{\red}}(C/S)$ are of finite type. Hence, it is enough to show that if $T$ is a torus then 
\begin{itemize}
\item $\Phi_T^{\delta}:\bg{T}^\delta\to \Mg$ is of finite type for any pair $g,n\geq 0$ and for any $\delta\in \pi_1(T)$;
\item $\Phi_T^{\delta}(C/S): \mathrm{Bun}^{\delta}_T(C/S)\to S$ is of finite type for  any family $C\to S$ (with $S$ connected) and for any $\delta\in \pi_1(T)$.
\end{itemize}
This follows from the fact that, fixing an isomorphism $T\cong \mathbb G_m^r$, the stack $\bg{T}^\delta\to \Mg$ (resp., $\mathrm{Bun}^{\delta}_T(C/S)\to S$) is isomorphic to the fibered product of $r$ connected components of the Jacobian stack $\bg{\Gm}\to \Mg$ (resp., $\mathrm{Bun}_{\Gm}(C/S)\to S$) which is of finite type. 

\fbox{\eqref{qc=a-2a} $\Rightarrow$ \eqref{qc=a-2b} and  \eqref{qc=a-3a} $\Rightarrow$ \eqref{qc=a-3b}}:  obvious.

\fbox{\eqref{qc=a-2b} $\Leftrightarrow$ \eqref{qc=a-2c}} Since $\Mg$ is quasi-compact and separated over $k$,  then using \cite[\href{https://stacks.math.columbia.edu/tag/050Y}{Tag 050Y}]{stacks-project} and \cite[\href{https://stacks.math.columbia.edu/tag/050W}{Tag 050W}]{stacks-project} we deduce that $ \bg{G}^{\delta}\to \Mg$  is  quasi-compact if and only if  $\bg{G}^{\delta}$  is quasi-compact  over $k$.

\fbox{\eqref{qc=a-2b} or \eqref{qc=a-3b}$\Rightarrow$\eqref{qc=a-1}} Both conditions \eqref{qc=a-2b} and \eqref{qc=a-3b} imply that there exists a curve $C$ over $k$ and an element $\delta\in \pi_1(G)$ such that  $\mathrm{Bun}^{\delta}_G(C/k)$ is quasi-compact. Since the morphism $\mathrm{Bun}_G(C/k)\to \mathrm{Bun}_{G^{\red}}(C/k)$ is surjective  by Corollary \ref{C:red-ft}, we deduce that $\mathrm{Bun}^{\delta}_{G^{\red}}(C/k)$ is quasi-compact by \cite[\href{https://stacks.math.columbia.edu/tag/050X}{Tag 050X}]{stacks-project}. 
Hence the proof will follow from the following

\un{Claim:} If $G$ is a reductive group such that $\mathrm{Bun}^{\delta}_G(C/k)$ is quasi-compact, for some curve $C$ over $k$ and some $\delta\in \pi_1(G)$, then $G$ is a torus. 

In order to prove the Claim, consider the upper semicontinuous function 
\begin{equation}
\begin{array}{cccl}
h:&\mathrm{Bun}^\delta_G(C/k)&\to& \mathbb{Z}\\
&E&\mapsto & \dim H^0(C,\mathrm{ad}(E)).
\end{array}
\end{equation}
Since $\mathrm{Bun}^\delta_G(C/k)$ is quasi-compact by assumption and $h$ is upper semicontinuous, then $h$ must be bounded. We now deduce from the boundedness of $h$ the fact that $G$ must be a torus. 

Fix a maximal torus and a Borel subgroup $T_G\subset B_G\subset G$. Let $d\in\pi_1(T_G)=\Lambda(T_G)$ be a lift of $\delta\in\pi_1(G)=\Lambda(T_G)/\Lambda(T_{G^{\sc}})$. Consider a $G$-bundle $E\to C$ in the image of the morphism $\mathrm{Bun}^d_{T_G}(C/k)\to\mathrm{Bun}^\delta_G(C/k)$. Then its adjoint bundle splits as direct sum of line bundles
$$
\mathrm{ad}(E)\cong\oo_C^{\dim T_G}\bigoplus_{\alpha \text{ root} }L_{\alpha},
$$
such that $\deg L_\alpha=(d, \alpha)$. By direct computation, for any integer $m$ there exists a lift $d_m$ of $\delta$ such that $(d_m,\alpha)\geq m$ for any positive root $\alpha$. In particular, for any $m\geq\mathrm{max}\{2g-2,0\}$ where $g$ is the genus of $C$, there exists a $G$-bundle $P_m\to C$ in $\mathrm{Bun}_G^\delta(C/k)$ such that
$$
\begin{array}{ll}
h(P_m)&=\dim T_G+\sum_{\alpha>0}\dim H^0(C,L_{\alpha})=\dim T_G+ \sum_{\alpha>0}((d_m,\alpha)+1-g)\geq\\
&\geq\dim T_G+ \#\{\alpha>0\}(m+1-g)=\dim T_G+ (\dim B_G-\dim T_G)(m+1-g).
\end{array}
$$
Since $h$ is bounded (as observed above), we must have that $\dim B_G=\dim T_G$, which then forces $T_G=B_G=G$, and the Claim is proved. 
\end{proof}

\subsection{Finite type open subsets of $\bg{G}$ and the instability exhaustion }\label{SS:open-ss}

In this subsection,  we study $k$-finite type open substacks of the moduli stack  $\bg{G}$ (and of $\mathrm{Bun}_G(C/S)$).

First of all, assuming that $G$ is reductive, we introduce the \emph{instability exhaustion } of $\bg{G}^{\delta}$, which provides a  cover of $\bg{G}^{\delta}$ by open substacks of finite type over $k$. 

\begin{defin}\label{D:in-de} 
Let $G$ be a reductive group over $k$ and let $E\to C$ be a $G$-bundle on a smooth curve over $k$. The bundle has \emph{instability degree less than or equal to $m$} if for any reduction $F$ to any parabolic subgroup $P\subseteq G$ (i.e., $F\to C$ is a principal $P$-bundle such that $E\cong (F\times G)/P$) we have 
	$$\text{deg(\text{ad}(F))}\leq m,$$ 
	where $\text{ad}(F):=(F\times\mathfrak p)/P$ is the adjoint bundle of $F$, i.e., the vector bundle on $C$ induced by $F$ via the adjoint representation $P\to GL(\mathfrak p)$.
	\end{defin}

\begin{rmk}\label{R:deg-pos}
Note that if $P=G$ in the above definition (so that $F=E$) then 
$$
\deg(\ad(E))=0,
$$
since the adjoint representation $\ad_G$ of a reductive group $G$ is such that $\Im(\ad_G)\subseteq \SL(\g)$. Hence, the instability degree of any $G$-bundle $E\to C$ is always non-negative.
\end{rmk}

For any $m\geq 0$, we denote by $ \bg{G}^{\delta,\leq m}\subset \bg{G}^{\delta}$ the locus of $G$-bundles (over $n$-pointed smooth curves) whose geometric fibers having instability degree less than or equal to $m$. The analogous locus in the relative situation $\mathrm{Bun}^{\delta}_G(C/S)$ will be denoted by $\mathrm{Bun}^{\delta,\leq m}_G(C/S)$. Note that the locus $\bg{G}^{\delta,\leq 0}$ is exactly the locus of semistable $G$-bundles. 
The properties of the above loci are collected in the following Proposition, which is based on the results of  \cite[Sec. 7]{BK}.

\begin{prop}\label{P:inst-cover}
Let $G$ be a reductive group over $k$. Then
\begin{enumerate}[(i)]
\item  \label{P:inst-cover1} the loci $\{\bg{G}^{\delta,\leq m}\}_{m\geq 0}$ form an exhaustive chain of open substacks of $\bg{G}^{\delta}$ (called the \emph{instability exhaustion } of $\bg{G}^{\delta}$);
\item \label{P:inst-cover2}  the stack $\bg{G}^{\delta,\leq m}$ is a smooth algebraic  stack of finite type over $\Mg$;
\item \label{P:inst-cover2b} if $G$ is a torus then  $\bg{G}^{\delta,\leq m}=\bg{G}^{\delta}$ for any $m\geq 0$; 
\item \label{P:inst-cover3} if  $G$ is not a torus then the complement of $\bg{G}^{\delta,\leq m}$ has codimension at least $g+m$. 
\end{enumerate}
The same holds true for the relative moduli stack $\mathrm{Bun}^{\delta}_G(C/S)$ for  any family of curves $C\to S$ with $S$ connected.
\end{prop}
\begin{proof}
We will   show the proposition in the relative case $\mathrm{Bun}^{\delta}_G(C/S)$; the universal case follows easily from the relative case. 

Part \eqref{P:inst-cover1}:  the openness of $\mathrm{Bun}^{\delta,\leq m}_G(C/S)\subseteq \mathrm{Bun}^{\delta}_G(C/S)$ follows from \cite[Theorem 7.2.4]{BK}.  By definition, it is clear that 
$\mathrm{Bun}^{\delta,\leq m}_G(C/S)\subseteq \mathrm{Bun}^{\delta,\leq m+1}_G(C/S)$, so that $\{\mathrm{Bun}^{\delta,\leq m}_G(C/S)\}_{m\geq 0}$ form a chain of open substacks of $\bg{G}^{\delta}$. The fact that the open substacks $\{\mathrm{Bun}^{\delta,\leq m}_G(C/S)\}_{m\geq 0}$ cover $\mathrm{Bun}^{\delta}_G(C/S)$ follows from \cite[Lemma 6.1.3]{BK}, which implies that for any $G$-bundle $E$ over a curve $C$ there exists $m\geq 0$ such that the  instability degree of $E$ is less than or equal to $m$.

Part \eqref{P:inst-cover2}: the fact that  $\mathrm{Bun}^{\delta,\leq m}_G(C/S)$ is a smooth algebraic  stack locally of finite type over $S$ follows from Theorem \ref{beh}. The fact that $\mathrm{Bun}^{\delta,\leq m}_G(C/S)$ is of finite type over $S$ can be proved with the same arguments of \cite[Theorem 8.2.6]{BK} (which treats the case $S=\Spec k$). Indeed, as in \emph{loc.cit.}, the moduli stack $\mathrm{Bun}^{\delta,\leq m}_G(C/S)$ has a cover given by a finite number of connected components of the moduli stack of $\mathrm{Bun}_{B_G}(C/S)$, with $B_G\subset G$ Borel subgroup, which are of finite type over $S$ by Proposition \ref{P:qc=a}. 

Part \eqref{P:inst-cover2b}: if $G=T$ is a torus, then its unique parabolic subgroup  is $T$ itself. Then Remark \ref{R:deg-pos} implies that each $T$-bundle on $C\to \Spec k$ has instability degree less than or equal to $0$, or, in other words, that  $\mathrm{Bun}_G(C/S)^{\delta,\leq 0}=\mathrm{Bun}_G(C/S)^{\delta}$. Then part \eqref{P:inst-cover1} implies that $\mathrm{Bun}_G(C/S)^{\delta,\leq m}=\mathrm{Bun}_G(C/S)^{\delta}$ for any $m\geq 0$.

Part \eqref{P:inst-cover3}: we follow the same strategy of \cite[Lemma 2.1]{BH12}. 
First of all, it is enough to prove the statement for $S=\Spec k$ with $k$ algebraically closed. Note that $\mathrm{Bun}_G^{\delta,\leq m}(C/k)\subsetneq \mathrm{Bun}_G^{\delta}(C/k)$, since 
 $\mathrm{Bun}_G^{\delta,\leq m}(C/k)$ is of finite type over $k$ by \eqref{P:inst-cover2} while  $\mathrm{Bun}_G^{\delta}(C/k)$ is not of finite type over $k$ by Proposition \ref{P:qc=a} because $G$ is not a torus. Pick an irreducible component $\mathcal V$ of $\mathrm{Bun}_G^{\delta}(C/k)\setminus\mathrm{Bun}_G^{\delta,\leq m}(C/k)$ and let $f:\Spec K\to \mathrm{Bun}_G^{\delta}(C/k)$ be a geometric point mapping onto the generic point of $\mathcal V$. The morphism $f$ is the classifying morphism of a $G$-bundle $E\to C_K:=C\times_k K$ whose degree of instability is greater of $m$.  By definition of degree of instability, the $G$-bundle $E\to C_K$ admits a reduction to a $P$-bundle $F\to C_K$ for some parabolic subgroup $P\subseteq G$ such that 
 \begin{equation}\label{E:deg-adF}
 \deg(\text{ad}(F))\geq m+1.
 \end{equation}
 Let $\mathrm{Bun}^{\epsilon}_P(C/k)$ be the connected component of $\mathrm{Bun}_P(C/k)$ containing $F\to C_K$. By construction, we have a morphism 
   $$\iota_\#:\mathrm{Bun}^{\epsilon}_P(C/k)\to\mathrm{Bun}^{\delta}_G(C/k),$$ 
which is dominant onto $\mathcal V$. Hence, using Theorem \ref{T:dim-comp} and \eqref{E:deg-adF}, we deduce that 
 $$
	\dim\mathcal V\leq \dim \mathrm{Bun}^{\epsilon}_P(C/k)=-\text{deg}(\text{ad}(\mathcal F))+(g-1)\text{dim}P\leq -(m+1)+(g-1)\text{dim}P.
	$$
Then, using that $\dim \mathrm{Bun}^{\delta}_G(C/k)=(g-1)\dim G$ again by Theorem \ref{T:dim-comp}, we conclude that 
	$$\text{codim}\mathcal V\geq (g-1)(\dim G-\dim P)+m+1\geq g+m,$$
where in the last inequality we used that $P\subsetneq G$ which follows from \eqref{E:deg-adF} and Remark \ref{R:deg-pos}.  
\end{proof}

A useful corollary of the above result is the following.

\begin{cor}\label{C:cod2-red} 
Assume that $G$ is a non-abelian reductive group and let $N\in \bbN$. Any morphism $f:\mt X\to \bg{G}^{\delta}$, with $\cX$ a quasi-compact algebraic stack (over $k$), factors through an open substack (of finite type over $k$) $ \bg{G}^{\delta,\leq m}\subset \bg{G}^{\delta}$ for some $m\gg 0$, such that the complementary substack has codimension at least $N$. The same holds true for the relative moduli stack $\mathrm{Bun}^{\delta}_G(C/S)$.
\end{cor}

\begin{proof}
We present the proof only for the universal case, the proof for the relative case uses the same argument. 

Since $\mathcal X$ is quasi-compact
and  $\{\bg{G}^{\delta,\leq m}\}_{m\geq 0}$ form an exhaustive chain of open  substacks  of $\bg{G}^{\delta}$ by Proposition \ref{P:inst-cover}, we get that $f(\cX)$ is contained in $ \bg{G}^{\delta,\leq m}$ (for some $m\geq 0$), which  is equivalent to say that $f$ factors through  $ \bg{G}^{\delta,\leq m}$.

Since the codimension of the complement $\bg{G}^{\delta}\setminus \bg{G}^{\delta, \leq m}$ goes to infinity as $m$ increases by Proposition  \ref{P:inst-cover}\eqref{P:inst-cover3},  we can obtain, up to increasing $m$, that the complementary substack of $\bg{G}^{\delta, \leq m}$ has codimension at least $N$. 
\end{proof}

Next, we prove that, for an arbitrary connected (smooth) linear algebraic group $G$,  every $k$-finite type open substack of $\bg{G}$ is a quotient stack over  $k$ (in the sense of Proposition \ref{Equi}), with a notable exception.

\begin{prop}\label{P:qstack-cover}
If $(g,n)\neq (1,0)$ then any open substack of $\bg{G}$ which is of finite type over $k$ is a quotient stack over $k$. 

	The same holds true for the relative moduli stack $\mathrm{Bun}_G(C/S)$, if the family of curves $C\to S$ has relative genus $g\neq 1$ or if it has a section.
\end{prop}
The above Proposition is false for $(g,n)=(1,0)$. For example, if $G=\{1\}$ is trivial, then $\mathrm{Bun}_{\{1\}, 1, 0}=\cM_{1,0}$ is of finite type over $k$ but it is not a quotient stack since its geometric points  have non-affine automorphism groups.  
\begin{proof}
Let us first prove the first statement.  Fix $\mt U\subset \bg{G}$ an open substack of finite type over $k$.  By \cite[Lemma 2.12]{EHKV}, it is enough to show that there exists a vector bundle $\mt V\to \mt U$ such that  the automorphism group of any geometric point $\ov x$ of $\mt U$ acts faithfully on the fiber $\mt V_{\ov x}$.
	
Assume first $G=GL_r$. Then we can identify $\bg{GL_r}$ with the moduli stack of objects $(\pi:C\to S, \un\sigma, E)$, where $(\pi:C\to S, \un\sigma)$ is a family of $n$-pointed curves of genus $g$ and   $E$ is a rank $r$ vector bundle  over $C$. Since $(g,n)\neq (1,0)$, there exists a relatively ample line bundle $L$ on the universal family $\Cg\to \Mg$	by Remark \ref{R:proj-fam}. Since $\mt U$ is of finite type, by standard arguments there exists an integer $k\gg 0$ such that for any object $(\pi:C\to S, \un\sigma, E)\in \mt U(S)$, if we denote by $L_S$ the pull-back of $L$ along the modular morphism $C\to \Cg$,  we have that 
\begin{equation}\label{E:cond-Lk}
\begin{aligned}
& L_C^k \: \text{ is relatively very ample on } \pi:C\to S, \\
& E(k):=E\otimes L_C^k \: \text{ is relatively globally generated on } \pi:C\to S,\\
&  \pi_*L_C^k \text{ and } \pi_*E(k) \text{ are locally free sheaves  on } S \: \text{ and commute with  base change.}
\end{aligned}
\end{equation}
Consider now the vector bundle $\mt V\to \mt U$ defined, for any $(\pi:C\to S, \un\sigma, E)\in \mt U(S)$, as
\begin{equation}\label{E:vec-bun}
\mt V((\pi:C\to S, \un\sigma, E)):=\pi_*L_C^k \oplus \pi_*E(k).
\end{equation}
We will now show that the automorphism group of any geometric point $\ov x$ of $\mt U$ acts faithfully on the fiber $\mt V_{\ov x}$, which will conclude the proof for $\GL_r$. 
Let $(C, \un\sigma, E)\in \mt U(K)$ be a geometric point of $\mt U$.  An element in $\text{Aut}(C,\un{\sigma}, E)$ is a pair $(\varphi, f)$, where $\varphi:C\xrightarrow{\cong} C$ is an automorphism of the curve preserving the marked points and $f:\varphi^*E\xrightarrow{\cong} E$ is an automorphism of vector bundles. Since the line bundle $L_C$ is the pull-back of a line bundle on the universal family $\Cg$, there exists a canonical isomorphism $L(\varphi):\varphi^*L_{C}\xrightarrow{\cong} L_C$. 
The action of $(\varphi, f)$ on a global section $(\eta,e)\in \mt V((C, \un\sigma, E))=H^0(C,L_C^k)\oplus H^0(C, E(k))$  is given by 
	\begin{equation}
	(\varphi, f)\circ (\eta, e)=\Big(L(\varphi)^k(\varphi^*\eta),(f\otimes L(\varphi)^k)(\varphi^*e)\Big).
	\end{equation}
	Assume that $(\varphi, f)$ acts trivially on any global section of  $\mt V((C, \un\sigma, E))$. Since $\varphi$ fixes all the sections of $H^0(C,L_C^k)$ and $L_C^k$ is very ample on $C$,  $\varphi$ must be the identity, i.e., $(\varphi, f)=(\id, f)$. Since $(\id, f)$ fixes all the sections of  $E(k)$ and $E(k)$ is global generated, then $f$ also is the identity.  Hence, we conclude that $\Aut(C,\un{\sigma}, E)$ acts faithfully on $\mt V((C, \un\sigma, E))$ and we are done.

	Let $G$ be an arbitrary (smooth and connected) linear algebraic group,  fix a faithful representation $\rho:G\to GL_r$ and consider the morphism 
	\begin{equation*}
	\rho_\#: \bg{G}\rightarrow \bg{\GL_r}.
	\end{equation*}
	Since  $\mt U\subset \bg{G}$ is of finite type over $k$ (hence quasi-compact), Corollary \ref{C:cod2-red} implies there exists a $k$-finite type open substack $\ov{\mt U}\subset \bg{\GL_r}$ such that $\rho_\#(\mt U)\subseteq \ov{\mt U}$.   We choose $\mt V\to\mt U$ as the pull-back, along the morphism $\rho_\#:\mt U \to \ov{\mt U}$,  of a vector bundle  $\ov{\mt V}\to\ov{\mt U}$ as in the previous case.  We now conclude since, given a geometric point $(C,\un \sigma, E)\in \mt U(K)$, the action of $\Aut((C,\un \sigma, E))$ on $\mt V((C,\un \sigma, E))=\ov{\mt V}((C,\un \sigma, \varphi_\#(E)))$ is faithful because we have an injective homomorphism 
	\begin{equation*}
	\Aut(C,\un \sigma ,E)\hookrightarrow \Aut(C,\un \sigma,\rho_\#(E))
	\end{equation*}
	and the group  $\Aut(C,\un \sigma,\rho_\#(E))$ acts faithfully on $\ov{\mt V}((C,\un \sigma, \varphi_\#(E)))$ by our choice of the vector bundle $\ov{\mt V}\to \ov{\mt U}$. 
	
	 The second statement about $\mathrm{Bun}_G(C/S)$ is proved with a similar argument starting from a relatively ample line bundle $L$ on $\pi:C\to S$, which exists by Remark \ref{R:proj-fam}. 
\end{proof}

\subsection{On the forgetful morphism $\Phi_G:\bg{G}\to \Mg$}\label{S:PhiG}

The aim of this subsection is to prove that, for a reductive group $G$,  the forgetful morphism $\Phi_G^{\delta}:\bg{G}^{\delta}\to \Mg$ is Stein (in the sense of Definition \ref{D:Stein}) for any $\delta\in \pi_1(G)$. 

Before doing this, we need to introduce an auxiliary stack, which is a slight generalization of the moduli stack of principal bundles. The first part of the subsection is a repetition of \cite[\S 4.2]{BH10} in the context of non-trivial families of curves. However, we briefly recall the main points and we refer the reader to \emph{loc. cit.} for more details.

Let $\widehat{G}$ be a reductive group sitting in the middle of an exact sequence
\begin{equation}\label{E:ex-tw}
1\to G\to \widehat{G}\xrightarrow{\dt} \mathbb G_m\to 1,
\end{equation}
with $G$ semisimple and simply connected. In particular, $\pi_1(\widehat G)=\pi_1(\mathbb G_m)=\bbZ$. 

Let $C\to S$ be a family of smooth curves of genus $g$ admitting a section $\sigma$. We will sometimes make the following assumption:
\begin{enumerate}
\item[$(\star)$]there exists an isomorphism between the formal completion of $C$ at $\sigma$ and $\mathrm{Spf}(\oo_S\llbracket t\rrbracket)$.
\end{enumerate}
We remark that such an isomorphism always exists Zariski-locally on $S$. 

Consider the  algebraic stack $\Phi_{\mathcal G_{d,\sigma}}(C/S): \mathrm{Bun}_{\mathcal G_{d,\sigma}}(C/S) \to S$ parametrizing \emph{twisted principal $G$-bundles} on $C\to S$, i.e., pairs $(E,\varphi)$ such that 
\begin{itemize}
\item $E\to C$ is a principal $\widehat{G}$-bundle,
\item $\varphi$ is an isomorphism $\dt_\#(E)\cong\oo(d\sigma)$ of $\mathbb G_m$-bundles.
\end{itemize}
In particular, there is a forgetful morphism of algebraic stacks over $S$
\begin{equation}
\begin{aligned}
\mathrm{Bun}_{\mathcal G_{d,\sigma}}(C/S)& \longrightarrow \mathrm{Bun}^d_{\widehat{G}}(C/S)\subset \mathrm{Bun}_{\widehat{G}}(C/S),\\
(E,\varphi) & \mapsto E.
\end{aligned}
\end{equation}

For any linear algebraic $k$-group $H$, we define the following functors over an affine $S$-scheme $T=\Spec(R)$:
\begin{enumerate}[(i)]
\item The loop group $\mathrm{L}H(T):=H(R\llparenthesis t\rrparenthesis)$,
\item The positive loop group $\mathrm{L}^+H(T):=H(R\llbracket t\rrbracket)$,
\item The affine Grassmannian $\mathrm{Gr}_H(T)$, which is the sheafification on the fpqc topology of the quotient $\mathrm{L}H(T)/\mathrm{L}^+H(T)$.
\end{enumerate}
If the family $C\to S$ satisfies $(\star)$,  the affine Grassmannian $\mathrm{Gr}_H$ can be identified with the functor of $H$-bundles on $C\to S$ with a trivialization on the complement $U:=C\setminus \Im(\sigma)$ of the section $\sigma$ (see \cite[\S 6]{Fa03}). In particular, there is a well-defined morphism of categories fibered in groupoids (over the category of schemes over $S$)
\begin{equation}
\glue_{H,\sigma}:\mathrm{Gr}_{H}\longrightarrow \mathrm{Bun}_{H}(C/S),
\end{equation}
which forgets the trivialization on $U$. 

Fix now a cocharacter $\delta:\mathbb{G}_m\to\widehat{G}$ such that $\dt\circ\delta=d\in\Hom(\mathbb G_m,\mathbb G_m)=\bbZ=\pi_1(\mathbb G_m)$. Denote by $t^{\delta}\in \mathrm{L}\widehat G(S)$ the image of $t\in \mathrm{L}\mathbb G_m(S)$ via the morphism $\delta_*: \mathrm{L}\mathbb G_m\to \mathrm{L}\widehat G$.  We then have a morphism of affine Grassmannians:
\begin{equation}
\begin{array}{cccc}
t^\delta_*:&\mathrm{Gr}_G&\longrightarrow &\mathrm{Gr}_{\widehat{G}}\\
&f\cdot \mathrm{L^+}G&\mapsto & t^\delta f\cdot \mathrm{L^+}\widehat G.
\end{array}
\end{equation}
The composition $\glue_{\widehat{G},\sigma}\circ t^\delta_*$ factors as
\begin{equation}\label{E:tw-glue}
\glue_{\widehat{G},\sigma}\circ t^\delta_*: \mathrm{Gr}_G\xrightarrow{\glue_{G,\sigma,d}} \mathrm{Bun}_{\mathcal G_{d,\sigma}}(C/S)\to \mathrm{Bun}^d_{\widehat{G}}(C/S)\subset \mathrm{Bun}_{\widehat{G}}(C/S)
\end{equation}
due to the fact that $\dt_*\circ t^\delta_*:\mathrm{Gr}_G\to\mathrm{Gr}_{\widehat{G}}\to\mathrm{Gr}_{\mathbb{G}_m}$ is constant with image equal to the line bundle $\oo(d\sigma)$, together with a trivialization on $U$. An important property of the morphism $\glue_{G,\sigma,d}$ is provided by the following result.

\begin{lem}\label{L:twisted-DS}
Let $\pi:C\to S$ be a family of curves with a section $\sigma$ and a $\widehat G$-bundle $E$. Then, any trivialization of the line bundle $\mathrm{dt}_\#(E)$ on $U:=C\setminus \Im(\sigma)$ can be lifted, after a suitable \'etale base change $S'\to S$, to a trivialization of $E\times_S S'$ on $U\times_S S'$. In particular, $\glue_{G,\sigma,d}$ admits a section \'etale-locally on $S$.
\end{lem}
\begin{proof}
It is a generalization of \cite[Theorem 3]{DS}. One can modify the argument in the proof of \emph{loc. cit.} arguing as in Lemma \cite[Lemma 4.2.2]{BH10}.
\end{proof}

We are now ready to show the following
\begin{prop}\label{P:SteinG}
If $G$ is reductive, then  $\Phi_G^{\delta}:\bg{G}^{\delta}\to\Mg$ is Stein for any $\delta\in\pi_1(G)$.
\end{prop}

\begin{proof}
Theorem \ref{beh} gives that $\Phi_G^{\delta}$ is faithfully flat and locally of finite presentation, which implies that $\Phi_G^{\delta}$ is fpqc. In order to conclude that $\Phi_G^{\delta}$ is Stein, it remains to show that the homomorphism $(\Phi_G^{\delta})^\#:\cO_{\Mg}\to (\Phi_G^{\delta})_*(\cO_{\bg{G}^{\delta}})$ is a universal isomorphism. For that purpose, it is also enough to show that, for any family of (smooth) curves  $C\to S$ of genus $g$ admitting a section $\sigma$ satisfying the condition $(\star)$ and with $S$ being a connected affine noetherian scheme, the morphism 
$\Phi_G^{\delta}(C/S):\mathrm{Bun}^{\delta}_G(C/S)\to S$ is such that  $\Phi_G^{\delta}(C/S)^\#:\cO_S\rightarrow \Phi_G^{\delta}(C/S)_*(\cO_{\mathrm{Bun}^{\delta}_G(C/S)})$ is an isomorphism (or equivalently that $\Phi_G^{\delta}(C/S)$ is a Stein morphism).

We divide the proof into the following steps:
	\begin{enumerate}[(i)]
		\item  $G=T$ a torus;
		\item twisted bundles, i.e., $\Phi_{\mathcal G_{d,\sigma}}(C/S):\mathrm{Bun}_{\mathcal G_{d,\sigma}}(C/S)\to S$;
		\item product of twisted bundles and torus bundles, i.e., $\Phi_{\mathcal G_{d,\sigma}}(C/S)\times \Phi_T^{\epsilon}(C/S):\mathrm{Bun}_{\mathcal G_{d,\sigma}}(C/S)\times_S \mathrm{Bun}_T^{\epsilon}(C/S)\to S$ with $T$ a torus;
		\item $G$ an arbitrary reductive group.
	\end{enumerate}
	
\noindent\fbox{Step (i)} The morphism $\Phi_T^{\delta}(C/S):\mathrm{Bun}^{\delta}_T(C/S)\to S$ is a Stein morphism since it is the composition of a $T$-gerbe $\mathrm{Bun}^{\delta}_T(C/S)\to \mathrm{Bun}^{\delta}_T(C/S)\fatslash T$ with an abelian algebraic space $\mathrm{Bun}^{\delta}_T(C/S)\fatslash T\to S$ (see \S \ref{Sec:gerbe-fib} for more details).

\noindent\fbox{Step (ii)} Consider the commutative diagram 
	$$
	\xymatrix{
	\mathrm{Gr}_G \ar[rr]^{\glue_{G,\sigma,d}} \ar[rd]_f& &\ar[dl]^{\Phi_{\mathcal G_{d,\sigma}}(C/S)}  \mathrm{Bun}_{\mathcal G_{d,\sigma}}(C/S),\\
	& S & 
	}
	$$
where $\glue_{G,\sigma,d}$ is the morphism \eqref{E:tw-glue}. Since $\glue_{G,\sigma,d}$ admits a section \'etale-locally on $S$ by Lemma \ref{L:twisted-DS}, we have that the natural homomorphism
	$$
	\Phi_{\mathcal G_{d,\sigma}}(C/S)_*(\glue_{G,\sigma,d}^\#): \Phi_{\mathcal G_{d,\sigma}}(C/S)_*\left(\oo_{\mathrm{Bun}_{\mathcal G_{d,\sigma}}(C/S)}\right)\to f_*\oo_{\text{Gr}_G}
	$$
	is injective. Hence, it is enough to show that  $f^\#:\oo_S\to f_*\oo_{\text{Gr}_G}$ is an isomorphism. 
	
The affine Grassmannian 	$f:\text{Gr}_G\to S$ is an ind-scheme, which is the direct limit of flat and projective morphisms  $f_i:Y_i\to S$ with  geometrically integral fibers (see \cite{Fa03}).
By standard arguments, each homomorphism $f^\#_i:\oo_S\to (\pi_i)_*\oo_{Y_i}$ is an isomorphism. Since, by definition of morphisms of ind-schemes, $f^\#$  is the inverse limit of the homomorphisms $f^\#_i$, we deduce that $f^\#$ is an isomorphism, q.e.d.
	
\noindent\fbox{Step (iii)} Let $\widehat{G}$ be a group sitting in the middle of \eqref{E:ex-tw}. The natural map $\mathrm{Bun}^{(d,\epsilon)}_{\widehat G\times T}(C/S)\to \mathrm{Bun}^{d}_{\widehat G}(C/S)$ is Stein by Remark \ref{cart} and Step (i). Since the diagram
\begin{equation}
\xymatrix{
\mathrm{Bun}_{\mathcal G_{d,\sigma}}(C/S)\times_{S}\mathrm{Bun}^{\epsilon}_{T}(C/S)\ar[r]\ar@{}[dr]|\square\ar[d]^{\mathrm{pr}_\#}&\mathrm{Bun}_{\widehat{G}\times T}^{\left(d,\epsilon\right)}(C/S)\ar[d]\\
\mathrm{Bun}_{\mathcal G_{d,\sigma}}(C/S) \ar[r]& \mathrm{Bun}^{d}_{\widehat{G}}(C/S)
}
\end{equation}
is cartesian, then also the morphism $\mathrm{pr}_\#$ must be Stein. We conclude using that $\Phi_{\mathcal G_{d,\sigma}}(C/S): \mathrm{Bun}_{\mathcal G_{d,\sigma}}(C/S)\to S$ is Stein by Step (ii). 
	
\noindent\fbox{Step (iv)}  By Lemma \cite[Lemma 5.3.2]{BH10}, for any $\delta\in\pi_1(G)$, there exists a reductive group $\widehat{G}$ sitting in the middle of an exact sequence as in \eqref{E:ex-tw},  equipped with a central isogeny $\widehat{G}\times \scr R(G)\to G\times\Gm$ inducing a faithfully flat map of moduli stacks
\begin{equation}\label{E:cover-iv}
\mathrm{Bun}_{\widehat{G}\times\mathscr  R(G)}^{\left(1,0\right)}(C/S)\to \mathrm{Bun}^{(\delta,1)}_{G\times\Gm}(C/S).
\end{equation}
Consider the following cartesian diagram
\begin{equation}\label{E:cart-cov}
\xymatrix{
\mathrm{Bun}_{\mathcal G_{d,\sigma}}(C/S)\times_{S}\mathrm{Bun}_{ \scr R(G)}^{0}(C/S)\ar[r]\ar@{}[dr]|\square\ar[d]&\mathrm{Bun}_{\widehat{G}\times\mathscr  R(G)}^{\left(1,0\right)}(C/S)\ar[d]\\
\mathrm{Bun}^{\delta}_{G}(C/S)\ar[r]& \mathrm{Bun}^{(\delta,1)}_{G\times\mathbb G_m}(C/S)
}
\end{equation}
Since, the map \eqref{E:cover-iv} is faithfully flat, also the left vertical map in \eqref{E:cart-cov} is faithfully flat; hence, in particular, its structural morphism is injective. This fact, together with the fact that the morphism $\mathrm{Bun}_{\mathcal G_{d,\sigma}}(C/S)\times_{S}\mathrm{Bun}_{ \scr R(G)}^{0}(C/S)\to S $ is Stein by Step (iii), implies that $\Phi_G^{\delta}(C/S):\mathrm{Bun}^{\delta}_{G}(C/S)\to S$ is also Stein.
\end{proof}

\subsection{On the reductions to a Borel subgroup}\label{S:redBor}

The aim of this subsection is to study the morphism from the moduli stack of $B_G$-bundles to the moduli stack of $G$-bundles, where $B_G$ is a Borel subgroup of a reductive group $G$. 
The main Theorem is the following one, which is based upon results of Drinfeld-Simpson \cite{DS} and Holla \cite{Hol} (which indeed relies on a result of Harder \cite{Ha74}).

\begin{teo}\label{T:Holla}
Let $G$ be a reductive group over $k$ and choose a maximal torus $T_G$ with associated Borel subgroup $j: B_G\hookrightarrow G$. 
\begin{enumerate}[(i)]
\item \label{T:Holla1} For any $d\in \pi_1(T_G)=\pi_1(B_G^{\red})=\pi_1(B_G)$, the morphism 
$$
j_\#:\bg{B_G}^d\longrightarrow\bg{G}^{[d]}
$$ 
is  of finite type. 
\item \label{T:Holla2} For any $\delta \in\pi_1(G)$ there exists a representative $d\in \pi_1(T_G)=\pi_1(B_G^{\red})=\pi_1(B_G)$, i.e., $[d]=\delta$, such that the morphism of moduli stacks
$$
j_\#:\bg{B_G}^d\longrightarrow\bg{G}^{\delta}
$$ 
is smooth with geometrically integral fibers of dimension $\sum_{\alpha<0} (\alpha, d)+(g-1)\dim( G/B_G)$, where $\{\alpha<0\}\subset \Lambda^*(T_G)$ is the set of negative roots with respect to the Borel subgroup $B_G$.
\end{enumerate}
\end{teo}
\begin{proof}
Part \eqref{T:Holla1}: first of all,  the morphism $j_\#$ is locally of finite type by Remark \ref{R:locfin}. Hence it is enough to show that $j_\#$ is quasi-compact. 

Theorem \ref{beh}\eqref{beh2} implies that the diagonal of $\Phi_G^{[d]}:\bg{G}^{[d]}\to \Mg$ is affine, and hence quasi-compact and quasi-separated, which by definition means that the morphism $\Phi_G^{[d]}$ is quasi-separated. Using this and the fact that $\Phi_{B_G}^d:\bg{B_G}^d \to \Mg$ is of finite type, and hence quasi-compact, by Proposition \ref{P:qc=a}, we deduce that $j_\#$ is quasi-compact by   \cite[\href{https://stacks.math.columbia.edu/tag/050Y}{Tag 050Y}]{stacks-project}.

Part \eqref{T:Holla2}: first of all, observe that the canonical diagram
\begin{equation}
\xymatrix{
\bg{B_G}^d\ar[r]^{j_\#}\ar@{}[dr]|\square\ar[d]& \bg{G}^{\delta}\ar[d]\\
\bg{B_G/\scr R(G)}^{d^{ss}}\ar[r]&\bg{G^{ss}}^{\delta^{ss}}
}
\end{equation}
is cartesian by Lemma \ref{prod}. So, without loss of generality, we can assume that $G$ is semisimple.

By \cite[\S 4]{DS}, the smoothness and the formula for the relative dimension of $j_\#$ hold if $(\alpha,d)\leq \text{min}\{1,2-2g\}$ for any positive root $\alpha$.

The fiber over an $S$-point $(C\to S, \un \sigma,E)$ is canonically identified with the moduli space $\text{Sec}^d(C,E/B_G)$ of sections of the flag bundle $E/B_G\to C$ of type $d$, i.e., the induced $B_G$-reduction is in the connected component $\bg{B_G}^d$.
Holla proved in \cite[Theorem 5.1 and Remark 5.14]{Hol} that there exists $N\in \bbZ$ such that, if $(\alpha,d)\leq N$ for any positive root $\alpha$, then the morphism $\text{Sec}^d(C,E/B_G)\to S$ has geometrically connected fibers. We remark that the statement in \emph{loc. cit.} is presented when $S$ is the spectrum of a field. However, in the proof the author shows that the integer $N$ exists for any family of curves over an integral affine scheme of finite type over $\mathbb Z$.

Hence, to conclude the proof, we need a representative $d$ such that $(\alpha,d)\leq \text{min}\{1,2-2g,N\}$ for any positive root $\alpha$. This can be shown either by direct computation or as direct consequence of \cite[Proposition 3]{DS}.
\end{proof}

\subsection{Tautological line bundles}\label{tau}

In this subsection, we will introduce certain natural line bundles on $\bg{G}$, that we call tautological line bundles.

With this in mind, let us recall two standard ways of producing line bundles on the base of a family of curves (see \cite[Chap. XIII, Sec. 4, 5]{GAC2} and the references therein). 
	Let $\pi:C\to S$ be a family of smooth curves over an algebraic  stack $S$ and let $\omega_\pi$ be the relative dualizing line bundle. To any coherent sheaf $\mathcal F$ on $C$ flat over $S$, we associate a line bundle $d_{\pi}(\mathcal F)$ over the base $S$, called the \emph{determinant of cohomology of $\mathcal F$ with respect to $\pi:C\to S$} and defined as follows: choose a complex of locally free sheaves of finite rank $f:K_0\to K_1$ such that $\ker(f)\cong \pi_*(\cal F)$ and $\coker(f)\cong R^1\pi_*(\cal F)$ (this is always possible), and define
\begin{equation}\label{D:detcoho}
d_{\pi}(\cal F):=\det \pi_*(K_0)\otimes \det \pi_*(K_1)^{-1}\in \Pic(\cal X). 
\end{equation}	
The determinant of cohomology is functorial with respect to base change, and it is multiplicative in short exact sequences and compatible with Serre Duality, i.e.:
	\begin{equation}\label{detprop}
	\begin{sis}
	& d_{\pi}(\mt G)\cong d_{\pi}(\mt F)\otimes d_{\pi}(\mt E) \quad \text{ for any short exact sequence } 0\to\mathcal F\to \mathcal G\to\mathcal E\to 0, \\
	& d_{\pi}(\mathcal F\otimes\omega_\pi)\cong d_{\pi}(\mt F^{\vee}) \quad \text{ if } \mt F \: \text{  is locally free.} 
	\end{sis}
	\end{equation}
	Moreover, if we have two line bundles $\mathcal F$ and $\mathcal G$ on $C$, we can produce a line bundle $\langle \mathcal F,\mathcal G\rangle_{\pi}=\langle \mathcal F,\mathcal G\rangle$ over $S$, called the \emph{Deligne pairing of $\mathcal F$ and $\mathcal G$ with respect to $\pi:\mt C\to\mt X$}, which is related to the determinant of cohomology via the following isomorphism:
	\begin{equation}\label{Del-pair}
		\def\arraystretch{1.5}\begin{array}{l}
	\langle \mathcal F,\mathcal G\rangle_{\pi}\cong d_{\pi}(\mathcal F\otimes\mathcal G)\otimes d_{\pi}(\mathcal F)^{-1}\otimes d_{\pi}(\mathcal G)^{-1}\otimes d_{\pi}(\mathcal O).
	\end{array}
	\end{equation}
	The Deligne pairing is symmetric, bilinear and compatible with sections, i.e.,
	\begin{equation}\label{Del-pair-comp}
	\begin{sis}
	& \langle \mathcal F,\mathcal G\rangle_{\pi}\cong \langle \mathcal G,\mathcal F\rangle_{\pi},\\
	& \langle \mathcal F_1\otimes\mathcal F_2,\mathcal G\rangle_{\pi}\cong\langle \mathcal F_1,\mathcal G\rangle_{\pi}\otimes\langle \mathcal F_2,\mathcal G\rangle_{\pi},\\
	& \langle\oo(\sigma),\mt F\rangle_{\pi}\cong\sigma^*(\mt F),\text{ where }\sigma\text{ is a section of }\pi.
	\end{sis}
    \end{equation}
    The first Chern class of the  Deligne pairing is given by 
    \begin{equation}\label{E:c1}
	 c_1(\langle \mathcal F,\mathcal G\rangle_{\pi})=\pi_*(c_1(\mathcal F)\cdot c_1(\mathcal G)).
    \end{equation}	
In our setting, we will apply the above two constructions to the universal family of $G$-bundles $(\pi:\mathcal C_{G,g,n}\to\bg{G},\un \sigma,\mathcal E)$.
To start with, we produce natural line bundles on  $\mathcal C_{G,g,n}$ via the following procedure. Any character $\chi:G\to\Gm$ gives rise to a morphism of stacks (see \eqref{E:fun1})
$$
\chi_\#:\bg{G}\to\bg{\Gm}.
$$
By pulling back via $\chi_{\#}$ the universal $\Gm$-bundle (i.e., line bundle) on the universal curve over $\bg{\Gm}$, we get a line bundle $\mathcal L_{\chi}$ over the universal curve $\pi:\mathcal C_{G,g,n}\to\bg{G}$. Then, using these natural line bundles on  $\mathcal C_{G,g,n}$ and the sections $\sigma_1,\ldots,\sigma_n$ of $\pi$, we define the following line bundles, that we call \emph{tautological line bundles}:
\begin{equation}\label{E:taut}
\begin{aligned}
& \mathscr L(\chi,\zeta):=d_{\pi}\big(\mt L_\chi(\zeta_1\cdot\sigma_1+\ldots+\zeta_n\cdot\sigma_n)\big),\\
& \langle (\chi,\zeta),(\chi',\zeta')\rangle:=\langle \mt L_\chi(\zeta_1\cdot\sigma_1+\ldots+\zeta_n\cdot\sigma_n), \mt L_{\chi'}(\zeta_1'\cdot\sigma_1+\ldots+\zeta_n'\cdot\sigma_n)\rangle_{\pi},\\
\end{aligned}
\end{equation}
for $\chi, \chi'\in \text{Hom}(G,\Gm)$ and $\zeta=(\zeta_1,\ldots,\zeta_n), \zeta'=(\zeta'_1,\ldots,\zeta'_n)\in \bbZ^n$. 
From (\ref{Del-pair}), we deduce that 
\begin{equation}\label{E:Del-det}
\langle (\chi,\zeta),(\chi',\zeta')\rangle:=\mathscr L(\chi+\chi',\zeta+\zeta')\otimes \mathscr L(\chi,\zeta)^{-1}\otimes \mathscr L(\chi',\zeta')^{-1}\otimes \mathscr L(0,0).
\end{equation}
When $n=0$, we will write $\mathscr L(\chi):=\mathscr L(\chi,0)$ and $\langle \chi,\chi'\rangle:=\langle (\chi,0),(\chi',0)\rangle$.
Note  that $\mathscr L(0,\xi)$ and $\langle (0,\xi),(0,\xi')\rangle$, for any $\xi,\xi'\in\bbZ^n$, are pull-back of line bundles on $\mt M_{g,n}=\bg{\{\id\}}$.

\begin{rmk}\label{R:taut-dual}
The reader may have noticed that we have not used  the relative dualizing sheaf $\omega_{\pi}$ (or powers of it) of the universal family $\mathcal C_{G,g,n}\to\bg{G}$ in the definition of the tautological line bundles. This is due to the following relations that hold for any family of curves $\pi:C\to S$, for any line bundle $\cal L$ on $C$ and for any $n\in \bbZ$:
\begin{equation}\label{E:tau-dual}
\begin{sis}
& \langle \cal L, \omega_{\pi}\rangle_{\pi}\cong d_{\pi}(\cal L)^{-1}\otimes d_{\pi}(\cal L^{-1})=\langle \cal L, \cal L\rangle_{\pi} \otimes d_{\pi}(\cal L)^{-2}\otimes d_{\pi}(\cO)^2, \\
& d_{\pi}(\omega_{\pi}^n\otimes  \cal L)\cong d_{\pi}(\cal L)^{1-2n}\otimes \langle \cal L, \cal L\rangle_{\pi}^n\otimes d_{\pi}(\cal O)^{6n^2-4n},
\end{sis}
\end{equation}
where the first formula is obtained by applying \eqref{Del-pair} twice, once to the pair $(\cal L, \omega_{\pi})$ and once to the pair $(\cal L, \cal L^{-1})$, and the second formula is obtained by applying  \eqref{Del-pair} to $\langle \cal L, \omega_{\pi}^n\rangle_{\pi}$ and  using Mumford's formula $d_{\pi}(\omega_{\pi}^n)\cong d_{\pi}(\cal O)^{6n^2-6n+1}$  (see \cite[Chap. XIII, Thm. 7.6]{GAC2}).
\end{rmk}

\section{Reductive abelian case}\label{Sec:tori}

In this section, we will  compute the relative Picard group of $\bg{T}^d$:
\begin{equation}\label{E:RelPicT}
\RPic(\bg{T}^d):=\Pic(\bg{T}^d)/(\Phi_T^d)^*(\Pic(\Mg)),
\end{equation}
where $T$ is a torus, $d\in \pi_1(T)=\Lambda(T)$ and $\Phi_T^d:\bg{T}^d\to \Mg$ is the natural forgetful morphism. 
The three cases $g\geq 2$, $g=1$ and $g=0$ behave differently. The first two cases are similar and they are collected in the following 

\begin{teo}\label{BunT} 
Assume $g\geq 1$. 
	\begin{enumerate}
	       \item \label{BunT1} The relative Picard group $\RPic(\bg{T}^d)$ is generated by the tautological line bundles \eqref{E:taut} and there are exact sequences of abelian groups
		\begin{equation}\label{seq-bunT}
		0\to\Sym^2\Lambda^*(T)\oplus\left(\Lambda^*(T)\otimes\bbZ^n\right)\xrightarrow{\tau_T+\sigma_T}\RPic\left(\bg{T}^d\right)\xrightarrow{\rho_T} \Lambda^*(T)\to 0 \quad \text{ if } g\geq 2,
		\end{equation}
		\begin{equation}\label{seq-bunTg1}
		0\to\Sym^2\Lambda^*(T)\oplus\left(\Lambda^*(T)\otimes\bbZ^n\right)\xrightarrow{\tau_T+\sigma_T}\RPic\left(\mathrm{Bun}_{T,1,n}^d\right)\xrightarrow{\rho_T} \frac{\Lambda^*(T)}{2\Lambda^*(T)} \to 0 \:\: \text{ if } g=1,
		\end{equation}
		where $\tau_T(=\tau_{T,g,n})$ (called \emph{transgression map}) and $\sigma_T(=\sigma_{T,g,n})$ are defined by 
		$$
		\begin{array}{cclll}
		\tau_T(\chi\cdot\chi')&=&\langle(\chi,0),(\chi',0)\rangle, &\text{ for any } \chi, \chi'\in\Lambda^*(T),\\
		\sigma_T(\chi\otimes \zeta)&=&\langle(\chi,0),(0,\zeta)\rangle, &\text{ for any }\chi\in\Lambda^*(T) \text{ and }\zeta\in\bbZ^n,\\
		\end{array}
		$$
		and $\rho_T(=\rho_{T,g,n})$ is the unique homomorphism such that
		$$ 
		\rho_T(\mathscr L(\chi,\zeta))=
		\begin{cases}
		\chi \in \Lambda^*(T) & \text{ if } g\geq 2,\\
		[\chi] \in \frac{\Lambda^*(T)}{2\Lambda^*(T)}& \text{ if } g=1,\\
		\end{cases} \quad \text{ for any } \: \chi\in\Lambda^*(T) \text{ and }\zeta\in\bbZ^n.
		$$
		
		Furthermore, the exact sequences \eqref{seq-bunT} and  \eqref{seq-bunTg1} are contravariant with respect to homomorphisms of tori.
		\item \label{BunT2} Fix an isomorphism $T\cong \bbG_m^r$ which induces an isomorphism $\Lambda^*(T)\cong\Lambda^*(\bbG_m^r)=\bbZ^r$. Denote by $\{e_i\}_{i=1}^r$ the canonical basis of 
		$\bbZ^r$ and by $\{f_j\}_{j=1}^n$ the canonical basis of $\bbZ^n$.
		The relative Picard group of $\bg{T}^d$ is freely generated by 
		$$
		\def\arraystretch{1.5}\begin{array}{ll}
		\left\langle (e_i,0),(0,f_j)\right\rangle=\sigma_j^*(\mt L_{e_i}), &\text{for } i=1,\ldots,r, \text{ and }j=1,\ldots,n,\\
		\left\langle (e_i,0),(e_k,0) \right\rangle=\langle \mt L_{e_i}, \mt L_{e_k}\rangle, &\text{for } 
		\begin{cases} 
		1\leq i\leq k\leq r  &\text{ if } g\geq 2,\\
		1\leq i< k\leq r  &\text{ if } g=1,\\
		\end{cases}\\
		\mathscr L(e_i,0)=d_{\pi}(\mt L_{e_i}),&\text{for } i=1,\ldots,r.
		\end{array}
		$$
	\end{enumerate}	
\end{teo}
Note that the homomorphisms $\tau_T$ and $\sigma_T$ are well-defined in the full Picard group of $\bg{T}^d$ (and not just in its relative Picard group).

The case $g=0$ is rather different from the other cases and we isolate it in the following

\begin{teo}\label{BunTg0} 
Assume $g=0$.
	\begin{enumerate}
	 
	       \item \label{BunTg0-1} The relative Picard group $\RPic(\mathrm{Bun}_{T,0,n}^d)$ is generated by the tautological line bundles \eqref{E:taut} 
	        and there is an injective homomorphism 
		\begin{equation}\label{seq-bunTg0}
		\RPic\left(\mathrm{Bun}_{T,0,n}^d\right)\xrightarrow{w_T^d=w_{T,0,n}^d} \Lambda^*(T),
		\end{equation}
		defined by 
		\begin{equation*}
\begin{sis}
&w_T^d(\scr L(\chi, \zeta))=[(d, \chi)+|\zeta|+1]\chi,  \\
&  w_T^d(\langle (\chi, \zeta), (\chi', \zeta') \rangle)= [(d, \chi')+|\zeta'|]\chi+[(d, \chi)+|\zeta|]\chi',\\
\end{sis}
\end{equation*}
where $(d, \chi), (d, \chi')\in \bbZ$ are obtained from the perfect pairing \eqref{E:perpair}, $|\zeta|=\sum_i \zeta_i\in \bbZ$ and similarly for $|\zeta'|$.
		
		Moreover, the homomorphism \eqref{seq-bunTg0} is contravariant with respect to homomorphisms of tori.
		
		\item \label{BunTg0-2}
		The image of $w_T^d$ is equal to 
		  $$
		  \Im(w_T^d)=
		  \begin{cases}
		  \Lambda^*(T) & \text{ if } n\geq 1,\\
		  \{\chi\in  \Lambda^*(T) \: : (d,\chi)\in 2\bbZ\} & \text{ if } n=0.
		  \end{cases}
		  $$
		  In particular, $w_T^d$ is an isomorphism if either $n\geq 1$ or $n=0$ and $d\in 2\Lambda(T)$, while it is an index two inclusion in the remaining cases. 
		\item \label{BunTg0-3} Fix an isomorphism $T\cong \bbG_m^r$ which induces  isomorphisms $\Lambda^*(T)\cong\Lambda^*(\bbG_m^r)=\bbZ^r$ and $\Lambda(T)\cong\Lambda(\bbG_m^r)=\bbZ^r$. Write $d=(d_1,\ldots, d_r)\in \bbZ^r$ under the above isomorphism $\Lambda(T)\cong \bbZ^r $.
		 Denote by $\{e_i\}_{i=1}^r$ the canonical basis of  $\bbZ^r$, by $\{f_j\}_{j=1}^n$ the canonical basis of $\bbZ^n$ and by $\{\epsilon_i\}_{i=1}^r$ a basis of the subgroup 
	        $$\{\chi=(\chi_1,\ldots, \chi_r)\in \bbZ^r\: : (\chi, d)=\sum_{i=1}^r \chi_i d_i\in 2\bbZ\}\subseteq \bbZ^r.$$ 	 
		
		The relative Picard group of $\mathrm{Bun}_{T,0,n}^d$ is freely generated by 
		$$
		\begin{cases} 
		\left\{\langle (e_i,0),(0,f_1)\right\rangle=\sigma_1^*(\mt L_{e_i})\}_{i=1}^r & \text{ if } n\geq 1, \\
		\left\{ d_{\pi}(\omega_{\pi}^{\frac{(d,\epsilon_i)}{2}}\otimes \mt L_{\epsilon_i})=\mathscr L(\epsilon_i)^{1-(d,\epsilon_i)}\otimes  
		\langle \epsilon_i, \epsilon_i\rangle^{\frac{(d,\epsilon_i)}{2}} \right\}_{i=1}^r &\text{if } n=0.\\
		\end{cases}
		$$
			\end{enumerate}	
\end{teo}

With the aim of proving the above Theorems, we will study in the next subsection the restriction of the line bundles on $\bg{T}$ to the geometric fibers over $\Mg$.

\subsection{The  restriction of the Picard group to the fibers I}\label{Sec:gerbe-fib}

The aim of this subsection is to define a $T$-gerbe structure on $\bg{T}$ and to study the restriction of the relative Picard group of $\bg{T}$ to a fiber over a geometric point of $\Mg$.
 
The automorphism group of any $S$-point $(C\to S, \sigma_1,\ldots, \sigma_n, E)$  of $\bg{T}$ contains  the torus $T$, which is indeed isomorphic to the subgroup of automorphisms of $(C\to S, \sigma_1,\ldots, \sigma_n, E)$ that induce the identity on the underlying family of pointed curves $(C\to S,\sigma_1, \ldots, \sigma_n)$. Therefore we can consider the rigidification 
$\Pi_T:\bg{T}\to \JT:=\bg{T}\fatslash T$  of $\bg{T}$ by the torus $T$ (see \cite[Sec. 5]{Rom}). 
By definition, $\bg{T}$ is a $T$-gerbe over $J_{T,g,n}$. This implies that the connected components of $\JT$ are exactly 
$$\JT^d:=\Pi_T(\bg{T}^d) \quad \text{ for } d\in \pi_1(T)=\Lambda(T)$$
and $\Pi_T$ restricts to a $T$-gerbe $\Pi_T^d:\bg{T}^d\to \JT^d$. 
The stack $\JT^{(d)}$ admits a  forgetful morphism $\Psi_T^{(d)}:\JT^{(d)}\to\mathcal M_{g,n}$ over the moduli stack of curves
and we set:
$$\RPic(\JT^{d}):=\operatorname{Pic}(\JT^{d})/(\Psi_T^{d})^*\operatorname{Pic}(\mathcal M_{g,n}).$$

The  Leray spectral sequence for the \'etale sheaf $\Gm$ with respect to the morphism $\Pi_T^{d}$ gives the exact sequence of (relative) Picard groups:
\begin{equation}\label{E:Pic-gerbe}	
	0\to\text{Rel}\Pic(\JT^{d})\xrightarrow{(\Pi_T^{d})^*} \text{Rel}\Pic(\bg{T}^{d})\xrightarrow{w_T^{d}}\Pic(\mathcal BT)=\Lambda^*(T).
\end{equation}
The homomorphism $w_T^d$, called the \emph{weight function}, can be computed as follows. Let $\mathcal F$ be a  line bundle on $\bg{T}^d$ and fix a $k$-point $\xi:=(C,x_1,\ldots,x_n, E)$ of $\bg{T}^d$. The automorphism group of $\xi$ acts on the fiber $\mathcal F_{\xi}$ of $\mathcal F$ over $\xi$. Since the torus $T$ is contained in the automorphism group of $\xi$, this defines an action of $T$ on $\mathcal F_{\xi}\cong k$ which  is given by a character of $T$.  This character, which is independent of the chosen $k$-point $\xi$ and on the chosen isomorphism $\mathcal F_{\xi}\cong k$,  coincides with $w_T^d(\mathcal F)$. 

We now describe the fiber of $\Pi_T:\bg{T} \to \JT$ over a geometric point $(C,p_1,\ldots,p_n)$ of $\Mg$. Consider the Jacobian stack $\mt J(C)$ parametrizing line bundles on $C$ and its rigidification $J(C):=\mt J(C)\fatslash \Gm$, which is the Jacobian  variety of $C$. The connected components of $\mt J(C)$ and of $J(C)$ are 
$$\mt J(C)=\coprod_{e\in \bbZ} \mt J^e(C) \to J(C)=\coprod_{e\in \bbZ} J^e(C),$$
where $\mt J^e(C)$ is the stack parametrizing line bundles on $C$ of degree $e$ and $J^e(C)=\mt J^e(C)\fatslash \Gm$.  Equivalently, there is a commutative diagram of sets
\begin{equation}\label{E:degJac}
\xymatrix{ 
\mt J(C) \ar[rr] \ar[dr] & & J(C) \ar[dl] \\
& \bbZ & 
} 
\end{equation}
whose fiber over $e\in \bbZ$ is the $\Gm$-gerbe $\mt J^e(C)\to J^e(C)$.

The fiber of $\Pi_T:\bg{T} \to \JT$ over a geometric point  $(C,p_1,\ldots, p_n)$ of $\Mg$ is canonically isomorphic to the $T$-gerbe 
\begin{equation}\label{E:fiberT}
\mt J_T(C):=\Hom(\Lambda^*(T), \mt J(C))\stackrel{\Pi_T(C)}{\longrightarrow} \Hom(\Lambda^*(T), J(C))=:J_T(C).
\end{equation}
The connected components of \eqref{E:fiberT} are exactly the fibers of the $T$-gerbe $\Pi_T^d:\bg{T}^d\to \JT^d$ and they are given by  
\begin{equation}\label{E:fiberTcom}
\mt J_T^d(C):=\Hom_d(\Lambda^*(T), \mt J(C))\stackrel{\Pi_T^d(C)}{\longrightarrow} \Hom_d(\Lambda^*(T), J(C))=:J_T^d(C).
\end{equation}
where $\Hom_d$ denote the set of homomorphisms whose composition with the diagonal homomorphisms of \eqref{E:degJac}  is the element $d\in X(T)=\Hom(X^*(T), \bbZ)$. 

We now recall the explicit description of the Picard group of $\mt J_T^d(C)$, where $C$ is a smooth curve defined over an algebraically closed field $k$, adapting the description of \cite[Sec. 3]{BH10} from $d=0$ to an arbitrary $d\in \Lambda(T)$. 

First of all, the $\Gm$-gerbe $\mt J(C) \to J(C)$ is trivial, i.e., $\mt J(C)\cong J(C)\times B\Gm$, since it admits sections each of which correspond to a Poincar\'e line bundle on $J(C)\times C$. 
This implies that also the  $T$-gerbe $\Pi_T(C)$ (and hence also $\Pi_T^d(C)$) is trivial, i.e., $\mt J_T(C)\cong J_T(C)\times BT$. 
Hence the Leray spectral sequence for the \'etale sheaf $\Gm$ with respect to the morphism $\Pi_T^d(C)$ gives the following split short exact sequence of  Picard groups:
\begin{equation}\label{E:gerbe-fib}	
\xymatrix{
	0\ar[r] & \Pic(J_T^{d}(C))\ar[r]^{\Pi_T^d(C)^*}&   \Pic(\mt J_T^d(C))\ar[r]^(.4){w_T^d(C)}& \Pic(\mathcal BT)=\Lambda^*(T)  \ar@/^1pc/[l]^{s_p}\ar[r]&0,
	}
\end{equation}
where the section $s_p$, that depends on a chosen point $p\in C(k)$, sends an element $\chi\in X^*(T)=\Hom(T, \Gm)$ into the line bundle on $\mt J_T^d(C)$ naturally associated to the $\Gm$-bundle $\chi_\#(\cP_C)_{|\mt J_T^d(C)\times \{p\}}$, where $\cP_C$ is the universal $T$-bundle on $\mt J_T(C)\times C$. 

The continuous part of the Picard groups of $J_T^{d}(C)$ and of $\mt J_T^d(C)$ can be described as follows. Any character $\chi\in \Lambda^*(T)$ determines a morphism 
$$\begin{aligned}
\chi_\#:\mt J^d_T(C)=\Hom_d(\Lambda^*(T),\mt J(C))& \to \mt J^{(d,\chi)}(C) \\
\phi & \mapsto \phi(\chi).
\end{aligned}$$
Denote by $\scr L_{C,\chi}$ the pull-back via $\chi_\#\times \id_C: \mt J^d_T(C)\times C \to \mt J^{(d,\chi)}(C)\times C$ of the universal line bundle $\scr L_C$ on $\mt J^{(d,\chi)}(C)\times C$, and let $p_1$ and $p_2$ be the projections of  $\mt J^d_T(C)\times C$ onto the first and second factor, respectively. There is an injective homomorphism 
\begin{equation}\label{E:jTC}
\begin{aligned}
j_T^d(C): \Hom(\Lambda(T),J_C(k))=\Lambda^*(T)\otimes J_C(k) & \hookrightarrow \Pic(\mt J_T^d(C)), \\
\chi\otimes N & \mapsto \langle \scr L_{C,\chi}, p_2^*(N)\rangle_{p_1}.
\end{aligned}
\end{equation}
Using the analogue of formula \eqref{Del-pair} for the morphism $p_1$ and the functoriality of the determinant of cohomology, we get  
 \begin{equation}\label{E:jTC2}
\langle \scr L_{C,\chi}, p_2^*(N)\rangle_{p_1}=d_{p_1}(\scr L_{C,\chi}\otimes p_2^*(N))\otimes d_{p_1}(\scr L_{C,\chi})^{-1}=t_N^*(d_{p_1}(\scr L_{C,\chi}))\otimes d_{p_1}(\scr L_{C,\chi})^{-1},
\end{equation}
where $t_N:\mt J_T^d(C) \to \mt J_T^d(C)$ is the translation by $N$. Arguing as in \cite[Lemma 6.2]{MV} and using that the line bundles $\scr L_{C,\chi}\otimes p_2^*(N)$ and $\scr L_{C,\chi}$ have both $p_1$-relative degree equal to $(d,\chi)$, we compute 
$$
w_T^d(C)(d_{p_1}(\scr L_{C,\chi}\otimes p_2^*(N)))=[(d,\chi)+1-g]\chi=w_T^d(C)(d_{p_1}(\scr L_{C,\chi})).
$$
This, together with formula \eqref{E:jTC2}, implies that the weight of the line bundles  $\langle \scr L_{C,\chi}, p_2^*(N)\rangle_{p_1}$ are zero; hence the morphism $j_T^d(C)$ factors as 
\begin{equation}\label{E:iTC}
\xymatrix{
j_T^d(C):\Lambda^*(T)\otimes J_C(k) \ar@{^{(}->}[rr]^(0.6){\ov j_T^d(C)} && \Pic(J_T^d(C))\ar@{^{(}->}[rr]^{\Pi_T^d(C)^*} && \Pic(\mt J_T^d(C)).
}
\end{equation}
The quotients 
$$
\xymatrix{
 \NS(J_T^d(C)):=\Pic(J_T^d(C))/\Im(\ov j_T^d(C))\ar@{^{(}->}[r]^{\Pi_T^d(C)^*} & \NS(\mt J^d_T(C)):=\Pic(\mt J_T^d(C))/\Im(j_T^d(C))
}
$$
are called the Neron-Severi groups;  they are discrete groups that admit the following description. 

First of all, at the level of the Neron-Severi groups the splitting of the  sequence \eqref{E:gerbe-fib} is canonical (i.e., independent of the point $p\in C$),  and hence we get 
a canonical isomorphism
\begin{equation}\label{E:NerSev}
\begin{aligned}
\NS(\mt J_T^d(C))& \xrightarrow{\cong} \Lambda^*(T)\oplus \NS(J_T^d(C)),\\
[L] & \mapsto ( w_T^d(L),[L\otimes s_p(w_T^d(L))^{-1}]).
\end{aligned}
\end{equation}

The Neron-Severi group of $J_T^d(C)$ admits the following explicit description.  The Jacobian $J_C=J^0(C)$  of $C$ is endowed with a standard principal polarization  $\phi_{\theta}:J_C\xrightarrow{\cong} J_C^{\vee}$  induced by the theta divisor. This determines an involution (called the Rosati involution) $\dagger:\End(J_C)\to \End(J_C)$ on the endomorphism algebra of $J_C$ by sending $\alpha$ to $\phi_{\theta}^{-1}\circ \alpha^{\vee}\circ \phi_{\theta}$. Denote by $\Hom^s(\Lambda(T)\otimes \Lambda(T),  \End(J_C))$  the abelian group of  homomorphisms that are symmetric with respect to the  Rosati involution on $\End(J_C)$, i.e., such that  $\phi(\lambda_1\otimes \lambda_2)=\phi(\lambda_2\otimes \lambda_1)^{\dagger}$. 
As explained in \cite[Cor. 3.1.3]{BH10}, there is an isomorphism 
\begin{equation}\label{E:NS-Rosati}
\begin{aligned}
c: \NS(J_T^d(C)) & \xrightarrow{\cong} \Hom^s(\Lambda(T)\otimes \Lambda(T),  \End(J_C)), \\
[L] & \mapsto \left\{(d_1,d_2) \mapsto \Big(J_C\xrightarrow{(d_1,-)\otimes \id_{J_C}} J_T^0(C) \xrightarrow{\phi_L} J_T^0(C)^{\vee} \xrightarrow{((d_2,-)\otimes \id_{J_C})^{\vee}} J_C^{\vee} 
\xrightarrow{\phi_{\theta}^{-1}} J_C\Big)\right\},
\end{aligned}
\end{equation}
where $\phi_L$ is the homomorphism of abelian varieties sending $a\in J_T^0(C)$ into $t_a^*(L)\otimes L^{-1}\in \Pic^0(J_T^d(C))=J_T^0(C)^{\vee}$. 

By putting everything together, we obtain the following description of the Picard groups of $\mt J^d_T(C)$ and of $J_T^d(C)$.

\begin{prop}\label{P:Pic-JTC}
Let $C$ be a curve over an algebraically closed field $k$ and let $d\in \Lambda(T)$. Then there is a diagram with exact rows and columns
\begin{equation*}
\xymatrix{ 
\Lambda^*(T)\otimes J_C(k) \ar@{^{(}->}[r]^(0.6){\ov j_T^d(C)} \ar@{=}[d]&\Pic(J_T^d(C)) \ar@{->>}[rr]^(0.3){\ov \gamma_T^d(C)} \ar@{^{(}->}[d]^{\Pi^d_T(C)^*}&& \Hom^s(\Lambda(T)\otimes \Lambda(T),  \End(J_C))\ar@{^{(}->}[d]\\
\Lambda^*(T)\otimes J_C(k) \ar@{^{(}->}[r]^(0.6){j_T^d(C)} &\Pic(\mt J_T^d(C)) \ar@{->>}[rr]^(0.3){w_T^d(C)\oplus \gamma_T^d(C)}  \ar@{->>}[d]^{w_T^d(C)} && \Lambda^*(T)\oplus \Hom^s(\Lambda(T)\otimes \Lambda(T),  \End(J_C)) \ar@{->>}[d]\\
& \Lambda^*(T) \ar@{=}[rr] && \Lambda^*(T)
}
\end{equation*}
where:
\begin{enumerate}
	\item $j_T^d(C)$ and $\ov j_T^d(C)$ are the homomorphisms defined in \eqref{E:jTC} and \eqref{E:iTC}, respectively;
	\item $\ov \gamma_T^d(C)$ is the composition of the morphism $\Pic(J_T^d(C))\twoheadrightarrow \NS(J_T^d(C))$ and the isomorphism \eqref{E:NS-Rosati};
	\item $\gamma_T^d(C)$ is the composition of the morphism $\Pic(\mt J_T^d(C))\twoheadrightarrow \NS(\mt J_T^d(C))$, the isomorphism \eqref{E:NerSev} and the isomorphism $\id_{\Lambda^*(T)}\oplus c$, where $c$ is the morphism \eqref{E:NS-Rosati}.
\end{enumerate} 
Moreover, the above diagram is functorial with respect to homomorphism of tori. 
\end{prop}
Note that if  $C$ has genus zero (in which case  $J_C=0$),   the above diagram gives that  $\Pic(J_T^d(C))=0$ and that the weight function $w_T^d(C):\Pic(\mt J_T^d(C))\to \Lambda^*(T)$ is an isomorphism. 

\hspace{0.1cm}

Now, given a geometric point $(C,p_1,\ldots, p_n)$ of $\Mg$, we are going to write down a formula for the restriction homomorphism towards the Neron-Severi group 
$$
 \RPic(\bg{T}^d)\xrightarrow{\res_T^d(C)} \Pic(\mt J_T^d(C)) \twoheadrightarrow \NS(\mt J_T^d(C))
$$
 on the tautological classes. 
Note that the canonical map $\id_{J_C}:\bbZ\to \End(J_C)$  given by the addition on the abelian variety $J_C$ induces a homomorphism 
\begin{equation}\label{E:inc-bs}
-\otimes \id_{J_C}: \Bil^s\Lambda(T)\to \Hom^s(\Lambda(T)\otimes \Lambda(T),  \End(J_C)),
\end{equation}
where the source is the group of integral bilinear symmetric forms on $\Lambda(T)$, which has been introduced in Subsection \ref{Sec:int-forms}. Note that this homomorphism is injective if  the genus of $C$ is positive and it is identically zero if the genus of $C$ is zero.

\begin{prop}\label{P:restr}
Let $(C,p_1,\ldots, p_n)$ be a geometric point of $\Mg$ and let $d\in \Lambda(T)$. 
\begin{enumerate}
\item \label{P:restr1} The composition 
$$ \RPic(\bg{T}^d)\xrightarrow{\res_T^d(C)} \Pic(\mt J_T^d(C))\xrightarrow{w_T^d(C)} \Lambda^*(T)
$$
coincides with the weight function $w_T^d$ of \eqref{E:Pic-gerbe} and it is given on the  tautological classes of $\RPic(\bg{T}^d)$  by 
\begin{equation*}
\begin{sis}
& w_T^d(\scr L(\chi, \zeta))=[(d, \chi)+|\zeta|+1-g]\chi,\\
& w_T^d(\langle (\chi, \zeta), (\chi', \zeta') \rangle)= [(d, \chi')+|\zeta'|]\chi+[(d, \chi)+|\zeta|]\chi',\\
\end{sis}
\end{equation*}
where $(d, \chi), (d, \chi')\in \bbZ$ are obtained from the perfect pairing \eqref{E:perpair}, $|\zeta|=\sum_i \zeta_i\in \bbZ$ and similarly for $|\zeta'|$.

\item \label{P:restr2} The composition 
$$ \wt \gamma_T^d(C): \RPic(\bg{T}^d)\xrightarrow{\res_T^d(C)} \Pic(\mt J_T^d(C))\xrightarrow{\gamma_T^d(C)} \Hom^s(\Lambda(T)\otimes \Lambda(T), \End(J_C))
$$
is given on the tautological classes of $\RPic(\bg{T}^d)$ by 
\begin{equation*}
\begin{sis}
& \wt\gamma_T^d(C)(\scr L(\chi, \zeta))=(\chi\otimes \chi)\otimes \id_{J_C} \\
& \wt\gamma_T^d(C)(\langle (\chi, \zeta), (\chi', \zeta') \rangle)=(\chi\otimes \chi'+\chi'\otimes \chi)\otimes \id_{J_C}, \\
\end{sis}
\end{equation*}
where $\chi\otimes \chi$ and $\chi\otimes \chi'+\chi'\otimes \chi$ are elements of $(\Lambda^*(T)\otimes \Lambda^*(T))^s$ which is canonically identified with $\Bil^s\Lambda(T)$ by \eqref{E:iso-b}.
\end{enumerate} 
\end{prop} 
Note that, once Theorems \ref{BunT} and \ref{BunTg0} are proved, we will know that $\RPic(\bg{T}^d)$ is generated by tautological classes, and hence the above Proposition gives a complete description of the restriction homomorphism towards the N\'eron-Severi group on the entire $\RPic(\bg{T}^d)$. The full restriction morphism  $\res_T^d(C): \RPic(\bg{T}^d)\to\Pic(\mt J_T^d(C))$ will be described in Proposition \ref{P:restrPic}.

\begin{proof}
The formulas for $\langle (\chi, \mu), (\chi',\mu')\rangle$ follow from the ones for $\scr L(\chi, \mu)$ and equation \eqref{E:Del-det}. Hence, it is enough to prove the formulas for $\mt L(\chi, \mu)$.

Let us first prove part \eqref{P:restr1}. Clearly,  the weight function $w_T^d$ is equal to the composition $w_T^d(C)\circ \res_T^d(C)$. 
Arguing as in \cite[Lemma 6.2]{MV} and using that $\scr L(\chi, \mu)=d_{\pi}(\mt L_{\chi}(\sum_i \zeta_i p_i))$ and that the $\pi$-relative Euler-Poincar\'e characteristic of $\mt L_{\chi}(\sum_i \zeta_i p_i)$ is equal to 
$[(d,\chi)+|\zeta|+1-g]$ since $\mt L_{\chi}(\sum_i \zeta_i p_i)$ has $\pi$-relative degree equal to $(d,\chi)+|\zeta|$, we deduce that 
\begin{equation*}
w_T^d(\scr L(\chi, \zeta))=[(d, \chi)+|\zeta|+1-g]\chi.
\end{equation*}

Let us now compute $\wt\gamma_T^d(C)(\scr L(\chi, \zeta))$. To ease the notation, we set 
$$
\begin{sis}
&\wt L:=\scr L(\chi, \zeta)_{|C}\in \Pic(\mt J_T^d(C)), \\
&L:=\wt L\otimes s_p(w_T^d(\scr L(\chi, \zeta)))^{-1} \in \Pic(J_T^d(C)).
\end{sis}
$$ 
By \eqref{E:NerSev} and \eqref{E:NS-Rosati} , we have that 
\begin{equation}\label{E:rho1}
\wt\gamma_T(C)(\scr L(\chi, \zeta))=c([L]).
\end{equation}
In order to compute $c([L])$, let us analyze the morphism $\phi_{L}:J_T^0(C)\to J_T^0(C)^{\vee}$. The fiber over $(C,p_1,\ldots, p_n)$ of the morphism $  \chi_\#:\bg{T}^d\to\bg{\Gm}^{(d,\chi)}$ (resp., of its rigidification $  \chi_\#: \JT^d \to \JGm^{(d,\chi)}$) is equal to the morphism 
$$\begin{aligned}
\chi:\mt J^d_T(C)=\Hom_d(\Lambda^*(T),\mt J(C))& \to \mt J^{(d,\chi)}(C) \\
\phi & \mapsto \phi(\chi)
\end{aligned}
\quad (\text{resp., }  \chi:J^d_T(C)\to J^{(d,\chi)}(C)).$$
Consider the universal line bundle $\scr L_C$ on $\mt J^{(d,\chi)}(C)\times C$ and set 
$$
\begin{sis}
& \wt M:=d_{p_1}(\scr L_C(\mu_1p_1+\ldots +\mu_n p_n))\in \Pic(\mt J^{(d,\chi)}(C)), \\
& M:= d_{p_1}(\scr L_C(\mu_1p_1+\ldots +\mu_n p_n))\otimes (\scr L_C^{-w_T^d(\scr L(\chi, \zeta))})_{|\mt J^{(d,\chi)}(C)\times \{p\}}\in \Pic(J^{(d,\chi)}(C)),
\end{sis}
$$ 
where $p_1: \mt J^{(d,\chi)}(C)\times C\to \mt J^{(d,\chi)}(C)$ is the first projection. It is well known (see, for example, \cite[Chap. 17]{Pol}) that the morphism $\phi_M:J_C\to J_C^{\vee}$ is equal to the standard principal polarization $\phi_{\theta}:J_C\xrightarrow{\cong} J_C^{\vee}$. 

By definition of the tautological line bundle  $\scr L(\chi, \zeta)$ in  \S \ref{tau}, it follows that 
$$ \wt L=\chi^*(\wt M) \quad \text{ and } \quad L=\chi^*(M). $$
Since the morphism $\chi:J^d_T(C)\to J^{(d,\chi)}(C)$ is equivariant with respect to the morphism of abelian varieties $\chi:J^0_T(C)\to J^{0}(C)=J_C$, we compute for any $a\in J^0_T(C)$:
$$
\phi_L(a)=t_a^*(\chi^*(M))\otimes \chi^*(M)^{-1}= \chi^*(t_{\chi(a)}^*(M))\otimes \chi^*(M^{-1})=\chi^*(\phi_M(\chi(a)).
$$
Thus we conclude that the morphism $\phi_L$ is equal to the following composition 
\begin{equation}\label{E:PhiLM}
\phi_L: J_T^0(C)\xrightarrow{\chi} J_C\xrightarrow{\phi_M=\phi_{\theta}} J_C^{\vee}\xrightarrow{\chi^{\vee}} J_T^0(C)^{\vee}.
\end{equation}
Using this, it follows from \eqref{E:NS-Rosati} that $c([L])$ sends $(d_1,d_2)\in \Lambda(T)\otimes \Lambda(T)$ to the endomorphism of $J_C$ given by the following composition 
$$
c([L])(d_1,d_2): J_C\xrightarrow{(d_1,-)\otimes \id_{J_C}} J_T^0(C) \xrightarrow{\chi} J_C\xrightarrow{\phi_M=\phi_{\theta}} J_C^{\vee}\xrightarrow{\chi^{\vee}}  J_T^0(C)^{\vee} \xrightarrow{((d_2,-)\otimes \id_{J_C})^{\vee}} J_C^{\vee}  \xrightarrow{\phi_{\theta}^{-1}} J_C.
$$
We conclude that $c([L])(d_1,d_2)$ is equal to the multiplication on $J_C$ by  $(d_1,\chi)\cdot (d_2,\gamma)$, which is equivalent to say that $c([L])=(\gamma\otimes \gamma)\otimes \id_{J_C}$. 
\end{proof}
 
 The results of this subsection are already sufficient to compute the relative Picard group of $\bg{T}^d$ in genus $0$.

 \begin{proof}[Proof of Theorem \ref{BunTg0}]
 First of all, since the Jacobian of a curve of genus $0$ is trivial, it follows that the morphism $\Psi_T:\text{J}_{T,0,n}\to\mathcal M_{0,n}$ is an isomorphism. 
 Hence, from the exact sequence \eqref{E:Pic-gerbe}, we deduce that the weight function 
 \begin{equation*}
\RPic\left(\mathrm{Bun}_{T,0,n}^d\right)\xrightarrow{w_T^d} \Lambda^*(T),
\end{equation*}
 is an injective  homomorphism. The weight function applied to the tautological classes in $\Pic( \mathrm{Bun}_{T,0,n}^d)$ has been computed in Proposition \ref{P:restr} and it coincides with the formula given in Theorem \ref{BunTg0}\eqref{BunTg0-1}.   The functoriality of the homomorphism   \eqref{seq-bunTg0} will follow, once we will have proved that 
 $\RPic\left(\mathrm{Bun}_{T,0,n}^d\right)$ is generated by tautological classes,  from the functoriality properties of the Deligne pairing and the determinant of cohomology.
 In order to prove the remaining statements in Theorem \ref{BunTg0}, we will  distinguish the two cases $n>1$ and $n=0$.
 
 \un{Case I:}  $n>0$.
 
 If we fix an isomorphism $T\cong \bbG_m^r$ and use the same notation as in the statement of Theorem \ref{BunTg0}\eqref{BunTg0-3}, we get that 
 $$w_T^d(\langle (e_i,0),(0,f_1)\rangle)=e_i \text{ for any } i=1,\ldots, r.$$
 We deduce that the weight function $w_T^d$ is an isomorphism and that the elements $\left\{ \langle (e_i,0),(0,f_1)\rangle\right\}_{i=1}^r$ form a basis of  $\RPic\left(\mathrm{Bun}_{T,0,n}^d\right)$, which is therefore generated by tautological classes.

 \un{Case II:} $n=0$.

First of all, with the notation of Theorem \ref{BunTg0}\eqref{BunTg0-3} and using  \eqref{E:tau-dual}, we get that 
$$
w_T^d\left(d_{\pi}(\omega_{\pi}^{\frac{(d,\epsilon_i)}{2}}\otimes \mt L_{\epsilon_i})\right)=w_T^d\left(\mathscr L(\epsilon_i)^{1-(d,\epsilon_i)}\otimes  
		\langle \epsilon_i, \epsilon_i\rangle^{\frac{(d,\epsilon_i)}{2}} \right)=(1-d_i)(d_i+1)\epsilon_i+\frac{d_i}{2}2d_i\epsilon_i=\epsilon_i.
$$
This implies that 
\begin{equation}\label{E:Im-incl}
 \{\chi\in  \Lambda^*(T) \: : (d,\chi)\in 2\bbZ\}\subseteq \Im w_T^d.
\end{equation}
It remains to prove that equality holds, to which also implies that the elements $\left\{d_{\pi}(\omega_{\pi}^{\frac{(d,\epsilon_i)}{2}}\otimes \mt L_{\epsilon_i})\right\}_{i=1}^r$ form a basis of $\RPic(\mathrm{Bun}_{T,0,0}^d)$. 

Consider the morphism $F_T:\mathrm{Bun}^d_{T,0,1}\to\mathrm{Bun}^d_{T,0,0}$ forgetting the section , which is also the universal curve over $\mathrm{Bun}^d_{T,0,0}$. The pull-back along $F_T$ induces an inclusion of Picard groups
$$
F_T^*:\Pic(\mathrm{Bun}^d_{T,0,0})\hookrightarrow \Pic(\mathrm{Bun}^d_{T,0,1}).
$$
Using the explicit descriptions $\cM_{0,0}\cong \cB\PGL_2$ and $\cM_{0,1}\cong \cB(\Ga\rtimes \Gm)$, it is easy to see that $\Pic(\cM_{0,0})=\Hom(\PGL_2,\Gm)=0$ and $\Pic(\cM_{0,1})=\Hom(\Ga\rtimes \Gm, \Gm)=\Hom(\Gm,\Gm)=\bbZ$ generated by $\psi_1=\sigma_1^*(\omega_{\pi})$. 
Therefore, we conclude that $\RPic(\mathrm{Bun}^d_{T,0,0})=\Pic(\mathrm{Bun}^d_{T,0})$ and, using the Case I proved above, that any line bundle on  $\mathrm{Bun}^d_{T,0,1}$ can be written as $\langle (\chi, 0), (0, f_1)\rangle\otimes \psi_1^n=\sigma_1^*(\mt L_{\chi}\otimes \omega_{\pi}^n)$ for some unique $\chi\in \Lambda^*(T)$ and some unique $n\in \bbZ$. Note also that $w_T^d(\sigma_1^*(\mt L_{\chi}\otimes \omega_{\pi}^n))=\chi$, as follows from the explicit formula for $w_T^d$. 

Consider now a line bundle $\cN\in \Pic(\mathrm{Bun}_{T,0,0}^d)=\RPic(\mathrm{Bun}_{T,0,0}^d)$ and let $\chi:=w_T^d(\cN)\in \Lambda^*(T)$.   Since the pull-back $F_T^*$ commutes with the weight functions for $\mathrm{Bun}_{T,0,0}^d$ and for $\mathrm{Bun}_{T,0,1}^d$ (which we have been denoting with the same symbol $w_T^d$), we have that $F_T^*(\cN)=\sigma_1^*(\mt L_{\chi}\otimes \omega_{\pi}^n)$ for some $n\in \bbZ$. Now $F_T^*(\cN)$ is trivial on the geometric fibers of $F_T$ while $\sigma_1^*(\mt L_{\chi}\otimes \omega_{\pi}^n)$ has $F_T$-relative degree equal to  $(d,\chi)-2n$. This is only possible if $(d, \chi)=2n$, which implies that 
$$
\chi \in  \{\chi\in  \Lambda^*(T) \: : (d,\chi)\in 2\bbZ\}.
$$
Hence, equality holds in \eqref{E:Im-incl} and we are done. 
 \end{proof}

\subsection{The relative Picard group of $\bg{T}^d$ in genera $g\geq 1$}\label{S:RPicT}

The aim of this  subsection is to prove Theorem \ref{BunT}.
We will first exhibit explicit generators of $\RPic(\bg{T}^d)$ for  $g\geq 1$, thus proving Theorem \ref{BunT},\eqref{BunT2} and then deduce from it Theorem \ref{BunT}\eqref{BunT1}. 

We fix an isomorphism  $T\cong \bbG_m^r$ which induces an isomorphism $\Lambda^*(T)\cong\Lambda^*(\bbG_m^r)=\bbZ^r$. Then we can identify the objects in $\bg{T}$ with the vector bundles of rank $r$ which are direct sums of line bundles. According to Theorem \ref{concomp}, the connected components of $\bg{T}$ correspond to the $r$-uples $d:=(d_1,\ldots,d_r)$ of integers, which correspond to the degrees of each line bundle in the splitting. 

We will first focus on the case $d=(0,\ldots, 0)$ and $n\geq 1$, and then deal with the other cases.

\begin{lem}\label{RigBunT0mark}
Assume $n\geq 1$ and $g\geq 1$.  The group $\RPic(\JT^0)$ is freely generated by
$$\def\arraystretch{1.5}\begin{array}{ll}
	\left\langle\mathcal L_{e_i}\big((g-1)\sigma_1\big),\omega_{\pi}\big((2-2g)\sigma_1\big)\right\rangle&\text{for } i=1,\ldots,r, \quad \text{ if } g\geq 2,\\
	\left\langle\mathcal L_{e_i}\big((g-1)\sigma_1\big),\oo(\sigma_j-\sigma_{j+1})\right\rangle& \text{for }i=1,\ldots,r, \text{ and }j=1,\ldots,n-1,\\
	 d_\pi\left(\mathcal L_{e_i}\big((g-1)\sigma_1\right) &\text{for } i=1,\ldots,r,\\
d_\pi\left(\mathcal L_{e_i}\otimes\mathcal L_{e_k}((g-1)\sigma_1)\right)& \text{for }1\leq i<k\leq r.
	\end{array}
$$
\end{lem}

\begin{proof}Set $\{J_r\to\mathcal M\}:=\{\JT^0\to \mathcal M_{g,n}\}$, where $r=\dim T$. We remark that $J_r\to\mathcal M$ is an abelian stack over $\mathcal M$, i.e., a representable morphism of stacks such that every geometric fiber is an abelian variety. The identity section is given by the trivial vector bundle $\oo_{C_{g,n}}^{r}$ on the universal curve of $\mathcal M_{g,n}$. We recall the following facts:
	\begin{enumerate}[(i)]
		\item  There exists an exact sequence of abstract groups:
		\begin{equation}\label{E:Picabst}
		0\to\text{Hom}_{\mathcal M}\left(\mathcal M, J^\vee_r\right)\xrightarrow{\nu} \RPic(J_r)\xrightarrow{\gamma_r(C)}\NS(J_C^{r}),
		\end{equation}
		where $J^\vee_r$ is the dual abelian stack of $J_r$ and $\gamma_r(C)$ is the restriction to the Neron-Severi group of $r$ copies of the Jacobian of some curve $C$.
		 It can be proved using an argument similar to the proof of \cite[Proposition 3.6]{FP16}.
		\item \label{gen-curve} There exists a curve of genus $g\geq 1$ over the base field $k$ such that the natural homomorphism $\id_{J_C}:\bbZ\to \End(J_C)$ is an isomorphism (see \cite{Mo76}).
	\end{enumerate}

If $r=1$, then the line bundle $d_\pi\left(\mathcal L_{e_1}((g-1)\sigma_1\right)$ induces a principal polarization on $J_1$. Since $J_r\cong J_1\times_{\mathcal M}\cdots\times_{\mathcal M}J_1$, it follows that the line bundle $\bigotimes_{i=1}^rd_\pi\left(\mathcal L_{e_i}((g-1)\sigma_1\right)$ induces a principal polarization on $J_r$. This implies that we have the following isomorphism of abelian stacks
$$
\begin{array}{rcc}
J_r&\xrightarrow{\cong}& J_r^{\vee}\\
(C\to S,\{\sigma_j\},\bigoplus_{i=1}^r F_i)&\mapsto& \left(C\to S, \{\sigma_j\}, \bigoplus_{i=1}^r d_\pi\Big( F_i\otimes\mathcal L_{e_i}\big((g-1)\sigma_1)\Big)\otimes d_\pi\Big(\mathcal L_{e_i}\big((g-1)\sigma_1)\Big)^{-1}\right).
\end{array}
$$
Using the above isomorphism and the identification $\text{Hom}_{\mathcal M}(\mathcal M,J_1)=\operatorname{RelPic^0}(\mathcal C_{g,n})$, we deduce an isomorphism of abelian groups: 
$$
\begin{array}{ccc}
\RPic^0(\mathcal C_{g,n})^{\oplus r}=\text{Hom}_{\mathcal M}(\mathcal M,J_r)&\xrightarrow{\cong}& \text{Hom}_{\mathcal M}\left(\mathcal M, J^\vee_r\right)\\
\displaystyle\bigoplus_{i=1}^rF_i&\mapsto& \displaystyle\bigoplus_{i=1}^rd_\pi\Big( F_i\otimes\mathcal L\big((g-1)\sigma_1)\Big)\otimes d_\pi\Big(\mathcal L\big((g-1)\sigma_1)\Big)^{-1}.
\end{array}
$$
Using Corollary \ref{C:Franch0} of the weak Franchetta's conjecture (see Theorem \ref{franchetta}) and the above isomorphism, we deduce that the group $\text{Hom}_{\mathcal M}\left(\mathcal M, J^\vee_r\right)$ is freely generated by:
\begin{equation*}\def\arraystretch{1.5}\begin{array}{ll}
d_\pi\Big( \omega_{\pi}\big((2-2g)\sigma_1)\otimes\mathcal L_{e_i}\big((g-1)\sigma_1)\Big)\otimes d_\pi\Big(\mathcal L_{e_i}\big((g-1)\sigma_1)\Big)^{-1}&\text{for } i=1,\ldots,r, \quad \text{ if } g\geq 2,\\
d_\pi\Big( \oo(\sigma_j-\sigma_{j+1})\otimes\mathcal L_{e_i}\big((g-1)\sigma_1)\Big)\otimes d_\pi\Big(\mathcal L_{e_i}\big((g-1)\sigma_1)\Big)^{-1}& \text{for }i=1,\ldots,r, \text{ and }j=1,\ldots,n-1.
\end{array}
\end{equation*}
Using \eqref{Del-pair}, we see that the image of the homomorphism $\nu$ in \eqref{E:Picabst} is freely generated by:
\begin{equation}\label{E:Imnu}
\def\arraystretch{1.5}\begin{array}{ll}
\Big\langle\omega_{\pi}\big((2-2g)\sigma_1),\mathcal L_{e_i}\big((g-1)\sigma_1\big)\Big\rangle&\text{for } i=1,\ldots,r, \quad \text{ if } g\geq 2,\\
\Big\langle\oo(\sigma_j-\sigma_{j+1}),\mathcal L_{e_i}\big((g-1)\sigma_1\big)\Big\rangle& \text{for }i=1,\ldots,r, \text{ and }j=1,\ldots,n-1.
\end{array}
\end{equation}
Consider now the following line bundles on $\bg{T}^0$:
\begin{equation}\label{genind}\def\arraystretch{1.5}\begin{array}{ll}
d_\pi\Big(\mathcal L_{e_i}\big((g-1)\sigma_1)\Big), &\text{for } i=1,\ldots,r,\\
d_\pi\Big(\mathcal L_{e_i}\otimes\mathcal L_{e_k}((g-1)\sigma_1)\Big),& \text{for }1\leq i<k\leq r.
\end{array}
\end{equation}
By Proposition \ref{P:restr}, the weights of the above line bundles are $0$. Hence, by the exact sequence \eqref{E:Pic-gerbe}, they descend to line bundles on $\JT^0$. If we take a  curve $C$ as in \eqref{gen-curve}, the image of these line bundles under the restriction homomorphism $\gamma_r(C)$ of \eqref{E:Picabst} freely generate $\NS(J_C^r)$ because of Proposition \ref{P:restr}\eqref{P:restr2} (see also \S \ref{Sec:int-forms}).  

Using the exact sequence \eqref{E:Picabst}, we deduce that $\RPic(J_r)$ is freely generated by the lines bundles in \eqref{E:Imnu} together with the ones in \eqref{genind}, which concludes our proof.
\end{proof}

\begin{lem}\label{BunT0mark}Assume $n\geq 1$ and $g\geq 1$. 
The group $\operatorname{RelPic}(\bg{T}^0)$ is freely generated by
	$$
	\def\arraystretch{1.5}\begin{array}{ll}
	\Big\langle\mathcal L_{e_i}\big((g-1)\sigma_1\big),\omega_{\pi}\big((2-2g)\sigma_1\big)\Big\rangle&\text{for } i=1,\ldots,r, \quad \text{ if } g\geq 2,\\
	\Big\langle\mathcal L_{e_i}\big((g-1)\sigma_1\big),\oo(\sigma_j-\sigma_{j+1})\Big\rangle& \text{for }i=1,\ldots,r, \text{ and }j=1,\ldots,n-1,\\
	d_\pi\Big(\mathcal L_{e_i}\big((g-1)\sigma_1)\Big) &\text{for } i=1,\ldots,r,\\
	d_\pi\Big(\mathcal L_{e_i}\otimes\mathcal L_{e_k}((g-1)\sigma_1)\Big)& \text{for }1\leq i<k\leq r,\\
	d_\pi\Big(\mathcal L_{e_i}(g\cdot\sigma_1)\Big)&\text{for } i=1,\ldots,r.\\
	\end{array}
	$$
\end{lem}
\begin{proof}
Consider the exact sequence  \eqref{E:Pic-gerbe}. By Proposition \ref{P:restr}\eqref{P:restr1}, the images via the weight function $w_T^0$ of the line bundles 
$$
d_\pi\Big(\mathcal L_{e_i}(g\cdot\sigma_1)\Big) \text{ for } i=1,\ldots,r
$$
freely generate $\Lambda^*(T)$. By combining this with Lemma \ref{RigBunT0mark}, the proof follows. 
\end{proof}

We are now ready to prove the second part of Theorem \ref{BunT}.

\begin{proof}[Proof of Theorem \ref{BunT}\eqref{BunT2}] 
We will distinguish three cases. 

\un{Case I:} $n\geq 1$ and $d=0$. 

 By Lemma \ref{BunT0mark},  for $g\geq 2$ (resp., for $g=1$), the relative Picard group $\RPic(\bg{T}^0)$ is free of rank $rn+r(r+1)/2+r$ (resp., $rn+r(r+1)/2$), which is equal to the number of line bundles appearing in Theorem \ref{BunT}\eqref{BunT2}. Then, to prove the Theorem, it is enough to show that the generators in Lemma \ref{BunT0mark} can be expressed as integral combinations of the line bundles appearing in  Theorem \ref{BunT}\eqref{BunT2}. This follows from the following formulas in $\RPic(\bg{T}^0)$  (in additive notation), which are obtained using Properties (\ref{Del-pair-comp}) and (\ref{Del-pair}) of the Deligne pairing and Equation \eqref{E:tau-dual}:
$$\def\arraystretch{1.5}
\begin{array}{lcl}
\Big\langle\mathcal L_{e_i}\big((g-1)\sigma_1\big),\omega_{\pi}\big((2-2g)\sigma_1\big)\Big\rangle&=& \langle \cL_{e_i}, \omega_{\pi}\rangle +(2-2g) \langle (e_i,0),(0,f_1)\rangle=\\
&& \langle(e_i,0),(e_i,0)\rangle-2\cdot\mathscr L(e_i,0)+(2-2g)\cdot\left\langle (e_i,0),(0,f_1)\right\rangle,\\
\left\langle\mathcal L_{e_i}\big((g-1)\sigma_1\big),\oo(\sigma_j-\sigma_{j+1})\right\rangle&=&\left \langle (e_i,0),(0,f_j)\right\rangle-\left\langle (e_i,0),(0,f_{j+1})\right\rangle,\\
d_\pi\left(\mathcal L_{e_i}(m\cdot\sigma_1)\right)&=&
m\cdot\left\langle (e_i,0),(0,f_{1})\right\rangle+\mathscr L(e_i,0) \quad \text{ for any } m\in \bbZ,\\
d_\pi\left(\mathcal L_{e_i}\otimes\mathcal L_{e_k}((g-1)\sigma_1)\right)&=&\left \langle \cL_{e_i}\otimes \cL_{e_k}, \cO((g-1)\sigma_1)\right\rangle+d_\pi(\cL_{e_i}\otimes \cL_{e_k})=\\
\end{array}$$
$$=(g-1)\cdot \Big(\left \langle (e_i,0),(0,f_1)\right\rangle+\left\langle (e_k,0),(0,f_{1})\right\rangle\Big)
+\left \langle (e_i,0),(e_k,0)\right\rangle+\mathscr L(e_i,0)+\mathscr L(e_k,0).
$$

\un{Case II:} $n\geq 1$ and $d=(d_1,\ldots, d_r)$ arbitrary.

Consider the isomorphism over $\Mg$
$$\begin{aligned}
\mathfrak{t}:\bg{T}^d& \xrightarrow{\cong} \bg{T}^0\\
(C\to S, \{\sigma_j\}, \oplus_i F_i) & \mapsto (C\to S, \{\sigma_j\}, \oplus_ i F_i(-d_i \sigma_1))
\end{aligned}$$
The induced morphism $\wt{\mathfrak{t}}: \cC_{T,g,n}^d\to \cC_{T,g,n}^0$ on the universal families satisfies 
\begin{equation}\label{E:tauL}
\wt{\mathfrak{t}}^*(\cL_{e_i})=\cL_{e_i}(-d_i\sigma_1)\quad \text{ for any } i=1,\ldots, r.
\end{equation}
By Case I, the relative Picard group of $\bg{T}^0$ is freely generated by the line bundles appearing in Theorem \ref{BunT}\eqref{BunT2}. Hence, the relative Picard group of $\bg{T}^d$ is freely generated by the pull-backs of these line bundles along the isomorphism $\mathfrak{t}$. Using the functoriality of the determinant of cohomology and of the Deligne pairing, together with Properties \eqref{Del-pair} and \eqref{Del-pair-comp},   and Relation \eqref{E:tauL}, we get the following relations in $\RPic(\bg{T}^d)$ (in additive notation)
\begin{equation}\label{E:pulltau}
\begin{aligned}
& \mathfrak{t}^*(\langle (e_i,0), (0,f_j)\rangle)=\langle (e_i,0), (0,f_j)\rangle, \\
& \mathfrak{t}^*(\langle (e_i,0), (e_k,0)\rangle)=\langle (e_i,0), (e_k,0)\rangle-d_i\langle (e_k,0), (0,f_1)\rangle -d_k \langle (e_i,0), (0,f_1) \rangle, \\
& \mathfrak{t}^*\scr L(e_i, 0)=\scr L(e_i, 0)-d_i\langle (e_i, 0), (0,f_1)\rangle.
\end{aligned}
\end{equation}
From the above relations, it follows that $\RPic(\bg{T}^d)$ is freely generated by the line bundles  appearing in Theorem \ref{BunT}\eqref{BunT2}.

\un{Case III:} $n=0$ and $d=(d_1,\ldots, d_r)$ arbitrary.

We have to show that $\RPic(\mathrm{Bun}^d_{T, g})$ is freely generated  by the line bundles
\begin{equation}\label{gen-n=0}
\displaystyle\begin{array}{ll}
\left\langle (e_i,0),(e_k,0)\right\rangle, &\text{for } 1\leq i\leq k\leq r,\\
\mathscr L(e_i,0),&\text{for } i=1,\ldots,r.
\end{array}
\end{equation}
Consider the morphism $F_T:\mathrm{Bun}^d_{T,g,1}\to\mathrm{Bun}^d_{T,g}$ forgetting the section. Since $F_T$ is the pull-back of the universal family $\cC_g=\cM_{g,1}\to \cM_g$ along the Stein morphism $\Phi_T^d: \mathrm{Bun}^d_{T,g}\to \cM_g$ (see Proposition \ref{P:SteinG}), Lemma \ref{base-change} implies that the pull-back along $F_T$ induces an inclusion of relative Picard groups 
$$
F_T^*:\RPic(\mathrm{Bun}^d_{T,g})\hookrightarrow \RPic(\mathrm{Bun}^d_{T,g,1}).
$$
By the definition of the tautological line bundles it follows  that $F_T^*\mathscr L(\chi)=\mathscr L(\chi,0)$ and $F_T^*(\langle (\chi, 0), (\chi',0)\rangle)=\langle (\chi, 0), (\chi',0)\rangle$ for any $\chi, \chi'\in \Lambda^*(T)$. Using this and Case II, it follows that the line bundles in \eqref{gen-n=0} are linearly independent  in $\RPic(\mathrm{Bun}^d_{T,g})$ and that $\RPic(\mathrm{Bun}^d_{T,g,1})$ is freely generated by their pull-back via $F_T^*$ and the line bundles $\left\langle (e_i,0),(0,f_1)\right\rangle$ for $i=1,\ldots,r$. So, in order to conclude the  proof, it is enough to show that 
$$
\bigotimes_{i=1}^r \left\langle (e_i,0),(0,f_1)\right\rangle^{a_i}\in F_T^*\operatorname{RelPic}(\mathrm{Bun}^d_{T,g})\Longrightarrow a_i=0 \text{ for any }i.
$$
Assume that there exists $\mt K:=\bigotimes_{i=1}^r \left\langle (e_i,0),(0,f_1)\right\rangle^{a_i}$ descending to $\operatorname{RelPic}(\mathrm{Bun}^d_{T,g})$. For $i=1,\ldots,r$, let $L_i$ be a line bundle of degree $d_i$ over a curve $C$. We have a cartesian diagram of stacks
\begin{equation}\label{universal}
\xymatrix{
C\ar[d]\ar[r]^{\tilde g}&\mathrm{Bun}^d_{T,g,1}\ar[d]^{F_T}\\
\Spec k\ar[r]^{ g}&\mathrm{Bun}^d_{T,g}
}
\end{equation}
where $g$ is the morphism corresponding to the $k$-object $(C,L_1\oplus\cdots\oplus L_r)$. Then $\tilde g$ corresponds to the $C$-object $$(C\times C\xrightarrow{pr_2} C,\Delta,L_1\boxtimes\oo_C\oplus\cdots\oplus L_r\boxtimes\oo_C),$$ where $\Delta:C\to C\times C$ is the diagonal embedding. By (\ref{Del-pair-comp}), we have the following equalities in $\RPic(\mathrm{Bun}_{T,g,1}^d)$
$$
\langle (e_i,0),(0,f_1)\rangle= \sigma_1^*\left(\mathcal L_{e_i}\right).$$
In particular, we have the following equality of line bundles on $C$:
$$
\tilde g^*\left\langle (e_i,0),(0,f_j)\right\rangle\cong \Delta^*\left(L_i\boxtimes\oo_C\right)\cong L_i,
$$
which implies that $\tilde g^*\mt K\cong \bigotimes_{i=1}^rL_i^{a_i}$. Now, if the coefficients $a_i$ are not all zero and the line bundles $L_i$ are generic, then the line bundle $\tilde g^*\mt K$ is non-trivial (using that the genus of $C$ is non zero). This implies that $\mt K$ is not the pull-back of a line bundle on  $\mathrm{Bun}_{T,g}$, as desired. 
\end{proof}

We can finally show that the first part of Theorem \ref{BunT} is implied by the second part.

\begin{proof}[Proof of Theorem \ref{BunT}\eqref{BunT1}]
First of all, because of the bilinearity and symmetry of the Deligne pairing (see \eqref{Del-pair-comp}), the maps $\tau_T$ and $\sigma_T$ are well-defined homomorphisms of groups 
and the image of $\tau_T+\sigma_T$ is  the subgroup of $\RPic\left(\bg{T}^d\right)$ generated by all the line bundles of the form $\langle (\chi, \zeta), (\chi', \zeta')\rangle$, with $\chi, \chi'\in X^*(T)$ and $\zeta, \zeta'\in \bbZ^n$. 
Theorem \ref{BunT}\eqref{BunT2} shows that the map $\tau_T+\sigma_T$ is injective and hence it gives rise to the short exact sequence of abelian groups 
\begin{equation}\label{E:seq-Del}
0\to\Sym^2\Lambda^*(T)\oplus\left(\Lambda^*(T)\otimes\bbZ^n\right)\xrightarrow{\tau_T+\sigma_T}\RPic\left(\bg{T}^d\right)\rightarrow \coker(\tau_T+\sigma_T)\to 0.
\end{equation}
Next, consider the map 
\begin{equation}\label{E:map-tau}
\begin{aligned}
\eta: \Lambda^*(T)& \longrightarrow  \coker(\tau_T+\sigma_T) \\
\chi & \mapsto [\mathscr L(\chi, 0)].
\end{aligned}
\end{equation}
Equation \eqref{E:Del-det} implies that the map $\eta$ is a homomorphism of abelian groups. Moreover, the same equation  \eqref{E:Del-det} also implies that 
 $$\langle (\chi, 0), (0,\zeta)\rangle \equiv \scr L(\chi, \mu)\otimes \scr L(\chi, 0)^{-1} \mod (\Phi_G^d)^*(\Pic(\Mg)),
 $$
which gives that $ [\mathscr L(\chi, 0)]=[\mathscr L(\chi, \zeta)]$ for any $\zeta\in \bbZ^n$. 

Since $\RPic\left(\bg{T}^d\right)$ is generated by tautological line bundles by Theorem \ref{BunT}\eqref{BunT2}, we deduce that the map $\eta$ is surjective and hence it induces an isomorphism 
\begin{equation}\label{E:map-tau2}
\begin{aligned}
\ov\eta: \frac{\Lambda^*(T)}{\ker(\eta)}& \xrightarrow{\cong} \coker(\tau_T+\sigma_T) \\
[\chi] & \mapsto [\mathscr L(\chi, \zeta)],
\end{aligned}
\quad \text{ for any } \zeta\in \bbZ^n.
\end{equation}

We now claim that 
\begin{equation}\label{E:ker-tau}
\ker(\eta)=
\begin{cases}
\{0\} & \text{ if } g\geq 2,\\
2\Lambda^*(T) & \text{ if } g=1.
\end{cases}
\end{equation}
Indeed, Theorem \ref{BunT}\eqref{BunT2} implies easily that $\ker(\eta)$ is trivial if $g\geq 2$. On the other hand, if $g=1$ then the first relation in \eqref{E:tau-dual}
gives that (using that $\omega_{\pi}$ is trivial) 
$$
\mathscr L(\chi, 0)^2=d_{\pi}(\mt L_{\chi})^2\equiv \langle (\chi, 0),(\chi, 0)\rangle=\langle \mt L_{\chi}, \mt L_{\chi}\rangle  \mod \Phi_G^*(\Pic(\Mg)),
$$
which implies that $2\Lambda^*(T)\subseteq \ker(\eta)$. And, moreover, equality holds by Theorem \ref{BunT}\eqref{BunT2}.

By putting together \eqref{E:seq-Del}, \eqref{E:map-tau2} and \eqref{E:ker-tau}, we get the exact sequences \eqref{seq-bunT} and \eqref{seq-bunTg1}, where the map $\rho_T$ is the composition of the surjection $\RPic\left(\bg{T}^d\right)\twoheadrightarrow \coker(\tau_T+\sigma_T)$ with the inverse of the isomorphism of \eqref{E:map-tau2}.

Finally, the functoriality of the exact sequences  \eqref{seq-bunT} and \eqref{seq-bunTg1} follows from the functoriality properties of the Deligne pairing and of the determinant of cohomology. 
\end{proof}

\subsection{The  restriction of the Picard group to the fibers II}\label{Sec:gerbe-fib2}

In this subsection, we complete the results of Subsection \ref{Sec:gerbe-fib} on the  restriction map $\res_T^d(C):\RPic(\bg{T}^d)\to \Pic(\mt J_T^d(C))$ to the  fiber over a geometric point $(C,p_1\ldots, p_n)$ of $\Mg$.

With this aim, we now study for $g\geq 1$ the image and the kernel of the direct sum of the weight function $w_T^d$ and of the following map
\begin{equation}\label{E:rhoTd}
\begin{aligned}
\gamma_T^d:\RPic(\bg{T}^d) & \longrightarrow \Bil^s\Lambda(T),\\
\scr L(\chi, \zeta)) & \mapsto \chi\otimes \chi, \\
\langle (\chi, \zeta), (\chi', \zeta') \rangle & \mapsto \chi\otimes \chi'+\chi'\otimes \chi, 
\end{aligned}
\end{equation}
which is a well-defined homomorphism by Proposition \ref{P:restr}, since the composition of $\gamma_T^d$ with the injective homomorphism $\id_{J_C}:\bbZ\to \End(J_C)$ (for $g\geq 1$) gives the homomorphism $\wt\gamma_T(C):\RPic(\bg{T}^d)\to \NS(J_T^d(C))$ over any geometric point $(C,p_1\ldots, p_n)$ of $\Mg$. 

\begin{prop}\label{P:wT+rhoT}
Assume that $g\geq 1$. Consider the following group
$$H_{g,n}:=
\begin{cases} 
\{(m,\zeta)\in \bbZ\oplus \bbZ^n\: : \: (2g-2)m+|\zeta|=0\} & \text{ if } g\geq 2,\\
\{\zeta\in \bbZ^n\: : \: |\zeta|=0\} & \text{ if } g=1.\\
\end{cases}
$$
There is an exact sequence 
\begin{equation}\label{E:wT+rhoT1}
0 \to \Lambda^*(T)\otimes H_{g,n}  \xrightarrow{j_T^d} \RPic(\bg{T}^d)\xrightarrow{w_T^d\oplus \gamma_T^d} \Lambda^*(T)\oplus   \Bil^s\Lambda(T), 
\end{equation}
where the morphism $j_T^d$ is defined as 
\begin{equation*}
\begin{aligned}
& j_T^d(\chi\otimes (m,\zeta))=  \langle \mt L_{\chi}, \omega_{\pi}^m(\sum_{i=1}^n \zeta_i \sigma_i) \rangle &   \text{ if } g\geq 2,\\ 
& j_T^d(\chi\otimes \zeta) = \langle \mt L_{\chi}, \cO(\sum_{i=1}^n \zeta_i \sigma_i) \rangle & \text{ if } g=1.\\
\end{aligned}
\end{equation*}
Moreover, the image of $w_T^d\oplus \gamma_T^d$ is equal to 
\begin{equation}\label{E:wT+rhoT2}
 \Im(w_T^d\oplus \gamma_T^d)=
 \begin{cases} 
 \Lambda^*(T)\oplus   \Bil^s\Lambda(T) & \text{ if } n\geq 1,\\
  \left\{(\chi, b)\in  \Lambda^*(T)\oplus   \Bil^s\Lambda(T):  
 \begin{aligned}
 & (2g-2)\vert (x,\chi)-b(d,x)+(g-1)b(x,x)\\
& \text{ for any } x\in \Lambda(T)
 \end{aligned}
 \right\} & \text{ if } n=0. 
\end{cases} 
 \end{equation}
\end{prop}
Note that $H_{g,n}=0$ (or, equivalently, the map $w_T^d\oplus \gamma_T^d$ is injective) if and only if either $n=0$ or $g=n=1$.
\begin{proof}
Let us first compute the kernel of $w_T^d\oplus \gamma_T^d$. 
Using Theorem \ref{BunT}\eqref{BunT2} (and the notation of loc. cit.), an element of $\RPic(\bg{T}^d)$ can be written uniquely as
\begin{equation}\label{E:elem-M}
\cM=\sum_{1\leq i \leq j \leq r} a_{ij}\langle (e_i,0),(e_j,0)\rangle +\sum_{1\leq k \leq r} \langle (e_k,0),(0,\zeta^k)\rangle + \sum_{1\leq l \leq r} b_l \mathscr L(e_l,0),
\end{equation}
for some $a_{ij}, b_l \in \bbZ$,  $\zeta^k\in \bbZ^n$, with the property that $a_{ii}=0$ if $g=1$. Using the definition \eqref{E:rhoTd}  of $\gamma_T^d$   and the formula for $w_T^d$ contained in Proposition \ref{P:restr}\eqref{P:restr1}, we compute 
\begin{equation}\label{E:comp-M}
\begin{aligned}
& w_T^d(\cM)= \sum_{1\leq i \leq j \leq r} a_{ij} (d_ie_j+d_je_i)+\sum_{1\leq k \leq r}  |\zeta^k|e_k  +\sum_{1\leq l \leq r} b_l (d_l+1-g) e_l,\\
&\gamma_T^d(\cM)=\sum_{1\leq i \leq j \leq r} a_{ij}(e_i\otimes e_j+e_j\otimes e_i) + \sum_{1\leq l \leq r} b_l e_l\otimes e_l,
\end{aligned}
\end{equation}
where $d=(d_1,\ldots, d_r)$ under the isomorphism $\Lambda(T)\cong \bbZ^r.$  From the above formulas, it follows that 
$$\cM\in \ker(w_T^d\oplus \gamma_T^d) \Leftrightarrow 
\begin{sis} 
& a_{ij}=0 \: \text{ for  } i<j,\\ 
& 2a_{ii}+b_i=0,  \\
& 2a_{ii}d_i+\vert \zeta^i\vert +b_i(d_i+1-g)=0,\\
\end{sis}
\Leftrightarrow 
\begin{sis} 
& a_{ij}=0 \: \text{ for  } i<j,\\ 
& b_i=-2a_{ii}, \\ 
& (a_{ii}, \zeta^i)\in H_{g,n}.
\end{sis}
$$ 
In other words, $\cM$ belongs to the kernel of $w_T^d\oplus \gamma_T^d$ if and only if $\cM$ has the following form
$$
\cM=\sum_{\substack{1\leq i \leq r\\ (a_{ii}, \zeta^i)\in H_{g,n}}} \left[a_{ii}\langle \mt L_{e_i}, \mt L_{e_i}\rangle-2a_{ii} d_{\pi}(\mt L_{e_i}) +\left\langle \mt L_{e_i}, \cO\left(\sum_k (\zeta^i)_k\sigma_k\right)\right\rangle \right]= $$
$$=\sum_{\substack{1\leq i \leq r\\ (a_{ii}, \zeta^i)\in H_{g,n}}} \left\langle \mt L_{e_i}, \omega_{\pi}^{a_{ii}}\left(\sum_k (\zeta^i)_k\sigma_k\right)\right\rangle,
$$
where the second equality follows from the first formula in \eqref{E:tau-dual}. This shows that $j_T^d$ is an injective homomorphism whose image is equal to the kernel of $w_T^d\oplus \gamma_T^d$.  

Now,  it is straightforward from \eqref{E:comp-M} that $w_T^d\oplus \gamma_T^d$ is surjective if $n\geq 1$. It remains to determine 
 the image of $w_T^d\oplus \gamma_T^d$ for $n=0$. Let $\{\epsilon_1,\ldots, \epsilon_r\}$ be the basis of $\Lambda(T)$ dual to the canonical  basis $\{e_1,\ldots, e_r\}$ of $\Lambda^*(T)\cong \bbZ^r$. From \eqref{E:comp-M}, it follows that  an element $(\chi, b)\in \Lambda^*(T)\oplus   \Bil^s\Lambda(T)$ is equal to the image via $w_T^d\oplus \gamma_T^d$ of  an element $\cM\in \RPic(\bg{T}^d)$ as in \eqref{E:elem-M} if and only if 
\begin{equation}\label{E:cond-im1}
\begin{aligned}
& b(\epsilon_i,\epsilon_j)=a_{ij} \quad \text{ for } i<j,\\
& b(\epsilon_i,\epsilon_i)=2a_{ii}+c_i,\\
& \chi(\epsilon_i)=\sum_{i\leq j} d_j a_{ij} +\sum_{j\leq i} d_j a_{ji} +c_i(d_i+1-g).
\end{aligned}
\end{equation}
Using the first two conditions in \eqref{E:cond-im1}, we can rewrite the third condition as 
$$
\chi(\epsilon_i)=\sum_{i\leq j} d_j a_{ij} +\sum_{j\leq i} d_j a_{ji} +c_i(d_i+1-g)=\sum_{j\neq i} d_jb(\epsilon_i,\epsilon_j) +(2a_{ii}+c_i)d_i+c_i(1-g)=
$$
\begin{equation}\label{E:cond-im2}
=\sum_{j} d_jb(\epsilon_i,\epsilon_j)+(b(\epsilon_i,\epsilon_i)-2a_{ii})(1-g)=b(\epsilon_i,d)-(g-1)b(\epsilon_i,\epsilon_i)+(2g-2)a_{ii}.
\end{equation}
 Therefore, we deduce that the element $(\chi, b)$ belongs to the image of $w_T^d\oplus \gamma_T^d$ if and only if 
 \begin{equation}\label{E:cond-im3}
 (2g-2)\vert \chi(\epsilon_i)-b(d,\epsilon_i)+(g-1)b(\epsilon_i,\epsilon_i)\quad \text{ for any } 1\leq i \leq r.
 \end{equation}
 Now if $x=\sum_{i=1}^r \lambda_i \epsilon_i$ we have that 
$$ \chi(x)-b(d,x)+(g-1)b(x,x)=\sum_{i=1}^r\lambda_i \left[ \chi(\epsilon_i)-b(d,\epsilon_i)\right]+\sum_{i=1}^r (g-1)\lambda_i^2 b(\epsilon_i,\epsilon_i) +
2(g-1)\sum_{i<j} \lambda_i\lambda_j b(\epsilon_i,\epsilon_j)$$
 $$\equiv \sum_{i=1}^r\lambda_i \left[ \chi(\epsilon_i)-b(d,\epsilon_i)(g-1) b(\epsilon_i,\epsilon_i) \right] \mod 2g-2.$$
 Hence, condition \eqref{E:cond-im3} is equivalent to the condition 
  \begin{equation*}
 (2g-2)\vert \chi(x)-b(d,x)+(g-1)b(x,x)\quad \text{ for any } x\in \Lambda(T),
 \end{equation*}
 appearing in the statement. 
\end{proof}

\begin{rmk}
For $n=0$, the invariant factors of $\Im(w_T^d\oplus \gamma_T^d)\subseteq \Lambda^*(T)\oplus \Bil^s\Lambda(T)$ are 
$$\left(\underbrace{1, \ldots, 1}_{\binom{\dim T+1}{2}}, \underbrace{2g-2,\ldots, 2g-2}_{\dim T}\right),
$$ 
as can be shown easily  using \eqref{E:cond-im3}.
\end{rmk}

We can finally describe the full restriction morphism $\res_T^d(C):\Pic(\bg{T}^d) \to \Pic(\mt J_T^d(C))$ for any geometric point $(C,p_1,\ldots, p_n)$ of $\Mg$ and $g\geq 1$ (for $g=0$ we have that $\res_T^d(C)=w^d_T$, with the identification $w_T^d(C):\Pic(\mt J^d_T(C))\xrightarrow{\cong} \Lambda^*(T)$, see Proposition \ref{P:Pic-JTC}). We will use the notation of Propositions \ref{P:Pic-JTC} and \ref{P:wT+rhoT}.

\begin{prop}\label{P:restrPic}
Assume that $g\geq 1$ and let $(C,p_1,\ldots, p_n)$ be a geometric point of $\Mg$ defined over an algebraically closed field $K$. Then the restriction morphism $\res_T^d(C):\Pic(\bg{T}^d) \to \Pic(\mt J_T^d(C))$ fits into the following commutative diagram with exact rows 
\begin{equation}\label{E:restrPic}
\xymatrix{
0 \ar[r] &  \Lambda^*(T)\otimes H_{g,n} \ar[r]^{j_T^d} \ar[d]^{\id\otimes \iota_C} & \RPic(\bg{T}^d)\ar[r]^{w_T^d\oplus \gamma_T^d} \ar[d]^{\res_T^d(C)}& \Lambda^*(T)\oplus   \Bil^s\Lambda(T) \ar@{^{(}->}[d]^{\id\oplus (-\otimes \id_{J_C})} & \\
0 \ar[r] & \Lambda^*(T)\otimes J_C(K) \ar[r]^{j_T^d(C)} &\Pic(\mt J_T^d(C)) \ar[r]^(0.3){w_T^d(C) \oplus \gamma_T^d(C)}   & \Lambda^*(T) \oplus \Hom^s(\Lambda(T)\otimes \Lambda(T),  \End(J_C))  \ar[r] & 0\\
}
\end{equation}
where $-\otimes \id_{J_C}$ is the inclusion \eqref{E:inc-bs} and the left vertical morphism is induced by the  morphism 
\begin{equation*}
\begin{aligned}
\iota_C=\iota_{(C,p_1,\ldots, p_n)}: H_{g,n} & \rightarrow J_C(K) \\
 (m,\zeta) & \mapsto \omega_{C}^m(\sum_{i=1}^n \zeta_i p_i) &   \text{ if } g\geq 2,\\ 
\zeta & \mapsto   \cO_C(\sum_{i=1}^n \zeta_i p_i)  & \text{ if } g=1.\\
\end{aligned}
\end{equation*}
In particular, the kernel of the restriction morphism $\res_T^d(C)$ is equal to $\Lambda^*(T)\otimes \ker(\iota_{C})$.  
\end{prop}
\begin{proof}
The top horizontal row is exact by Proposition \ref{P:wT+rhoT} while the  bottom left horizontal arrow is injective by Proposition \ref{P:Pic-JTC}. It remains therefore to prove the commutativity of the diagram. The right square is commutative by Proposition \ref{P:restr}. The commutativity of the left square follows straightforwardly by comparing the definition of $j_T^d$ (see Proposition \ref{P:wT+rhoT}) with the definition \eqref{E:jTC} of $j_T^d(C)$. 
\end{proof}

\section{Reductive non-abelian case}\label{Sec:non-ab-rd}

The aim of the section is to describe the relative Picard group of $\bg{G}^{[d]}$ 
\begin{equation}\label{E:RelPicG}
\RPic(\bg{G}^{[d]}):=\Pic(\bg{G}^{[d]})/(\Phi_G^{[d]})^*(\Pic(\Mg)),
\end{equation}
for a reductive group $G$ and any $[d]\in \pi_1(G)$. 
Throughout this section, we fix a maximal torus $\iota:T_G\hookrightarrow G$ and we denote by $\scr W_G$ the Weyl group of $G$ that we identify with $\scr N(T_G)/T_G$. 
We will denote by $B_G$ the unique Borel subgroup of $G$ containing $T_G$.

As for the reductive abelian case in \S\ref{Sec:tori}, we need to distinguish the genus zero and non-zero cases. Below are the main theorems of this section.

\begin{teo}\label{T:BunG}
Assume $g>0$. Let $d\in \Lambda(T_G)$. We denote by
 $[d]$ its image in $\pi_1(G)=\Lambda(T_G)/\Lambda(T_{G^{\sc}})$ and by $[d]^{\ab}:=[d^{\ab}]:=[\Lambda_\ab(d)]$ its image in $\pi_1(G^{\ab})=\Lambda(G^\ab)$.
\begin{enumerate}
\item \label{BunG-1} There exists a unique homomorphism  (called \emph{transgression map})
\begin{equation}\label{E:trasgr}
\tau_G(=\tau_{G,g,n}):(\Sym^2\Lambda^*(T_G))^{\mathscr W_G}\longrightarrow \RPic\Big(\bg{G}^{[d]}\Big),
\end{equation}
such that the composition $\iota_\#^*\circ \tau_G$ is equal to the $\scr W_G$-invariant part of the homomorphism $\tau_{T_G}: \Sym^2 \Lambda^*(T_G)\to \RPic(\bg{T_G}^d)$ defined in Theorem \ref{BunT}\eqref{BunT1}. 

In particular, $\tau_G$ is injective. 
\item \label{BunG-2} The commutative diagram of abelian groups
\begin{equation}\label{E:amalg}
\xymatrix{
\Sym^2 \Lambda^*(G^{\ab})\ar@{^{(}->}[rr]^{\Sym^2 \Lambda^*_\ab}\ar@{^{(}->}[d]^{\tau_{G^\ab}}&& (\Sym^2\Lambda^*(T_G))^{\mathscr W_G}\ar@{^{(}->}[d]^{\tau_G}\\
\RPic\left(\bg{G^{\ab}}^{[d]^{\ab}}\right)\ar@{^{(}->}[rr]^{\ab^*} &&\RPic\Big(\bg{G}^{[d]}\Big)
}
\end{equation}
is formed by injective morphisms and it is a push-out, i.e., 
$$\RPic\left(\bg{G^{\ab}}^{[d]^{\ab}}\right)\amalg_{\Sym^2\Lambda^*(G^{\ab})}(\Sym^2\Lambda^*(T_G))^{\mathscr W_G} \xrightarrow[\cong]{\ab^*\coprod \tau_G} \RPic\left(\bg{G}^{d}\right).$$
\end{enumerate}
Furthermore, the homomorphism \eqref{E:trasgr} and the diagram \eqref{E:amalg} are functorial for all the homomorphisms of reductive groups $\phi:H\to G$ such that $\phi(T_H)\subseteq T_G$.
\end{teo}

The case $g=0$ is completely different. In particular,  the description of the Picard group may vary depending on the connected component of $\mathrm{Bun}_{G,0,n}$. 

\begin{teo}\label{T:BunGg0} Assume $g=0$. Let $d\in \Lambda(T_G)$ and set $d^{\ss}:=\Lambda_{\ss}(d)\in \Lambda(G^{\ss})\subseteq \Lambda(G^{\ad})$. 
\begin{enumerate}
\item \label{BunGg0-1} If $d^{\ss}$  satisfies condition (*) of Lemma \ref{L:inj-cont}\eqref{L:inj-cont1}, then the homomorphism  
		\begin{equation}\label{seq-bunGg0}
		w_G^d:\RPic\left({\mathrm{Bun}}_{G,0,n}^{[d]}\right)\xrightarrow{\iota_\#^*}\RPic\left(\mathrm{Bun}_{T_G,0,n}^d\right)\xrightarrow{w_{T_G}^d} \Lambda^*(T_G)
		\end{equation}
is injective. Moreover,  given $[d]\in \pi_1(G)$, it is always possible to choose a representative $d\in \Lambda(T_G)$ such that $d^{\ss}$ satisfies condition (*) of Lemma \ref{L:inj-cont}\eqref{L:inj-cont1}. 

Furthermore, the homomorphism $w_G^d$ is functorial   for all the homomorphisms of reductive groups $\phi:H\to G$ such that $\phi(T_H)\subseteq T_G$.
\item \label{BunGg0-2}  Let $\Omega_d^*(T_G)\subset \Lambda^*(T_G)$ be the subgroup of characters whose image in  $\Lambda^*(T_{G^{\sc}})$ is equal to $b(d^\ss,-)$ for some element $b\in(\Sym^2\Lambda^*(T_{G^{\sc}}))^{\scr W_G}$. The image of $w_G^d$ is equal to 
		  $$
		  \Im(w_G^d)=
		  \begin{cases}
		  \Omega_d^*(T_G)&\text{if }n\geq 1,\\
		  \{\chi\in  \Omega_d^*(T_G)\: :  (\chi,d)\in 2\mathbb Z\}& \text{if }n=0.
		  \end{cases}
		  $$
\end{enumerate}
\end{teo}

The theorem has been proved in \cite{BH10} under the assumption $n=3$ (in which case $\mathrm{Bun}_{G,0,3}=\mathrm{Bun}_{G}(\bbP^1_k)$ since $\cC_{0,3}=\bbP^1_k$), and this result  is a fundamental ingredient of our proof.

\subsection{The pull-back to the maximal torus bundles}\label{SS:inj}

The aim of this subsection is to study the injectivity  of the pull-back of the (relative) Picard groups  along the morphism 
$$\iota_\#:\bg{T_G}^d\to \bg{G}^{[d]}.$$
The results of this subsection are also true for the pull-back of the (relative) Picard groups  along the morphism 
$$\iota_\#(C/S):\mathrm{Bun}^{d}_{T_G}(C/S)\hookrightarrow \mathrm{Bun}_G^{[d]}(C/S),$$
for any family $C\to S$ of curves, when $S$ is an integral and regular quotient stack over $k$. 


\begin{prop}\label{P:seesaw}
Assume that $(g,n)\neq (1,0)$. The restriction of $\Phi_G^{[d]}:\bg{G}^{[d]}\to\Mg$ over the geometric generic point $\ov\eta$ of $\Mg$ gives rise to an injective homomorphism 
$$
\RPic(\bg{G}^{[d]})\longrightarrow \Pic(\mathrm{Bun}_G^{[d]}(\cC_{\ov \eta})),
$$
where $\cC_{\ov\eta}$ is the curve corresponding to $\ov\eta$. 

The same holds true for the relative moduli stack $\mathrm{Bun}_G^{[d]}(C/S)$ for any family of curves $C\to S$, provided that $S$ is an integral and regular quotient stack over $k$.
\end{prop}
\begin{proof}
Note that $\Mg$ is an integral and regular algebraic stack over the base field $k=\ov k$ and it is a quotient stack over $k$ if $(g,n)\neq (1,0)$ (see e.g., Proposition \ref{P:qstack-cover} with $G$ equal to the trivial group). 
Then, up to replacing it with an equivariant approximation as in Proposition \ref{Equi}, we can assume that $\Mg$ is (generically) an integral  regular $k$-scheme for any $g$ and $n$ and we denote its generic point by $\eta$. 

Now, if $G$ is a torus then  $\bg{G}^{[d]}$ is of finite type over $\Mg$ by Proposition \ref{P:qc=a}. If $G$ is not a torus,  then consider the open substack $\bg{G}^{[d],\leq 2}\subset \bg{G}^{[d]}$ of the instability exhaustion (as in \S \ref{SS:open-ss}), which is of finite type over $\Mg$ by Proposition \ref{P:inst-cover}\eqref{P:inst-cover2}. Since the complement of $\bg{G}^{[d],\leq 2}$ in $\bg{G}^{[d]}$ has codimension at least two on each geometric fiber of $\Phi_G^{[d]}:\bg{G}^{[d]}\to \Mg$ and $\bg{G}^{[d]}\to \Mg$ is smooth  by Theorem \ref{beh}\eqref{beh1},  Lemma \ref{restr} implies that the restriction maps induce isomorphisms 
$$\RPic(\bg{G}^{[d]})\xrightarrow{\cong} \RPic(\bg{G}^{[d], \leq 2}) \quad \text{ and } \quad \Pic(\mathrm{Bun}_G^{[d]}(\cC_{\ov \eta}))\xrightarrow{\cong} \Pic(\mathrm{Bun}_G^{[d], \leq 2}(\cC_{\ov \eta})).
$$
Hence, up to replacing $\bg{G}^{[d]}$ with $\bg{G}^{[d], \leq 2}$ if $G$ is not a torus, we can assume that $\Phi_G^{[d]}:\bg{G}^{[d]}\to \Mg$ is of finite type. 

Note that  $\Phi_G^{[d]}$ is flat with integral fibers since it is smooth by Theorem \ref{beh}\eqref{beh1} and with (geometrically) connected fibers by Theorem \ref{concomp}. Hence, we can apply  Proposition \ref{P:fiber} to the morphism $\Phi_G^{[d]}:\bg{G}^{[d]}\to \Mg$ and deduce that the restriction  to the generic fiber of $\Phi_G^{[d]}$ induces an isomorphism 
\begin{equation}\label{E:injec1}
\RPic(\bg{G}^{[d]})\cong \Pic(\mathrm{Bun}_G^{[d]}(\cC_{\eta})).
\end{equation}
Consider now the geometric generic point $\ov\eta=\Spec \ov{k(\eta)}$. Since $\mathrm{Bun}_G^{[d]}(\cC_{\eta})\to \eta$ is Stein (as follows from Proposition \ref{P:SteinG} by base change along $\eta\to \Mg$)  and $\ov\eta\to \eta$ is fpqc, then Lemma \ref{base-change} implies that we have an injective base change morphism 
\begin{equation}\label{E:injec2}
\Pic(\mathrm{Bun}_G^{[d]}(\cC_{\eta})) \hookrightarrow \Pic(\mathrm{Bun}_G^{[d]}(\cC_{\ov \eta})).
\end{equation}
We conclude by putting together \eqref{E:injec1} and \eqref{E:injec2}. 

The proof for  $\mathrm{Bun}_G^{[d]}(C/S)$ is the same, using the given assumptions on $S$. 
\end{proof}

\begin{cor}\label{C:inj}
\noindent 
\begin{enumerate}[(i)]
\item \label{C:inj1} If $g\geq 1$ then for any $d\in \Lambda(T_G)$ the pull-back map  
$$\iota_\#^*:\RPic(\mathrm{Bun}_{G,g,n}^{[d]})\longrightarrow \RPic(\mathrm{Bun}_{T_G,g,n}^{d})$$ 
is injective. 
\item \label{C:inj2} If $g=0$ and  $d\in \Lambda(T_G)$ is such that $d^{\ss}:=\Lambda_{\ss}(d)\in \Lambda(G^{\ss})\subseteq \Lambda(G^{\ad})$ satisfies condition (*) of Lemma \ref{L:inj-cont}\eqref{L:inj-cont1} then  
the pull-back map
$$\iota_\#^*:\RPic(\mathrm{Bun}_{G,g,0}^{[d]})\longrightarrow \RPic(\mathrm{Bun}_{T_G,g,0}^{d})$$ 
is injective. 
\end{enumerate}
The same holds true for the relative moduli stack $\mathrm{Bun}_G^{[d]}(C/S)$ for any family of curves $C\to S$, provided that $S$ is an integral and regular quotient stack over $k$.
\end{cor}
Note that Lemma \ref{L:inj-cont}\eqref{L:inj-cont2} guarantees that any $\delta\in \pi_1(G)$ admits a representative $d\in \Lambda(T_G)$, i.e., $[d]=\delta\in \pi_1(G)$, satisfying condition (*).
\begin{proof}
We will give the proof for $\bg{G}^{[d]}$ by distinguishing the case $(g,n)\neq (1,0)$ from the case $(g,n)=(1,0)$;  the proof for $\mathrm{Bun}_G^{[d]}(C/S)$ follows the same argument of the first case. 

\noindent\fbox{Case I: $(g,n)\neq (1,0)$.} 
For a given $d\in \Lambda(T_G)=\pi_1(T_G)$, consider the following commutative diagram
\begin{equation}\label{E:res-geofb}
\xymatrix{
\RPic(\bg{G}^{[d]}) \ar[d]^{\iota_\#^*} \ar[r]& \Pic(\mathrm{Bun}_G^{[d]}(\cC_{\ov \eta})) \ar[d]^{\iota_\#(\cC_{\ov \eta})^*}\\
\RPic(\bg{T_G}^{d}) \ar[r]& \Pic(\mathrm{Bun}_{T_G}^{d}(\cC_{\ov \eta})) \\
}
\end{equation}
where the horizontal arrows are induced by the restriction to geometric generic fibers over $\Mg$ and the vertical arrows are the pull-backs induced by the maps  $\iota_\#:\bg{T_G}^d\to \bg{G}^{[d]}$
and $\iota_\#(\cC_{\ov \eta}):  \mathrm{Bun}_{T_G}^{d}(\cC_{\ov \eta})\to \mathrm{Bun}_G^{[d]}(\cC_{\ov \eta})$. 
The top horizontal arrow (and also the bottom horizontal one) is injective by Proposition \ref{P:seesaw}. Therefore, the injectivity of $\iota_\#^*$ will follow from the above diagram \eqref{E:res-geofb} if we show that the map $\iota_\#(\cC_{\ov \eta})^*$ is injective. 

According to \cite[Thm. 5.3.1(iv)]{BH10}, the map $\iota_\#(\cC_{\ov \eta})^*$ fits into the following commutative diagram of abelian groups with exact arrows
\begin{equation}\label{E:back-tor}
\xymatrix{
0 \ar[r]& \un{Hom}(\pi_1(G),J_{\cC_{\ov \eta}})(\ov k(\eta)) \ar[r]\ar[d]^{\pi_1(\iota)^*} & \Pic(\mathrm{ Bun}_G^{[d]}(\cC_{\ov \eta})) \ar[r]  \ar[d]^{\iota_\#(\cC_{\ov \eta})^*} &  \NS(\mathrm{ Bun}_G^{[d]}(\cC_{\ov \eta})) 
\ar[d]^{(\iota_G)^{\NS, d}} \ar[r]& 0 \\
0 \ar[r]& \un{Hom}(\pi_1(T_G),J_{\cC_{\ov \eta}})(\ov k(\eta)) \ar[r]& \Pic(\mathrm{ Bun}_{T_G}^{d}(\cC_{\ov \eta})) \ar[r] &  \NS(\mathrm{ Bun}_{T_G}^{d}(\cC_{\ov \eta})) \ar[r]&0 \\
}
\end{equation}
where the map $\pi_1(\iota)^*$ is induced by the natural homomorphism $\pi_1(\iota):\pi_1(T_G)\to \pi_1(G)$ and  the map $(\iota_G)^{\NS, d}$ is  defined in  \cite[Def. 5.2.5]{BH10}.
Now, the map $\pi_1(\iota)^*$ is injective since the homomorphism $\pi_1(\iota):\pi_1(T_G)=\Lambda(T_G)\to \pi_1(G)=\frac{\Lambda(T_G)}{\Lambda(T_{G^{\sc}})}$ is surjective. On the other hand, 
from \cite[Def. 5.2.5]{BH10} it follows that:
\begin{itemize}
\item if $g\geq 1$ then $(\iota_G)^{\NS, d}$ is injective for any $d\in \Lambda(T_G)$, using that natural homomorphism $\bbZ\to \End(J_{\cC_{\ov \eta}})$ is injective; 
\item if $g=0$ then $(\iota_G)^{\NS, d}$ is injective precisely when $d^{\ss}$ satisfies condition (*) of Lemma \ref{L:inj-cont}\eqref{L:inj-cont1} (see also the proof of \cite[Lemma 4.3.6]{BH10}).
\end{itemize}

\noindent\fbox{Case II: $(g,n)=(1,0)$.} 
Consider the cartesian diagram of families of genus one curves
\begin{equation}\label{E:diag10}
\xymatrix{
\mathrm{Bun}_{T_G,1,1}^d\ar[r]^{\iota_{\#,1,1}}\ar[d]_{F_{T_G}} \ar@{}[dr]|\square& \mathrm{Bun}_{G,1,1}^{[d]} \ar[rr]^{\Phi_{G,1,1}}\ar[d]^{F_G}  \ar@{}[drr]|\square &&\cM_{1,1} \ar[d]^{F} \\
\mathrm{Bun}_{T_G,1,0}^d\ar[r]_{\iota_{\#,1,0}}& \mathrm{Bun}_{G,1,0}^{[d]} \ar[rr]_{\Phi_{G,1,0}}  &&\cM_{1,0}\\
}
\end{equation}
By pull-back we obtain the following commutative diagram of relative Picard groups
\begin{equation}\label{E:diagPic10}
\xymatrix{
\RPic(\mathrm{Bun}_{T_G,1,1}^d) & \RPic(\mathrm{Bun}_{G,1,1}^{[d]})  \ar@{_{(}->}[l]_{\iota_{\#,1,1}^*}    \\
\RPic(\mathrm{Bun}_{T_G,1,0}^d)   \ar@{^{(}->}[u]^{F_{T_G}^*} & \RPic(\mathrm{Bun}_{G,1,0}^{[d]}) \ar@{_{(}->}[u]_{F_G^*} \ar[l]^{\iota_{\#,1,0}^*} \\
}
\end{equation}
where $F_G^*$ (resp., $F_{T_G}^*$) is injective by Lemma \ref{base-change}, which can be applied since  $\Phi_{G,1,0}$ (resp., $\Phi_{G,1,0}\circ \iota_{\#,1,0}=\Phi_{T_G,1,0}$) is  Stein by Proposition \ref{P:SteinG} and $F$ is fpqc (being a family of curves), and $\iota_{\#,1,1}^*$ is injective by Case I. From \eqref{E:diagPic10} we get that $\iota_{\#,1,0}^*$ is injective, and we are done. 
\end{proof}

\subsection{The transgression map}\label{SS:trasgr}

In this subsection, we will construct the transgression map  \eqref{E:trasgr} and prove Theorem \ref{T:BunG}\eqref{BunG-1}. 

We first show that the map \eqref{E:trasgr} exists for the rational  first  operational Chow group of the moduli stack $\bg{G}^{[d],\leq m}$ of $G$-bundles with instability degree less than or equal to $m$ (see \S\ref{SS:open-ss}).

\begin{prop}\label{P:Q-transfr} 
For any $[d]\in \pi_1(G)$ and for any $m\geq 0$, there exists a unique homomorphism of groups
\begin{equation}\label{E:trasgrS}
c_1(\tau_G)_{\bbQ}^{\leq m}:(\Sym^2\Lambda^*(T_G))^{\scr W_G} \to A^1\left(\bg{G}^{[d],\leq m}\right)_{\mathbb Q},
\end{equation}
such that, for any object $(C\xrightarrow{\pi} S, \un{\sigma}, E)$ and $\sum_i \chi_i\cdot \mu_i \in (\Sym^2\Lambda^*(T_G))^{\scr W_G}$, it satisfies the following properties.
\begin{enumerate}[(i)]
\item\label{E:dG-cG}We have the equality 
\begin{equation*}
\left(c_1(\tau_G)_{\bbQ}^{\leq m}(\sum_i \chi_i\cdot \mu_i)\right)(C\xrightarrow{\pi} S, \un{\sigma}, E)=\sum_i\frac{1}{|\scr W_G|}(\pi\circ p)_*\left(\cc{G}{B}{E}\left(\chi_i\cdot\mu_i\cdot\prod_{\alpha>0}\alpha\right)\right)\in A^1(S)_{\mathbb Q}.
\end{equation*}
\item\label{E:trasg-red} If $(C\xrightarrow{\pi} S, \un{\sigma}, E)$ admits a reduction to a $T$-bundle $Q$, we have the equality
\begin{equation*}
\begin{split}
\left(c_1(\tau_G)_{\bbQ}^{\leq m}(\sum_i \chi_i\cdot \mu_i)\right)(C\xrightarrow{\pi} S, \un{\sigma}, E)
&= \sum_i c_1\left(\bigotimes_i\langle \chi_{i \sharp}(Q), \mu_{i \sharp}(Q)\rangle_{\pi}\right)=\\&=\sum_i \pi_*\Big(c_1(\chi_{i \sharp}(Q))\cdot c_1(\mu_{i \sharp}(Q))\Big) \in A^1(S)_{\mathbb Q}.
\end{split}
\end{equation*}
\end{enumerate}
\end{prop}

\begin{proof}
We remove the $n$-sections $\un{\sigma}$  from the notation. For any family $(C\xrightarrow{\pi} S, E)$ of $G$-bundles, let $E/B\xrightarrow{p} C\xrightarrow{\pi} S$ be the associated flag bundle. Following the notations of $\S$\ref{SS:Ch-Flag}, we define
\begin{equation}\label{E:def-trasg}
\begin{array}{lcll}
c_1(\tau_G)_{\bbQ}^{\leq m}(C\xrightarrow{\pi} S, E):& (\Sym^2\Lambda^*(T))^{\scr W_G}&\to& A^1(S)_{\mathbb Q}\\
& v &\mapsto & \frac{1}{|\scr W_G|}(\pi\circ p)_*\left(\cc{G}{B}{E}\left(v\cdot\prod_{\alpha>0}\alpha\right)\right).
\end{array}
\end{equation}
It is compatible with base change and, so, it defines a homomorphism to the first rational operational Chow group of the moduli stack $\bg{G}^{d,\leq m}$, satisfying the equality (\ref{E:dG-cG}). We need to check the equality (\ref{E:trasg-red}). Assume that $(C\to S, E)$ has a reduction to a $T$-bundle $Q$. A reduction to $T$ is equivalent to the existence of a section $\sigma:C\to E/T$ for the natural morphism $q:P/T\to C$. In other words, we have a commutative diagram as follows
\begin{equation}
\xymatrix{
Q\ar@{}[dr]|\square\ar[d]\ar[r] &E\ar@{}[dr]|\square\ar[r]\ar[d] &\Spec(k)\ar[d]\\
C\ar@{=}[rd]\ar[r]^\sigma &E/T\ar@{}[dr]|\square\ar[r]\ar[d]^q &\mathcal BT\ar[d]\\
 & C\ar[r]& \mt BG
}
\end{equation}
where all the squares are cartesian. Using the notations of \S\ref{SS:Ch-Flag}, we have the following equalities 
$$
c_1(\chi_{\sharp}(Q))\cdot c_1(\mu_{\sharp}(Q))=\cc{T}{T}{Q}(\chi\cdot\mu)=\sigma^*\cc{G}{T}{E}(\chi\cdot\mu) \in A^1(C) \quad \text{ for any } \chi, \mu\in \Lambda^*(T).
$$
Hence, the cycle in the second row of (\ref{E:trasg-red}) is equal to 
\begin{equation}\label{E:Del-Ch}
\sum_i \pi_*\left(c_1(\chi_{i \sharp}(Q))\cdot c_1(\mu_{i \sharp}(Q))\right)=\sum_i \pi_*(\sigma^*\cc{G}{T}{E}(\chi_i\cdot \mu_i )) \quad \text{ for any } \sum_i \chi_i\cdot \mu_i \in (\Sym^2\Lambda^*(T))^{\scr W_G}.
\end{equation}
If $v:=\sum_i \chi_i\cdot \mu_i \in (\Sym^2\Lambda^*(T))^{\scr W_G}$, we have the following equalities in $A^1(S)_\mathbb{Q}$
\begin{equation}\label{E:Br}
\displaystyle\begin{array}{lcl}
c_1(\tau_G)_{\bbQ}^{\leq m}(v)(C\xrightarrow{\pi} S, E)&:=&(\pi\circ p)_*\left(\cc{G}{B}{E}\left(v\cdot \frac{\prod_{\alpha>0}\alpha}{|\scr W_G|}\right)\right)=\\
&=&\pi_*\sigma^* q^* p_*\left(\cc{G}{B}{E}\left(v\cdot \frac{\prod_{\alpha>0}\alpha}{|\scr W_G|}\right)\right)=\\
&=&\pi_*\sigma^*\cc{G}{T}{E}(v).
\end{array}
\end{equation}
The first equality follows because $\sigma$ is a section of $q:P/T\to C$ and the second one by applying the functor $\pi_*\sigma^*$ to the equality in Corollary \ref{C:Chow-flag}. Putting (\ref{E:Del-Ch}) and (\ref{E:Br}) together, we have the assertion.
\end{proof}

We are now ready for the proof of Theorem \ref{T:BunG}\eqref{BunG-1}, distinguishing the case $(g,n)\neq (1,0)$ from the special case $(g,n)= (1,0)$.

\begin{proof}[Proof of Theorem \ref{T:BunG}\eqref{BunG-1} in the case $(g,n)\neq (1,0)$]
First of all, the uniqueness follows from the injectivity of $\iota_\#^*$, see Corollary \ref{C:inj}\eqref{C:inj1} (recall that we are assuming that $g\geq 1$).  

The remaining part of the proof is devoted to the existence of the map $\tau_G$. Observe that we can assume that $G$ is not a torus, for otherwise the statement is a tautology. 

Consider the factorization of the morphism $\iota_\#$ as 
$$
\iota_\#:\bg{T_G}^d\xrightarrow{l_\#} \bg{B_G}^d \xrightarrow{j_\#} \bg{G}^{[d]},
$$
induced by the inclusions $l:T_G\hookrightarrow B_G$ and $j:B_G\hookrightarrow G$, where we have used that $\pi_1(B_G)=\pi_1(T_G)$ since $B_G^{\red}=T_G$. 

According to Theorem \ref{T:Holla}, up to choosing a different  representative of $[d]\in \pi_1(G)$, we can assume that the morphism $j_\#$ is  smooth, of finite type and with geometrically integral fibers. 

Since $\bg{B_G}^d$ is  quasi-compact by Proposition \ref{P:qc=a} and $G$ is non-abelian, Corollary \ref{C:cod2-red} implies that there exists $m\gg 0$ (which we fix from now on) such that 
we have a factorization 
$$j_\#: \bg{B_G}^d\xrightarrow{\ov{j_\#}} \bg{G}^{[d], \leq m}\subset \bg{G}^{[d]}.
$$
and such the complementary substack of $\bg{G}^{[d], \leq C}$ has codimension at least two in $\bg{G}^{[d]}$. Clearly the morphism $\ov{j_\#}$ is still smooth (hence flat), of finite type and with geometrically integral fibers.  Moreover, the algebraic stacks $\bg{B_G}^d$ and $\bg{G}^{[d], \leq m}$ are regular and integral by Corollary \ref{C:regint}. Furthermore, the algebraic stacks $\bg{B_G}^d$ and $\bg{G}^{[d], \leq m}$ are both of finite type over $\Mg$ (and hence over $k$) by, respectively, Proposition \ref{P:qc=a}  and Proposition \ref{P:inst-cover}\eqref{P:inst-cover2}; hence, since $(g,n)\neq (1,0)$, they are both quotient stacks over $k$ by Proposition \ref{P:qstack-cover}. 

Observe that we have the following chain of homomorphisms
$$\xymatrix{
\Pic(\bg{G}^{[d]})\ar[r]_{\cong}^{\res} & \Pic(\bg{G}^{[d], \leq m}) \ar[r]_{\cong}^{c_1} \ar@(ur,ul)[rrr]^{c_1^{\bbQ}}  & A^1(\bg{G}^{[d], \leq m})\ar@{->>}[r]& A^1(\bg{G}^{[d], \leq m})_{\tf}  \ar@{^{(}->}[r] & A^1(\bg{G}^{[d], \leq m})_{\bbQ}
}$$
where the map $\res$ is given by restriction and it is an isomorphism by Lemma \ref{restr}, the map $c_1$ is an isomorphism by Proposition \ref{Equi}\eqref{Equi2} (using that $\bg{G}^{[d], \leq m}$ is a regular quotient $k$-stack)  and  $A^1(\bg{G}^{[d], \leq m})_{\tf}$ is the torsion-free quotient of $A^1(\bg{G}^{[d], \leq m})$.

The proof of the Theorem will follow from the 

\un{Claim I:}  The homomorphism $c_1(\tau_G)_{\bbQ}^{\leq m}$ of  \eqref{E:trasgrS} factors set-theoretically through $\Pic(\bg{G}^{[d],\leq m})$, i.e., there exists a map of sets 
\begin{equation}\label{E:delta-m}
\tau_G^{\leq m}:(\Sym^2\Lambda^*(T_G))^{\scr W_G} \to \Pic(\bg{G}^{[d],\leq m}) \quad \text{ such that } c_1^{\bbQ}\circ \tau_G^{\leq m}=c_1(\tau_G)_{\bbQ}^{\leq m}.
\end{equation}

\vspace{0.1cm}

Let us first show that Claim I allows us to conclude the proof of the Theorem. Indeed, the Claim implies that $c_1(\tau_G)_{\bbQ}^{\leq m}$ factors through a homomorphism 
\begin{equation}\label{E:c1delta-m}
c_1(\tau_G)^{\leq m}:(\Sym^2\Lambda^*(T_G))^{\scr W_G} \to A^1(\bg{G}^{[d],\leq m})_{\tf}. 
\end{equation}
Since $\RPic(\bg{G}^{[d]})$ is torsion-free by Corollary \ref{C:inj}\eqref{C:inj1} and Theorem \ref{BunT}\eqref{BunT2}, the composition 
$$A^1(\bg{G}^{[d], \leq m})\xrightarrow[\cong]{\res^{-1}\circ c_1^{-1}} \Pic(\bg{G}^{[d]}) \twoheadrightarrow \RPic(\bg{G}^{[d]})
$$
factors through a homomorphism
\begin{equation}\label{E:A1-RPic}
\alpha^{\leq m}: A^1(\bg{G}^{[d], \leq m})_{\tf} \twoheadrightarrow  \RPic(\bg{G}^{[d]}).
\end{equation}
Now we define the transgression map as
$$
\tau_G: (\Sym^2\Lambda^*(T_G))^{\scr W_G}\xrightarrow{c_1(\tau_G)^{\leq m}} A^1(\bg{G}^{[d],\leq m})_{\tf} \xrightarrow{\alpha^{\leq m}}  \RPic(\bg{G}^{[d]}).
$$
Proposition \ref{P:Q-transfr}\eqref{E:trasg-red} implies that the composition $\iota_\#^*\circ \tau_G$ is equal to the $\scr W_G$-invariant part of the homomorphism $\tau_{T_G}$ defined in Theorem \ref{BunT}\eqref{BunT1}, and we are done. 

\vspace{0.1cm}

It remains to prove Claim I. To this aim, consider the following commutative diagram induced by the morphism $\ov{j_\#}$:
\begin{equation}\label{E:diag-j-c1}
\xymatrix{
 \Pic(\bg{G}^{[d], \leq m}) \ar[r]_{\cong}^{c_1}\ar@{^{(}->}[d]^{\ov{j_\#}^*} \ar@(ur,ul)[rr]^{c_1^{\bbQ}}   & A^1(\bg{G}^{[d], \leq m})\ar[r] \ar@{^{(}->}[d]^{\ov{j_\#}^*}& A^1(\bg{G}^{[d], \leq m})_{\bbQ} \ar@{^{(}->}[d]^{\ov{j_\#}^*}\\
  \Pic(\bg{B_G}^{d}) \ar[r]_{\cong}^{c_1} \ar@(dr,dl)[rr]_{c_1^{\bbQ}}  & A^1(\bg{B_G}^{d})\ar[r]& A^1(\bg{B_G}^{d})_{\bbQ}
}
\end{equation}
where both the maps  $c_1$ are isomorphisms by Proposition \ref{Equi}\eqref{Equi2} (using that $\bg{G}^{[d], \leq m}$ and $\bg{B_G}^d$ are  regular quotient $k$-stacks) and the maps $\ov{j_\#}^*$ are injective since the composition 
$$\iota_\#^*:\Pic(\bg{G}^{[d]})\xrightarrow[\cong]{\res}  \Pic(\bg{G}^{[d], \leq m})\xrightarrow{\ov{j_\#}^*}  \Pic(\bg{B_G}^{d}) \xrightarrow{l_\#^*} \Pic(\bg{T_G}^d)
$$
is injective by Corollary \ref{C:inj}\eqref{C:inj1}.

Fix $v\in(\Sym^2\Lambda^*(T))^{\mathscr W_G}$ and set $l:=c_1(\tau_G)^{\leq m}_{\bbQ}(v)\in A^1(\bg{G}^{[d], \leq m})_{\bbQ}$. The pull-back $\ov{j_\#}^*(l)\in  A^1(\bg{B_G}^{d})_{\bbQ}$ is an integral class by Proposition \ref{P:Chow-flag}\eqref{Chow-flag-ii}, that is 
\begin{equation}\label{E:lbM}
\ov{j_\#}^*(l)=c_1^{\bbQ}(M) \quad \text{ for some } M\in  \Pic(\bg{B_G}^{d}).
\end{equation}

\un{Claim II:}  The line bundle $M$ is trivial along the fibers of $\ov{j_\#}$.

Indeed, let $x: \Spec K\to \bg{G}^{[d],\leq m}$ be a point (with $K$ some field) corresponding to a $G$-bundle $E\to C\xrightarrow{\pi} \Spec(K)$, together with $n$ ordered pairwise fiberwise disjoint sections of $\pi$ (that we omit from the notation, as usual). The  fiber of $\ov{j_\#}$ above the point $x$  can be identified with the $K$-stack  $\text{Sect}^d(C,E/B)$ 
 of sections of the flag bundle $E/B\xrightarrow{p} C$ such that the associated reduction to a $B_G$-bundle is a $K$-rational point of  $\bg{B_G}^d$. 
By  Proposition \ref{P:Chow-flag}\eqref{Chow-flag-ii}, the first Chern class of $M$ is such that for any morphism $S\xrightarrow{f} \text{Sect}^d(C,E/B)$ we have that 
$$
c_1(M)(f)=(\pi\times_k \id_S)_*\left(\ee{G}{E\times_K S}(v)\right)\in A^1(S),
$$
where $\ee{G}{E\times_K S}(v)\in A^2(C\times_K S)$ is the class appearing in Proposition \ref{P:Chow-flag}\eqref{Chow-flag-ii}. By the functoriality of the class $\ee{G}{E}(-)$, we have that 
$$
\ee{G}{E\times_K S}(v)=\mathrm{pr}_C^* \ee{G}{E}(v),
$$
where $\mathrm{pr}_{C}:C\times_K S\to C$ is the projection onto the first factor. The class $\ee{G}{E}(v)$ belongs to $A^2(C)$ which is zero since $C$ is a curve over $K$. Hence, $\ee{G}{E\times_K S}(v)=0$ which  implies that $c_1(M)=0$ and proves Claim II.

We already observed that $\ov{j_\#}:\bg{B}^d\to \bg{G}^{[d], \leq m}$ is a smooth morphism of finite type between regular integral stacks with geometrically integral fibers. Up to replacing the quotient stack $ \bg{G}^{[d], \leq m}$ with an equivariant approximation as in Proposition \ref{Equi}(iii), we may assume that $\bg{G}^{[d], \leq m}$ is a regular integral algebraic space over $k$ (in particular, it is generically a scheme by \cite[\href{https://stacks.math.columbia.edu/tag/06NH}{Tag 06NH}]{stacks-project}). By Proposition \ref{P:fiber}, we have an exact sequence of Picard groups
\begin{equation}\label{E:lbI}
\Pic(\bg{G}^{[d], \leq m})\xrightarrow{\ov{j_\#}^*}\Pic(\bg{B}^d)\xrightarrow{\res_{\eta}}\Pic(\operatorname{Sect}^d(C_{\eta},E_{\eta}/B)\to 0,
\end{equation}
where the group on the right-hand side is the Picard group of the generic fiber of $\ov{j_\#}$. By \eqref{E:lbI} and Claim II, we have that 
\begin{equation}\label{E:lbL}
\ov{j_\#}^*(L)=M \quad \text{ for some } L\in  \Pic(\bg{G}^{[d], \leq m}).
\end{equation}
From \eqref{E:lbM} and \eqref{E:lbL} and using the injectivity of the maps $\ov{j_\#}^*$ in the diagram \eqref{E:diag-j-c1}, we infer that $c_1^{\bbQ}(L)=l$, which concludes the proof of Claim I.
\end{proof}

\begin{proof}[Proof of Theorem \ref{T:BunG}\eqref{BunG-1} in the case $(g,n)= (1,0)$]
Consider the commutative diagram
\begin{equation}\label{E:diagtra10}
\xymatrix{
\Sym^2 \Lambda^*(T_G) \ar[d]^{\tau_{T_G,1,1}}  \ar@/_4pc/[dd]_{\tau_{T_G,1,0}}  & (\Sym^2 \Lambda^*(T_G))^{\scr W_G} \ar@{_{(}->}[l] \ar[d]^{\tau_{G,1,1}}\\
\RPic(\mathrm{Bun}_{T_G,1,1}^d) & \RPic(\mathrm{Bun}_{G,1,1}^{[d]})  \ar@{_{(}->}[l]_{\iota_{\#,1,1}^*}    \\
\RPic(\mathrm{Bun}_{T_G,1,0}^d)   \ar@{^{(}->}[u]_{F_{T_G}^*} & \RPic(\mathrm{Bun}_{G,1,0}^{[d]}) \ar@{_{(}->}[u]_{F_G^*} \ar@{_{(}->}[l]_{\iota_{\#,1,0}^*} \\
}
\end{equation}
where the bottom square is \eqref{E:diagPic10}, the commutativity of the top square follows from Theorem \ref{T:BunG}\eqref{BunG-1} applied to $(g,n)=(1,1)$ and the equality $\tau_{T_G,1,1}=F_{T_G}^*\circ \tau_{T_G,1,0}$ follows from the fact that the transgression maps for tori  do not involve the marked sections (see Theorem \ref{BunT}\eqref{BunT1}). 

Now, since $F_G$ (resp., $F_{T_G}$) is a  family of curves, by the seesaw principle the image of the pull-back $F_G^*$ (resp.,  $F_{T_G}^*$) consists of the classes of line bundles on $\mathrm{Bun}^{[d]}_{G,1,1}$ (resp., on $\mathrm{Bun}^d_{T_G,1,1}$) that are trivial on the fibers of $F_G$ (resp., of $F_{T_G}$). Moreover, since $F_{T_G}$ is the pull-back of $F_G$ 
via the morphism $\iota_{\#, 1,0}$ (see \eqref{E:diag10}), a line bundle $L$ on $\mathrm{Bun}^{[d]}_{G,1,1}$ is trivial on the fibers of $F_G$ if and only if $\iota_{\#, 1,1}^*(L)$ is trivial on the fibers of $F_{T_G}$.  In other words, the bottom square in \eqref{E:diagtra10} is cartesian. 
Using this and the fact that $\tau_{T_G,1,1}$ factors through the inclusion $F_{T_G}^*$ (via the homomorphism $\tau_{T_G,1,0}$, see \eqref{E:diagtra10}), we deduce that $\tau_{G,1,1}$ must factor through the inclusion $F_{T_G}^*$ giving rise to a homomorphism 
$$\tau_{G,1,0}:(\Sym^2 \Lambda^*(T_G))^{\scr W_G}\to \RPic(\mathrm{Bun}^{[d]}_{G,1,0}),$$
that, by construction, satisfies the property stated in Theorem \ref{T:BunG}\eqref{BunG-1}. 
\end{proof}

\begin{rmk}\label{R:lift-tras}
The transgression map $\tau_G: \Sym^2\Lambda^*(T_G)^{\scr W_G} \to \RPic(\bg{G}^{[d]})$ admits a canonical lift to $\Pic(\bg{G}^{[d]})$. This follows from Theorem \ref{T:BunG}\eqref{BunG-1} using that $\tau_{T_G}$ admits a lifting to $\Pic(\bg{T_G}^{[d]})$ (as we observed after Theorem \ref{BunT})  and   Corollary \ref{C:inj}\eqref{C:inj1}.
\end{rmk}

\subsection{$\scr W_G$-invariant line bundles on the moduli stack of $G$-bundles}\label{SS:WGinv}

The aim of this subsection is to prove Theorem \ref{T:BunG}\eqref{BunG-2}. Hence we will assume that $g\geq 1$ throughout this subsection.

The morphisms of linear algebraic groups $T_\ab:T_G\xrightarrow{\iota} G\xrightarrow{\ab} G^{\ab}$ induce the following diagram relating the three (injective) transgression maps $\tau_{T_G}$, $\tau_G$ and $\tau_{G^{\ab}}$:
\begin{equation}\label{E:Bigamalg}
\xymatrix{
\Sym^2\Lambda^*(G^{\ab})\ar@{^{(}->}[r]_{\alpha}\ar@{^{(}->}[d]^{\tau_{G^\ab}} \ar@{^{(}->}@(ur,ul)[rrr]^{\Sym^2\Lambda^*_{\ab}}& \left(\Sym^2\Lambda^*(T_G)\right)^{\mathscr W_G}\ar@{^{(}->}[d]\ar@{^{(}->}@/^2pc/[rdd]^{\tau_G} \ar@{^{(}->}[rr]_{\beta}& & \Sym^2(\Lambda^*(T_G))\ar@{^{(}->}[dd]^{\tau_{T_G}}\\
\RPic\left(\bg{G^{\ab}}^{[d]^{\ab}}\right)\ar@{^{(}->}[r]\ar@{^{(}->}@/_2pc/[rrd]_{\ab_\#^*} &\mathrm{PO}\ar@{-->}[rd]^{\Pi}\\
&&\RPic\left(\bg{G}^{[d]}\right) \ar@{^{(}->}[r]^{\iota_\#^*} & \RPic\left(\bg{T_G}^{d}\right)
}
\end{equation}
where we have used that the injective homomorphism $\Sym^2\Lambda_{\ab}^*$ factors through the invariant subgroup $\beta: (\Sym^2\Lambda^*(T_G))^{\mathscr W_G}\hookrightarrow\Sym^2 \Lambda^*(T_G) $ and where we have also inserted the pushout 
$$
\PO:=\RPic\left(\bg{G^{\ab}}^{[d]^{\ab}}\right)\amalg_{\Sym^2 \Lambda^*(G^{\ab})}\Big(\Sym^2\Lambda^*(T_G)\Big)^{\mathscr W_G}.
$$
Observe that 
\begin{enumerate}[(a)]
\item $(T_\ab)_\#^*=  \iota_\#^*\circ \ab_\#^*$ is injective (and hence also $\ab_\#^*$ is injective) by Theorem \ref{BunT}\eqref{BunT1};
\item $\iota_\#^*$ is injective by Corollary \ref{C:inj}\eqref{C:inj1}.
\item $\tau_{T_G}\circ \beta= \iota_\#^*\circ \tau_G$ by Theorem \ref{T:BunG}\eqref{BunG-1};
\item  $\tau_{T_G}\circ \beta\circ \alpha=\tau_{T_G}\circ   \Sym^2\Lambda_{\ab}^*=(T_\ab)_\#^*\circ \tau_{G^{\ab}}=  \iota_\#^*\circ \ab_\#^*\circ \tau_{G^{\ab}}$ since the transgression map for tori is functorial by Theorem \ref{BunT}\eqref{BunT1}.
\end{enumerate}
Combining (b), (c) and (d), we  get  $\tau_G\circ \alpha=\ab_\#^*\circ \tau_{G^{\ab}}$, which implies, by the universal property of the push-out,  that there exists a morphism (indicated by a dotted arrow in the above diagram) 
$$\ab^*\coprod \tau_G:=\Pi:\PO\to  \RPic\left(\bg{G}^{[d]}\right).$$
In order to prove Theorem \ref{T:BunG}\eqref{BunG-2}, it is enough to show that:
\begin{enumerate}[(i)]
\item\label{i:po-fi} the homomorphism $\Pi$ is an inclusion of finite index;
\item\label{i:prim-lat} the inclusion
$\iota_\#^*\circ \Pi: \PO\hookrightarrow \RPic(\bg{T_G}^d)$ is a primitive sub-lattice, i.e., $$\mathrm{PO}=\mathrm{PO}_{\mathbb Q}\cap \RPic(\bg{T_G}^d)\subset \RPic(\bg{T_G}^d)_{\mathbb Q}.$$
\end{enumerate}
Indeed, property \eqref{i:po-fi} implies $\mathrm{PO}_{\mathbb Q}=\RPic(\bg{G}^{[d]})_{\mathbb Q}$. Combining with \eqref{i:prim-lat}, we have
$$
\RPic(\bg{G}^{[d]})\subset \RPic(\bg{G}^{[d]})_{\mathbb Q}\cap \RPic(\bg{T_G}^d)=\mathrm{PO}_{\mathbb Q}\cap \RPic(\bg{T_G}^d)=\mathrm{PO}.
$$
Hence, the homomorphism $\Pi$ is an isomorphism and  Theorem \ref{T:BunG}\eqref{BunG-2} follows.

\vspace{0.1cm}

The remaining of this subsection is devoted to proving  \eqref{i:po-fi} and \eqref{i:prim-lat}. Before doing this, we need to identify the push-out with a certain subgroup of $\RPic(\bg{T_G}^d)$.
\begin{defin}\label{D:weil-act}
We call the \emph{algebraic action} of the Weyl group $\scr W_G$ on the group   $\RPic(\bg{T_G}^d)$ the unique action such that, on the tautological bundles, it is defined as follows
\begin{equation}
\begin{array}{lcl}
w.\mathscr L(\chi,\zeta)&=& \mathscr L(w.\chi,\zeta),\\
w.\langle(\chi,\zeta),(\chi',\zeta')\rangle&=& \langle(w.\chi,\zeta),(w.\chi',\zeta')\rangle,
	\end{array}
\end{equation}
where $w.\chi$ is the natural action of $\mathscr W_G$ on the character lattice $\Lambda^*(T_G)$. We will denote by $\RPic(\bg{T_G}^d)^{\scr W_G}$ the subgroup in $\RPic(\bg{T_G}^d)$ of the invariant elements.
\end{defin}
Observe that the algebraic action is well-defined on $\RPic(\bg{T_G}^d)$ because of Theorem \ref{BunT}. Moreover, the exact sequences \eqref{seq-bunT} and \eqref{seq-bunTg1} for $\RPic(\bg{T_G}^d)$ are equivariant with respect to the natural action of $\scr W_G$ on $\Lambda^*(T_G)$. Furthermore, we can extend the algebraic action of $\scr W_G$ to an action on  $\Pic(\bg{T_G}^d)$ by letting $\scr W_G$ act trivially on the pull-backs of the line bundles on $\Mg$.

\begin{rmk}The Weyl group $\scr W_G$ acts naturally on the universal moduli stack $\bg{T_G}$ of $T_G$-bundles. Indeed, for to $w\in\scr N(T_G)$ and to any $T_G$-bundle $E$ on a family $(C\to S, \un{\sigma})$ of $n$-pointed curves of genus $g$, we can associate a new $T_G$-bundle $w.E$ on $(C\to S, \un{\sigma})$, whose total space is the same as $E$, but the action of the torus is twisted by $w$, i.e., the action $\sigma_{w.E}:w.E\times T_G\to w.E$ is defined as $\sigma_{w.E}(p,t):=\sigma_E(p,w(t))$, where $\sigma_E$ is the action of $T_G$ on $E$.
In general, $E$ and $w.E$ are not isomorphic as $T_G$-bundles. However, if $w\in T_G$, the morphism $E\to w.E$ sending $p$ into $w.p$ is an isomorphism of $T_G$-bundles. In particular, the group $\scr W_G$ has a \emph{natural action} on $\Pic(\bg{T_G})$. In general, $\mathscr W_G$ does not preserve the connected components of $\bg{T_G}$. Indeed, the action by $w\in\mathscr W_G$ defines an isomorphism
$$
\sigma_G(w,-):\bg{T_G}^d\cong\bg{T_G}^{w.d}
$$
where $w.d$ is the natural action of $\mathscr W_G$ on the cocharacter lattice $\Lambda(T_G)=\pi_1(T_G)$. 
It is easy to check that the algebraic action of Definition \ref{D:weil-act} coincides with the natural action restricted to $\Pic(\bg{T_G}^d)\subset \Pic(\bg{T_G})$ if (and only if) $w.d=d$.
\end{rmk}

\begin{lem}\label{L:push=inv}
The commutative diagram of abelian groups
\begin{equation}\label{E:amalg2}
\xymatrix{
\Sym^2\Lambda^*(G^{\ab})\ar@{^{(}->}[rr]^{\Sym^2 \Lambda^*_\ab}\ar@{^{(}->}[d]^{\tau_{G^\ab}}&& (\Sym^2\Lambda^*(T_G))^{\mathscr W_G}\ar@{^{(}->}[d]^{\tau_{T_G}^{\scr W_G}}\\
\RPic\left(\bg{G^{\ab}}^{[d]^{\ab}}\right)\ar@{^{(}->}[rr]^{(T_\ab)_\#^*} &&\RPic(\bg{T_G}^d)^{\scr W_G}
}
\end{equation}
is a push-out, where $\tau_{T_G}^{\scr W_G}$ is the $W_G$-invariant part of the transgression map $\tau_{T_G}$ and $(T_\ab)_\#^*$ is the pull-back map (which lands in the $\scr W_G$-invariant subgroup) induced by the morphism of tori $T_\ab:T_G\xrightarrow{\iota} G \xrightarrow{\ab} G^{\ab}$.  
\end{lem}
\begin{proof}
We have already observed that the exact sequences in Theorem \ref{BunT}\eqref{BunT1} for $\RPic(\bg{T_G}^d)$  are $\scr W_G$-equivariant with the action defined in Definition \ref{D:weil-act}. Moreover, since the exact sequences in Theorem \ref{BunT}\eqref{BunT1} are functorial with respect to morphisms of tori, we can pull-back the exact sequences for $G^{\ab}$ along the morphism $T_\ab:T_G\to G^{\ab}$ and we will land in the $\scr W_G$-invariant parts of the exact sequences for $T_G$ since $\Lambda_{\ab}^*(\Lambda^*(G^{\ab}))\subset \Lambda^*(T_G)^{\scr W_G}$.  In other words, we have a morphism of exact sequences
\begin{equation}\label{E:GabTG}
\xymatrix{
 \Sym^2\Lambda^*(G^\ab)\oplus\big(\Lambda^*(G^\ab)\otimes\bbZ^n\big)\ar@{^{(}->}[r]^(0.6){\tau_{G^{\ab}}+\sigma_{G^{\ab}}}\ar@{^{(}->}[d]^{\Sym^2\Lambda^*_\ab\oplus (\Lambda_\ab^*\otimes \id_{\bbZ^n})}& \RPic\Big(\bg{G^\ab}^{[d]^\ab}\Big)\ar@{->>}[r]^(0.6){\rho_{G^{\ab}}}\ar@{^{(}->}[d]^{(T_\ab)_\#^*}&\frac{\Lambda^*(G^\ab)}{m\Lambda^*(G^\ab)}\ar@{^{(}->}[d]^{[\Lambda^*_\ab]}\\
		 (\Sym^2\Lambda^*(T_G))^{\mathscr W_G}\oplus\Big(\Lambda^*(T_G)\otimes\bbZ^n\Big)^{\scr W_G}\ar@{^{(}->}[r]_(0.6){\tau_{T_G}^{\scr W_G}+\sigma_{T_G}^{\scr W_G}}& \RPic\Big(\bg{T_G}^d\Big)^{\scr W_G}\ar[r]_(0.6){\rho_{T_G}^{\scr W_G}}&\Big(\frac{\Lambda^*(T_G)}{m\Lambda^*(T_G)}\Big)^{\mathscr W_G}\\
}
\end{equation}
where $m=0$ if $g\geq 2$ and $m=2$ if $g=1$, and the injectivity of $[\Lambda^*_\ab]$ for $m=2$ follows from the equality $(\Lambda_{\ab}^*)^{-1}(2\Lambda^*(T_G))=2\Lambda^*(G^{\ab}))$ which is easily deduced from the fact that the embedding $\Lambda_{\ab}^*:\Lambda^*(G^{\ab})\hookrightarrow \Lambda^*(T_G)$ is primitive (see \eqref{E:tori-car}).  Note that the morphism $\rho_{T_G}^{\scr W_G}$  in the above diagram \eqref{E:GabTG} could be non-surjective since taking invariants with respect to the  $\scr W_G$-action  is not an exact functor (only left exact). 
We now make the following 

\un{Claim:}  $\Im(\rho_{T_G}^{\scr W_G})=\Im([\Lambda^*_\ab])$.

Indeed, if $m=0$ then the Claim follows from the fact that $\Lambda_{\ab}^*(\Lambda^*(G^{\ab}))=\Lambda^*(T_G)^{\scr W_G}$ by Lemma \ref{L:inv-char}, which implies that $[\Lambda^*_\ab]$ is surjective and hence that also $\rho_{T_G}^{\scr W_G}$ is surjective.  If $m=2$, we argue as follows. The inclusion $\Im(\rho_{T_G}^{\scr W_G})\supset\Im([\Lambda^*_\ab])$ follows from the surjectivity of $\rho_{G^{\ab}}$. In order to prove the reverse inclusion, consider the exact sequence $0\to\Lambda^*(T_G)\xrightarrow{2.} \Lambda^*(T_G)\to \frac{\Lambda^*(T_G)}{2\Lambda^*(T_G)}\to 0$. Taking the long exact sequence of cohomology  groups $H^i(\scr W_G,-)$ attached to the above sequence and using again Lemma \ref{L:inv-char}, we get a new exact sequence:
\begin{equation}
0\to \frac{\Lambda^*(G^\ab)}{2\Lambda^*(G^\ab)}\xrightarrow{[\Lambda^*_\ab]}\left(\frac{\Lambda^*(T_G)}{2\Lambda^*(T_G)}\right)^{\mathscr W_G}\xrightarrow{\partial} H^1(\scr W_G,\Lambda^*(T_G)).
\end{equation}
Hence, the Claim is equivalent to $\tau\circ\rho_{T_G}^{\scr W_G}=0$. We recall here the basic properties of the map $\partial$, the details are left to the reader. A crossed homomorphism is a function $f:\scr W_G\to \Lambda^*(T_G)$ such that $f(w_1\cdot w_2)=f(w_1)+w_1.f(w_2)$. A crossed homomorphism $f$ is called principal if there exists $m\in\Lambda^*(T_G)$ such that $f(w)=w.m-m$. The group $H^1(\scr W_G,\Lambda^*(T_G))$ is the quotient of the group of crossed homomorphisms by the subgroup of the principal ones. For any character $[\chi]\in(\Lambda^*(T_G)/2\Lambda^*(T_G))^{\scr W_G}$, the element $\partial([\chi])$ is the class of the crossed homomorphisms
\begin{equation}
\begin{array}{ccc}
\scr W_G&\to& H^1(\scr W_G,\Lambda^*(T_G))\\
w&\mapsto& \frac{1}{2}(w.\chi -\chi).
\end{array}
\end{equation}
Let $L$ be a $\scr W_G$-invariant line bundle of $\RPic(\mathrm{Bun}_{T_G,1,n}^d)$. Using the $\scr W_G$-equivariance of the homomorphism $\rho_{T_G}$, we have that
$$(\partial\circ\rho_{T_G}(L))(w)=\frac{1}{2}(w\cdot \rho_{T_G}(L)- \rho_{T_G}(L))=\frac{1}{2}(\rho_{T_G}(w.L-L))= 0 \quad \text{ for any } w\in \scr W_G,$$ 
thus concluding the proof of Claim.

Now, applying the snake lemma to the diagram \eqref{E:GabTG} and using the above Claim together with the fact that $\Lambda_{\ab}^*(\Lambda^*(G^{\ab}))=\Lambda^*(T_G)^{\scr W_G}$ by Lemma \ref{L:inv-char}, we get that 
\begin{equation*}
\frac{(\Sym^2\Lambda^*(T_G))^{\mathscr W_G}}{\Sym^2\Lambda^*(G^\ab)}=\coker \left(\Sym^2\Lambda^*_\ab\oplus (\Lambda_\ab^*\otimes \id_{\bbZ^n})\right)\xrightarrow{\cong} \coker((T_{\ab})_\#^*).
\end{equation*}
The above isomorphism implies that the commutative diagram \eqref{E:amalg2} is a push-out.
\end{proof}

We are now ready to prove

\noindent\fbox{Part \eqref{i:po-fi}.} 
From the commutative diagram \eqref{E:Bigamalg} it follows that the canonical isomorphism $\PO\cong \RPic(\bg{T_G}^d)^{\scr W_G}$ provided by Lemma \ref{L:push=inv} identifies the homomorphism $\iota_\#^*\circ \Pi$ with the inclusion $ \RPic(\bg{T_G}^d)^{\scr W_G}\subset  \RPic(\bg{T_G}^d)$. In particular,  $\Pi$ must be injective. 

In order to show that $\Pi$ is a finite index inclusion, it remains to show that the rank of the codomain of $\Pi$ is at most the rank of the domain of $\Pi$ (hence they must have the same rank). This will be shown  in the next two lemmata.

\begin{lem}\label{L:rank-pic}
We have the following inequality
\begin{equation}\label{E:rank-pic}
\rk\Pic\left(\bg{G}^{[d]}\right)\leq \rk\Pic\left(\bg{G^\ab}^{[d]^\ab}\right)+s,
\end{equation}
where $s$ is the number of simple factors in  the adjoint quotient $G^\ad$.
\end{lem}
\begin{proof}
For any geometric point $\ov \eta\to \Mg$ corresponding to an $n$-pointed curve $(\cC_{\ov\eta}, \un \sigma)$ of genus $g$, we have a commutative diagram
\begin{equation}\label{E:res-Pic}
\xymatrix{
\frac{\Pic\left(\bg{G}^{[d]}\right)}{\Pic\left(\bg{G^\ab}^{[d]^\ab}\right)}\ar@{^{(}->}[rr]^{[\iota_\#^*]}\ar[d]_{[\res_G^{[d]}(\cC_{\ov\eta})]}&& \frac{\Pic\left(\bg{T_G}^{d}\right)}{\Pic\left(\bg{G^\ab}^{[d]^\ab}\right)}\ar[d]^{[\res_{T_G}^{[d]}(\cC_{\ov\eta})]}\\
\frac{\Pic\left(\mathrm{Bun}_G^{[d]}(C_{\ov\eta})\right)}{\Pic\left(\mathrm{Bun}_{G^\ab}^{[d]^\ab}(C_{\ov\eta})\right)}\ar@{^{(}->}[rr]^{[(\iota_\#(\cC_{\ov \eta})^*]}&& \frac{\Pic\left(\mathrm{Bun}_{T_G}^{d}(C_{\ov\eta})\right)}{\Pic\left(\mathrm{Bun}_{G^\ab}^{[d]^\ab}(C_{\ov\eta})\right)}
}
\end{equation}
where the vertical arrows are given by restriction to the fibers over $\ov \eta\to \Mg$ and the the two horizontal maps are injective by Corollary \ref{C:inj}\eqref{C:inj1}.

\un{Claim:}  If $\ov \eta\to \Mg$ is the geometric generic point then $[\res_{T_G}^{[d]}(\cC_{\ov\eta})]$ (and hence also $[\res_G^{[d]}(\cC_{\ov\eta})]$) is injective. 

Indeed, the Claim is equivalent to the assertion that, under the assumption that $\ov \eta\to \Mg$ is the geometric generic point, we have 
\begin{equation}\label{E:res-TGGab}
(\res^d_{T_G}(\cC_{\ov \eta}))^{-1}\left(\Pic\left(\mathrm{Bun}_{G^\ab}^{[d]^\ab}(C_{\ov\eta})\right)\right)=\RPic\Big(\bg{G^{\ab}}^{[d]^{\ab}}\Big). 
\end{equation}
This follows from Proposition \ref{P:restrPic} applied to the curve $\cC_{\ov\eta}$ and to the tori $T_G$ and $G^{\ab}$. More precisely, using the notation of loc. cit., we first observe that, if $\ov \eta\to \Mg$ is the geometric generic point, then the weak Franchetta conjecture (see Theorem \ref{franchetta}) implies that morphism $\iota_{\cC_{\ov \eta}}:H_{g,n}\to J_{\cC_{\ov \eta}}(\ov \eta)$ is injective. 
Then, using the exact sequence \eqref{E:restrPic}, the equality \eqref{E:res-TGGab} is a consequence of the following two easily checked equalities
$$
\begin{aligned}
&\left(\id_{\Lambda^*(T_G)}\otimes \iota_{\cC_{\ov \eta}}\right)^{-1}\left(\Lambda^*(G^{\ab})\otimes J_{\cC_{\ov \eta}}(\ov \eta) \right)=\Lambda^*(G^{\ab})\otimes H_{g,n}, \\
&\left(\id_{\Lambda^*(T_G)}\oplus (-\otimes \id_{J_{\cC_{\ov \eta}}})\right)^{-1}\left(\Lambda^*(G^{\ab})\oplus \Hom^s(\Lambda(G^{\ab})\otimes \Lambda(G^{\ab}), \End(J_{\cC_{\ov \eta}})) \right)=
\Lambda^*(G^{\ab})\oplus \Bil^s(\Lambda(G^{\ab}).
\end{aligned}
$$

\vspace{0.1cm}

Finally,  the rank of the abelian group on the bottom left corner of \eqref{E:res-Pic} can be computed using the results of Biswas-Hoffmann \cite{BH10}. Indeed, using that $\pi_1(G^{\ab})=\Lambda(G^{\ab})$ is the torsion-free quotient of $\pi_1(G)$ (see \eqref{E:seq-pi1}), we get from \cite[Theorem 5.3.1]{BH10}  an isomorphism 
\begin{equation}\label{E:Pic-NS}
\left(\frac{\Pic\left(\mathrm{Bun}_G^{[d]}(C_{\ov\eta})\right)}{\Pic\left(\mathrm{Bun}_{G^\ab}^{[d]^\ab}(C_{\ov\eta})\right)}\right)_{\mathbb Q}\xrightarrow{\cong} \left(\frac{\NS\left(\mathrm{Bun}_G^{[d]}(C_{\ov\eta})\right)}{\NS\left(\mathrm{Bun}_{G^\ab}^{[d]^\ab}(C_{\ov\eta})\right)}\right)_{\mathbb Q}
\end{equation}
where $\NS(-)$ is the group of \cite[Definition 5.2.1]{BH10}. Moreover, from the discussion at the end of \cite[\S 5.2]{BH10}, we deduce that
\begin{equation}\label{E:rkNS}
\rk\left(\frac{\NS\left(\mathrm{Bun}_G^{[d]}(C_{\ov\eta})\right)}{\NS\left(\mathrm{Bun}_{G^\ab}^{[d]^\ab}(C_{\ov\eta})\right)}\right)=s.
\end{equation}
By putting together \eqref{E:Pic-NS} and \eqref{E:rkNS} and using the injectivity of the  map $[\res_G^{[d]}(\cC_{\ov\eta})]$ (see the Claim), the inequality \eqref{E:rank-pic} follows.
\end{proof}

\begin{lem}\label{L:rank-sym}
We have the following inequality
\begin{equation}\label{E:rank-pic2}
\rk \mathrm{PO}=\rk \Pic\Big(\bg{T_G}^d\Big)^{\scr W_G}=\rk\Pic\left(\bg{G^\ab}^{[d]^\ab}\right)+s,
\end{equation}
where $s$ is the number of simple factors in  the adjoint quotient of $G^\ad$.
\end{lem}
\begin{proof}
From Lemma \ref{L:push=inv} and the fact that the maps that appear in the push-out diagram are all injective, it follows that 
\begin{equation}\label{E:rank-push}
\rk \mathrm{PO}=\rk\Pic\Big(\bg{T_G}^d\Big)^{\scr W_G}=\rk \Pic\left(\bg{G^\ab}^{[d]^\ab}\right)+ \rk (\Sym^2\Lambda^*(T_G))^{\mathscr W_G} - \rk \Sym^2\Lambda^*(G^{\ab}). 
\end{equation}
Recall that $G^\sc$ is the universal cover of the derived subgroup $\scr D(G)$ of $G$ and $\scr R(G)$ is the radical subgroup of $G$   (see \S \ref{red-grps}). We then have isogenies of linear algebraic groups
$$G^\sc\times \scr R(G)\twoheadrightarrow G \quad \text{ and  }  \quad \scr R(G)\twoheadrightarrow G^{\ab}$$ 
which identify their character lattices after tensoring with $\mathbb Q$:
\begin{equation}\label{E:iso-L*}
\begin{aligned}
\Lambda^*(T_G)_{\mathbb Q} & \xrightarrow{\cong} \Lambda^*(\scr R(G)\times T_{G^\sc})_{\mathbb Q}= \Lambda^*(\scr R(G))_{\mathbb Q}\oplus \Lambda^*(T_{G^\sc})_{\mathbb Q}, \\
\Lambda^*(G^{\ab})_{\mathbb Q}& \xrightarrow{\cong} \Lambda^*(\scr R(G))_{\mathbb Q}. \\
\end{aligned}
\end{equation}
From the above isomorphisms, we deduce that 
\begin{equation}\label{E:iso-Sym}
\begin{aligned}
 \rk (\Sym^2\Lambda^*(T_G))^{\mathscr W_G}= \rk (\Sym^ 2\Lambda^*(T_{G^\sc} \times \scr R(G)))^{\scr W_{G}},\\
 \rk \Sym^2\Lambda^*(G^{\ab})=\rk  \Sym^2\Lambda^*(\scr R(G)). 
\end{aligned}
\end{equation}
Using that the first isomorphism in \eqref{E:iso-L*} commutes with the action  of the Weyl group $\scr W_G\cong\scr W_{\scr R(G) \times G^\sc}\cong\scr W_{G^\sc}$ and that $\scr W_{G^\sc}$ acts trivially on $\Lambda^*(\scr R(G))$, we compute 
\begin{equation}\label{E:rk-Sym2}
\begin{aligned}
(\Sym^ 2\Lambda^*(T_{G^\sc} \times \scr R(G)))^{\scr W_{G}} & = (\Sym^2\Lambda^*(T_{G^\sc}))^{\scr W_{G^\sc}}\oplus \left[\Lambda^*(T_{G^\sc})^{\scr W_{G^\sc}}\otimes \Lambda^*(\scr R(G))\right]\oplus \Sym^2\Lambda^*(\scr R(G)) = \\
& = (\Sym^2\Lambda^*(T_{G^\sc}))^{\scr W_{G^\sc}}\oplus \Sym^2\Lambda^*(\scr R(G)),
\end{aligned}
\end{equation}
where we have used that $\Lambda^*(T_{G^\sc})^{\scr W_{G^\sc}}\cong \Lambda^*((G^\sc)^{\ab})$ by Lemma \ref{L:inv-char} and the latter group is zero because the abelianization of a semi-simple group is always trivial.
Finally, Lemma \ref{L:Sym2inv} implies that 
\begin{equation}\label{E:rk-Sym-ss}
\rk(\Sym^2\Lambda^*(T_{G^\sc}))^{\scr W_{G^\sc}}=s.
\end{equation}
We conclude by putting together \eqref{E:rank-push}, \eqref{E:iso-Sym}, \eqref{E:rk-Sym2} and  \eqref{E:rk-Sym-ss}. 
\end{proof}

\fbox{Part \eqref{i:prim-lat}.} 
By what we said at the beginning of the proof of part (i), we need to show that the sublattice $\RPic(\bg{T_G}^d)^{\scr W_G}\subset\RPic(\bg{T_G}^d)$ is primitive. This follows from 

\begin{lem}
Let $W$ be a group  acting $\mathbb Z$-linearly on a torsion-free abelian group $A$. Then the subgroup $A^W$ of $W$-invariants is a primitive subgroup of $A$, i.e., if $m\cdot a$ is $W$-invariant for some $a\in A$ and $m\in\mathbb Z$ then $a$ is $W$-invariant.
\end{lem}
\begin{proof}
Let $\in A$ and $m\in \bbZ$ such that $m\cdot a\in A^W$. Then, using the $\bbZ$-linearity of the action, we get for any $w\in W$
$$0=w\cdot (m\cdot a)-(m\cdot a)=m(w\cdot a-a).$$
Since $A$ is torsion-free, we deduce $w.a=a$ for any $w\in W$, i.e., $a\in A^W$. 
\end{proof}

\subsection{Genus zero case} \label{SS:genus0}

Here, we show Theorem \ref{T:BunGg0}. The first part is easy.

\begin{proof}[Proof of Theorem \ref{T:BunGg0}\eqref{BunGg0-1}]
The injectivity of $w_G^d$ follows from the injectivity of $w_{T_G}^d$ (see Theorem \ref{BunTg0}\eqref{BunTg0-1}) and the injectivity of $\iota^\#$, which holds  by Corollary \ref{C:inj}\eqref{C:inj2} provided that $d^{\ss}$ satisfies condition (*) of Lemma \ref{L:inj-cont}\eqref{L:inj-cont1}. The existence of a representative $d$ of a given class $[d]\in \pi_1(G)$ with the property that $d^{\ss}$ satisfies condition (*) follows from Lemma \ref{L:inj-cont}\eqref{L:inj-cont2}.

The functoriality of $w_G^d$ follows from the functoriality of $w_{T_G}^d$  (see Theorem \ref{BunTg0}\eqref{BunTg0-1}) and the functoriality of $\iota_\#$ which is clear from the definition. 
\end{proof}

The second part of  Theorem \ref{T:BunGg0} will be deduced from the following alternative description of $\RPic\left(\mathrm{Bun}_{G,0,n}^{[d]}\right)$ as a pull-back.

\begin{prop}\label{P:g0pull}
For any $d\in \Lambda(T_G)$, there exists a  homomorphism 
$$\RPic\left(\mathrm{Bun}_{G,0,n}^{[d]}\right)\xrightarrow{f_G^d} (\Sym^2\Lambda^*(T_{G^{\sc}}))^{\scr W_G},$$
such that the diagram
\begin{equation}\label{E:g0pull}
\xymatrix{
\RPic\left(\mathrm{Bun}_{G,0,n}^{[d]}\right)\ar[rr]^{\iota_\#^*}\ar[d]^{f_G^d}&&\RPic\left(\mathrm{Bun}_{T_G,0,n}^d\right)\ar[d]^{w_{\g^\ss}^d}\\
(\Sym^2\Lambda^*(T_{G^{\sc}})^{\scr W_G}\ar[rr]^{(d^\ss,-)}&& \Lambda^*(T_{G^{\sc}})
}
\end{equation}
is cartesian (i.e., it is a pull-back diagram), where $(d^\ss,-)$ is the contraction homomorphism \eqref{E:ctrc-hom} and $w_{\g^\ss}^d$ is the composition 
$$w_{\g^{\ss}}^d: \RPic(\bg{T_G}^d) \xrightarrow{w_T^d} \Lambda^*(T_G) \xrightarrow{\sc^*} \Lambda^*(T_{G^{\sc}})
$$
where the last homomorphism is induced by the morphism $\sc:G^{\sc}\to G$ (see \S\ref{red-grps}).
\end{prop}

Let us first show how, using the above Proposition, we can prove  Theorem \ref{T:BunGg0}.  

\begin{proof}[Proof of Theorem \ref{T:BunGg0} \eqref{BunGg0-2}]
From the cartesian diagram \eqref{E:g0pull}, it follows that an element $\chi\in \Lambda^*(T_G)$ belongs to the image of $w_G^d$ if and only if it belongs to the image of $w_{T_G}^d$ and its image in 
$\Lambda^*(T_{G^{\sc}})$ belongs to the image of $(d^{\ss},-)$. The second condition is equivalent to requiring that $\chi\in \Omega_d^*(T_G)$. We now conclude using the description of  $\Im(w_{T_G}^d)$ from Theorem \ref{BunTg0}\eqref{BunTg0-2}.  
\end{proof}

The remainder of this subsection is devoted to the proof of the above Proposition. 

\begin{proof}[Proposition \ref{P:g0pull}]
We will distinguish three cases.

\fbox{Case I: $n=3$.}
Observe that, since $\mathcal M_{0,3}=\Spec(k)$, we have canonical isomorphisms of stacks 
$$\mathrm{Bun}_{G,0,3}^{[d]}\cong \mathrm{Bun}_G^{[d]}(\mathbb P^1/k) \quad \text{ and } \quad \mathrm{Bun}_{T_G,0,3}^{d}\cong \mathrm{Bun}_{T_G}^{d}(\mathbb P^1/k).$$

By \cite[Thm. 5.3.1]{BH10}, we have a commutative diagram 
\begin{equation}\label{E:PicNS-g0}
\xymatrix{
\Pic\left(\mathrm{Bun}_G^{[d]}(\mathbb{P}^1/k)\right)\ar[rr]^{\iota_\#^*(\bbP^1/k)}\ar[d]^{\cong}_{c_G}&&\Pic\left(\mathrm{Bun}_{T_G}^{d}(\mathbb{P}^1/k)\right)\ar[d]_{\cong}^{c_{T_G}}\\
\NS(\mathrm{Bun}_G^{[d]}(\mathbb{P}^1/k))\ar[rr]^{\iota^{\NS,d}}&&\NS(\mathrm{Bun}_{T_G}^{d}(\mathbb{P}^1/k))\\
}
\end{equation}
where, by  \cite[Def. 5.2.1]{BH10},  $\NS\left(\mathrm{Bun}_{T_G}^{d}(\mathbb{P}^1/k)\right)=\Lambda^*(T_G)$ and 
$$\NS\left(\mathrm{Bun}_G^{[d]}(\mathbb{P}^1/k)\right)\subset \Lambda^*(\scr R(G))\oplus (\Sym^2\Lambda(T_{G^{\sc}}))^{\scr W_G}$$ 
is the subgroup of pairs $(l,b)$ such that the induced character $l+b(d^\ss,-)\in\Lambda^*(\scr R(G)\times T_{G^\sc})$ belongs to the subgroup $\Lambda^*(T_G)\subseteq \Lambda^*(\scr R(G)\times T_{G^\sc})$. From the discussion in \S\ref{Sec:gerbe-fib}, it follows that the above homomorphism $c_{T_G}$ coincides with the weight function $w_{T_G}^d$ of \eqref{E:Pic-gerbe}. 
Moreover, from \cite[Def. 5.2.5]{BH10} it follows that the above homomorphism $\iota^{\NS,d}$ coincide with the restriction to $\NS\left(\mathrm{Bun}_G^{[d]}(\mathbb{P}^1/k)\right)$ of the homomorphism 
$$\Lambda^*(\scr R(G))\oplus (\Sym^2\Lambda(T_{G^{\sc}}))^{\scr W_G} \xrightarrow{\id+(d^\ss,-)}\Lambda^*(\scr R(G)\times T_{G^\sc})$$
which lands in the subgroup $\Lambda^*(T_G)$ by the definition of $\NS\left(\mathrm{Bun}_G^{[d]}(\mathbb{P}^1/k)\right)$. 
By putting everything together, we obtain the following cartesian diagram 
\begin{equation}\label{E:cart-proof}
\xymatrix{
\Pic\left(\mathrm{Bun}_G^{[d]}(\mathbb{P}^1/k)\right)\ar[rr]^{\iota_\#^*(\bbP^1/k)}\ar[d]^{\cong}_{c_G}\ar@{}[drr]^{\square}&&\Pic\left(\mathrm{Bun}_{T_G}^{d}(\mathbb{P}^1/k)\right)\ar[d]_{\cong}^{w_{T_G}^d}\\
\NS(\mathrm{Bun}_G^{[d]}(\mathbb{P}^1/k))\ar[rr]^{\id+(d^\ss,-)}\ar[d]_{pr_2}\ar@{}[drr]^{\square}&&\Lambda^*(T_G)\ar[d]^{\sc^*}\\
(\Sym^2\Lambda(T_{G^{\sc}}))^{\scr W_G}\ar[rr]^{(d^\ss,-)}&& \Lambda^*(T_{G^{\sc}})
}
\end{equation}
where the square on the bottom is cartesian by the definition of $\NS(\mathrm{Bun}_G^{[d]}(\mathbb{P}^1/k))$. The outer  cartesian diagram in \eqref{E:cart-proof} gives the desired cartesian diagram \eqref{E:g0pull} for $n=3$ with $f_G^d:=pr_2\circ c_G$. 

\fbox{Case II: $n>0$.} Consider the commutative diagram 
\begin{equation}\label{E:F-forg}
\xymatrix{ 
\mathrm{Bun}_{G,0,n+1}^{[d]} \ar[d]_{F_{G,n}} & \ar[l]^{\iota_\#} \mathrm{Bun}_{T_G,0,n+1}^{d}\ar[d]^{F_{T_G,n}} \\
\mathrm{Bun}_{G,0,n}^{[d]} & \ar[l]^{\iota_\#} \mathrm{Bun}_{T_G,0,n}^{d}
} 
\end{equation}
where the morphisms $F_{G,n}$ and $F_{T_G,n}$ forget the last marked section. We get an induced commutative diagram between the relative Picard groups 
\begin{equation}\label{E:iota-forg}
\xymatrix{ 
\RPic\left(\mathrm{Bun}_{G,0,n+1}^{[d]} \right)\ar[r]^{\iota_\#^*} & \RPic\left(\mathrm{Bun}_{T_G,0,n+1}^{d} \right)\\
\RPic\left(\mathrm{Bun}_{G,0,n}^{[d]}\right)  \ar[u]^{\ov{F_{G,n}^*}} \ar[r]^{\iota_\#^*} &\RPic\left( \mathrm{Bun}_{T_G,0,n}^{d}\right)\ar[u]_{\ov{F_{T_G,n}^*}}
} 
\end{equation}

\un{Claim:}  If $n>0$ then $\ov{F_{G,n}^*}$ and $\ov{F_{T_G,n}^*}$ are isomorphisms.

\vspace{0.1cm}

The above Claim implies the statement for $n>0$. Indeed, using also that the weight morphism $w_{T_G}^d$ is compatible with the forgetful morphism $\ov{F_{T_G,n}^*}$, we deduce that we have a cartesian diagram  \eqref{E:g0pull}  for $\RPic\left(\mathrm{Bun}_{G,0,n}^{[d]}\right)$ if and only if we have a similar cartesian diagram for $\RPic\left(\mathrm{Bun}_{G,0,n+1}^{[d]}\right)$. Hence we conclude using Case I. 

\vspace{0.1cm}

It remains to prove the Claim. First of all,  the stack $\mathrm{Bun}_{T_G,0,n}^{d}$ is of finite type by Proposition \ref{P:qc=a}, hence it is a quotient stack by Proposition \ref{P:qstack-cover}. Then, using Proposition \ref{Equi}, we can assume that $F_{T_G,n}:\mathrm{Bun}_{T_G,0,n+1}^{d}\to \mathrm{Bun}_{T_G,0,n}^{d}$ is a morphism of smooth integral algebraic spaces of finite type over $k$. 
The same holds for $F_{G,n}$ after restricting to a suitable open subset of the instability exhaustion , see Proposition \ref{P:inst-cover}. 

We will prove the Claim for $\ov{F_{G,n}^*}$; the proof for $\ov{F_{T_G,n}^*}$ is analogous. To this aim, consider the cartesian diagram
\begin{equation}\label{E:F-forgMgn}
\xymatrix{ 
\mathrm{Bun}_{G,0,n+1}^{[d]} \ar[d]_{F_{G,n}}  \ar[r]^{\Phi_{G,0,n+1}} \ar@{}[dr]^{\square} &  \cM_{0,n+1}\ar[d]^{F_n} \\
\mathrm{Bun}_{G,0,n}^{[d]}  \ar[r]^{\Phi_{G,0,n}} &  \cM_{0,n}
} 
\end{equation}
where the morphism $F_n:\cM_{0,n+1}\to \cM_{0,n}$  forgets the last marked section. The above diagram induces the following commutative diagram 
\begin{equation}\label{E:Pic-F}
\xymatrix{ 
0 \ar[r] & \Pic(\cM_{0,n})\ar[r]^{\Phi_{G,0,n}^*} \ar[d]^{F_n^*}& \Pic(\mathrm{Bun}_{G,0,n}^{[d]}) \ar[r] \ar[d]^{F_{G,n}^*}&  \RPic(\mathrm{Bun}_{G,0,n}^{[d]}) \ar[d]^{\ov{F_{G,n}^*}}\ar[r] & 0\\
0 \ar[r] & \Pic(\cM_{0,n+1})\ar[r]^{\Phi_{G,0,n+1}^*} & \Pic(\mathrm{Bun}_{G,0,n+1}^{[d]}) \ar[r] &  \RPic(\mathrm{Bun}_{G,0,n+1}^{[d]}) \ar[r] & 0
}
\end{equation}
with exact rows. By the snake lemma, the fact that $\ov{F_{G,n}^*}$ is an isomorphism is a consequence of the following two facts:
\begin{enumerate}[(a)]
\item \label{Fact-a} $F_{G,n}^*$ and $F_n^*$ are surjective;
\item \label{Fact-b} $\ker(F_{G,n}^*)\subseteq \Im(\Phi_{G,0,n}^*)$. 
\end{enumerate}
 
Indeed, \eqref{Fact-a} is a consequence of Proposition \ref{P:fiber} applied (after a suitable equivariant approximation, see Proposition \ref{Equi}) to the smooth morphisms $F_{G,n}$ and $F_n$
with integral fibers, using the fact that the generic  fiber of  $F_{G,n}$ and of $F_n$ is the projective line with $n>0$ points removed, and hence it has trivial Picard group. 

In order to prove  \eqref{Fact-b}, let $L\in \ker(F_{G,n}^*)$. Since   $F_{G,n}:\mathrm{Bun}_{G,0,n+1}^{[d]}\to \mathrm{Bun}_{G,0,n}^{[d]}$ is obtained from the universal curve $\cC_{G,0,n}^{[d]}\to \mathrm{Bun}_{G,0,n}^{[d]}$ by removing the $n$ universal sections $\sigma_{G,i}$, we have that $L$ belongs to the subgroup of $\Pic(\mathrm{Bun}_{G,0,n+1}^{[d]})\subset \Pic(\cC_{G,0,n}^{[d]})$ generated by the line bundles $\langle \cO(\Im(\sigma_{G,i})-\Im(\sigma_{G,j}))\rangle$. Since the universal curve $\cC_{G,0,n}^{[d]}\to \mathrm{Bun}_{G,0,n}^{[d]}$ together with the $n$ universal sections $\sigma_{G,i}$ is a pull-back of the universal curve $\cC_{0,n}\to \cM_{0.n}$ together with the $n$ universal sections $\sigma_i$ along the morphism $\Phi_{G,0,n}$, we deduce that $\cO(\Im(\sigma_{G,i})-\Im(\sigma_{G,j}))=\Phi_{G,0,n}^*\left(\cO(\Im(\sigma_{i})-\Im(\sigma_{j}))\right)$. Hence, we conclude that $L\in \Im(\Phi_{G,0,n}^*)$.

\fbox{Case III: $n=0$.} Arguing as in Case II,  we observe that the diagram \eqref{E:F-forg} with $n=0$ is cartesian with the vertical morphisms $F_{G,0}$ and $F_{T_G,0}$ being smooth and proper with geometric fibers isomorphic to $\bbP^1$, or, in other words, they are families of curves of genus zero. This implies that the pull-back morphism $F_{G,0}^*$ (resp., $F_{T_G,0}^*$) is injective with image equal to the subgroup of line bundles in $\RPic\left(\mathrm{Bun}_{G,0,1}^{[d]} \right)$ (resp., in $\RPic\left(\mathrm{Bun}_{T_G,0,1}^{d} \right)$) that have degree zero on the geometric fibers of the family $F_{G,0}$ (resp., $F_{T_G,0}$). We deduce that the diagram \eqref{E:iota-forg} for $n=0$ is cartesian (and with injective vertical homomorphisms).  Therefore, using also that the weight morphism $w_{T_G}^d$ is compatible with the forgetful morphism $F_{T_G,0}^*$, we deduce that the existence of a cartesian diagram  \eqref{E:g0pull}  for $\RPic\left(\mathrm{Bun}_{G,0,1}^{[d]}\right)$ implies the existence of the same cartesian diagram for $\RPic\left(\mathrm{Bun}_{G,0,0}^{[d]}\right)$.
\end{proof}

\section{Non-Reductive case}

The aim of the section is to show that  the  Picard group of $\bg{G}$ for an arbitrary  connected smooth linear algebraic group $G$ is isomorphic to the Picard group of $\bg{G^{\red}}$ where $G^{\red}$ is the reductive quotient of $G$. 

More precisely,  consider the  smooth, surjective, finite type morphism (see Corollary \ref{C:red-ft})
$$
\red_\#(=\red_{\#,g,n}):\bg{G}^{\delta}\to \bg{G^{\red}}^{\delta} \quad \text{ for any } \delta\in \pi_1(G)\cong \pi_1(G^{\red}),
$$
induced by the reductive quotient morphism $\red(=\red_G): G\twoheadrightarrow G^{\red}$ of \eqref{E:red-quot},  where we have used the canonical isomorphism $\pi_1(\red):\pi_1(G)\xrightarrow{\cong} \pi_1(G^{\red})$ of \eqref{E:pired}. The main result of this section is the following

\begin{teo}\label{T:BunG-nr}
For any $\delta\in \pi_1(G)\cong\pi_1(G^{\red})$,  we have that 
	$$
	\red_\#^*:\Pic(\bg{G^{\red}}^{\delta})\xrightarrow{\cong} \Pic(\bg{G}^{\delta})
	$$
	is an isomorphism. 
The same holds true for the homomorphism $\red_\#(C/S)^*$ for any family of curves $C\to S$, provided that $S$ is an integral and regular quotient stack over $k$.
\end{teo}

\begin{proof}
We will present the proof just for the universal case, the proof for the relative one uses the same argument. 

First of all, it  is enough to prove the Theorem for $n>0$.  Indeed, consider the following  cartesian diagrams of universal curves
\begin{equation}\label{E:redFG}
\xymatrix{
\mathrm{Bun}_{G,g,1}^{\delta}\ar@{}[dr]|\square\ar[d]_{F_G}\ar[r]^{\red_{\#,g,1}}&\ar[d]^{F_{G^{\red}}}\mathrm{Bun}_{G^{\red},g,1}^{\delta}  \ar@{}[drr]|\square \ar[rr]^{\Phi^{\delta}_{G^{\red},g,1}} && \Cg \ar[d]_{F}\\
\mathrm{Bun}_{G,g,0}^{\delta}\ar[r]_{\red_{\#,g,0}}&\mathrm{Bun}_{G^{\red},g,0}^{\delta}  \ar[rr]^{\Phi^{\delta}_{G^{\red},g,0}} && \Mg
}
\end{equation}
which, by pull-back, induces the following diagram of relative Picard groups
\begin{equation}\label{E:redFPic}
\xymatrix{
\Pic(\mathrm{Bun}_{G,g,1}^{\delta})&\ar[l]_{\red_{\#,g,1}^*}\Pic(\mathrm{Bun}_{G^{\red},g,1}^{\delta})\\
\Pic(\mathrm{Bun}_{G,g,0}^{\delta})\ar@{^{(}->}[u]^{F_G^*}&\Pic(\mathrm{Bun}_{G^{\red},g,0}^{\delta})\ar[l]^{\red_{\#,g,0}^*}\ar@{^{(}->}[u]_{F_{G^{\red}}^*}.
}
\end{equation}
By the seesaw principle, $F_G^*$  is injective and its image consists of the line bundles on $\mathrm{Bun}_{G,g,1}^{\delta}$ that are trivial on each geometric fiber of the family $F_G$, and similarly for $F_{G^{\red}}^*$. Since the family $F_G$ is a pull-back of the family $F_{G^{\red}}$ by \eqref{E:redFG}, we get that the diagram \eqref{E:redFPic} is cartesian (with injective vertical homomorphisms). Hence, if $\red_{\#,g,1}^*$ is an isomorphism, then $\red_{\#,g,0}^*$ is also an isomorphism.  Therefore, for the rest of the proof, we may assume $n>0$ and, in particular, $(g,n)\neq (1,0)$. 

\noindent\fbox{Surjectivity of $\red_\#^*$} This will follow from the more general

\underline{Claim:} Let $1\to U\to G\xrightarrow{\varphi} H\to 1$ be an exact sequence of connected smooth linear algebraic groups with $U$ unipotent. For any $\delta\in\pi_1(G)\xrightarrow[\cong]{\pi_1(\varphi)} \pi_1(H)$, the pull-back homomorphism 
$$
\varphi_\#^*:\Pic(\bg{H}^{\delta}) \to \Pic(\bg{G}^\delta) 
$$
is surjective. 

We now prove the claim. By Lemma \ref{L:lin-filtr}, the group $U$ admits a linearly filtered filtration
\begin{equation}\label{E:uni-filtr2}
\{1\}\subset U_r\subset\ldots\subset U_1\subset U_0=U.
\end{equation}
We proceed by induction on the length $r:=\text{Length}(U_\bullet)$ of the filtration. 

Assume first that $r=0$, i.e., that $U$ is vector group and the action of $G$ by conjugation is linear. 
For any connected smooth linear algebraic group $N$, we denote by $\bg{N}^{\delta,\leq m}$ the open substack in $\bg{N}^\delta$ of those $N$-bundles whose image in $\bg{N^{\red}}^{\delta}$ has instability degree less than or equal to $m$.  Observe that these open substacks are of finite type over $k$ by Proposition \ref{P:inst-cover} and Corollary \ref{C:red-ft}. Since the case $(g,n)=(1,0)$ is excluded by the assumption, $\bg{N}^{\delta,\leq m}$ is a quotient stack over $k$  by Proposition \ref{P:qstack-cover}. 

Consider the following cartesian diagrams of stacks
\begin{equation}\label{E:appro-phi}
\xymatrix{
X\times_{\bg{H}^{\delta,\leq 2}} \bg{G}^{\delta,\leq 2}:=Y\ar[d]^{\ov{\varphi}}\ar[r]\ar@{}[dr]|\square&\bg{G}^{\delta,\leq 2}\ar@{^{(}->}[r]^{\inst_G^{\leq 2}}\ar[d]^{\varphi_\#^{\leq 2}}\ar@{}[dr]|\square& \bg{G}^{\delta}\ar[d]^{\varphi_\#}\\
X\ar[r]^f&\bg{H}^{\delta,\leq 2}\ar@{^{(}->}[r]^{\inst_H^{\leq 2}}&\bg{H}^{\delta}
}
\end{equation}
where $\inst_G^{\leq 2}$ and $\inst_H^{\leq 2}$ are  open embeddings, $f$ is an equivariant approximation of the (smooth) $k$-quotient stack $\bg{H}^{\delta,\leq 2}$ (as in Proposition \ref{Equi}\eqref{Equi1}),  the square on the left is cartesian by definition and the square on the right  is cartesian since there is a  canonical isomorphism $H^{\red}\cong G^{\red}$  compatibile with the morphism $\varphi$, i.e., $\red_G=\red_H\circ \varphi$.
Observe that: 
\begin{itemize}
\item the horizontal arrows in the right square in \eqref{E:appro-phi} induce isomorphisms of  Picard groups by Proposition \ref{P:inst-cover}, Corollary \ref{C:red-ft} and Lemma \ref{restr};
\item the horizontal arrows in the left square in \eqref{E:appro-phi} induce isomorphisms of  Picard groups by Proposition \ref{Equi}\eqref{Equi3}. 
\end{itemize}
It follows that $\varphi_\#^*$ is surjective if and only if $\ov{\varphi}^*$ is surjective, i.e., 
$$
\RPic(\ov{\varphi}):=\Pic(Y)/\ov{\varphi}^*\Pic(X)=0.
$$
The morphism $\varphi_\#$ is surjective, smooth and of finite type  between integral $k$-smooth stacks by Corollary \ref{C:regint} and Proposition \ref{P:red-ft}, which then implies that  $\ov{\varphi}$ is a surjective, smooth and of finite type morphism between $k$-smooth integral algebraic spaces.
By \cite[Prop. 4.2.4]{BK}, the fiber of $\varphi_\#$ over a point $(C\to \Spec K,\un\sigma,F)\in \bg{H}^{\delta}(K)$ is the  stack $\mathrm{Bun}_{U_H^F}(C/K)$ of torsors over the curve $C\to \Spec(K)$ with respect to the vector bundle  $U_H^F:=(F\times U)/H\to C$.  Moreover, by \cite[Cor. 8.1.2]{BK}, we have an isomorphism of stacks
\begin{equation*}\label{E:fb-alm-vec}
\mathrm{Bun}_{U_H^F}(C/K)\cong [H^1(C,U_H^F)/H^0(C,U_H^F)],
\end{equation*}
where the underlying additive group of the vector space $H^0(C,U_H^F)$ acts trivially on the affine space $H^1(C,U_H^F)$. In particular, the map $\varphi_\#$ (and hence also $\ov{\varphi}$) has integral fibers. Hence, $\ov{\varphi}$ satisfies the hypothesis of Proposition \ref{P:fiber}, which then gives
$$
\RPic(\ov{\varphi})\cong \Pic(Y_\eta),
$$ with $Y_\eta$ being the fiber of $\ov{\varphi}$ over the generic point $\eta$ of $X$. By the above discussion,
$$
Y_\eta=\mathrm{Bun}_{\mt W}(\cC/\eta)\cong [H^1(\cC,\mt W)/H^0(\cC,\mt W)],
$$
for some vector bundle $\mt W$ over a curve $\cC\to\eta$. Since any line bundle over an affine space is trivial, we have
$$
\Pic(Y_\eta)\cong \Hom(H^0(\cC,\mt W),\Gm)=0,
$$
where the latter equality follows because $H^0(\cC,\mt W)\cong\mathbb G_a^{\dim H^0(\cC,\mt W)}$ as algebraic group over $\eta$. This concludes the proof for $r=0$. 

The case $r>0$ follows by splitting $\varphi_\#$ as a composition of morphism of stacks
$$
\varphi_\#: \bg{G}\to\bg{G/U_1}\to\bg{H},
$$
where $U_1$ is a proper subgroup of $U$ in the filtration \eqref{E:uni-filtr2} and using the inductive hypothesis. This concludes the proof of the claim.

\noindent\fbox{Injectivity of $\red_\#^*$} Let $T$ be a maximal torus of the reductive quotient $G^{\red}$. Then, the exact sequence $1\to G_u\to G\to G^{\red}\to 1$ restricted to the torus gives an exact sequence as follows
\begin{equation}\label{E:levi-dec}
1\to G_u\to P\to T\to 1.
\end{equation}
Since being solvable is preserved by extensions, $P$ is a connected smooth solvable group. By \cite[Thm 10.6]{Bo}, $P$ admits a Levi decomposition, i.e., $P= P_u\rtimes T_P$, where $P_u$ is the unipotent radical of $P$ and $T_P\subset P$ is a maximal torus. By \eqref{E:levi-dec}, we must have $P_u=G_u$ and $T_P=T$. In particular, the torus $\iota: T\hookrightarrow G^{\red}$ lifts to a (maximal) torus in $G$. Hence, the homomorphism 
$$
\iota_\#^*:\RPic(\bg{G^{\red}}^\delta)\to\RPic(\bg{T}^d), \text{ for any } d\in\pi_1(T)\text{ s.t. }\delta=[d]\in \pi_1(G)\cong\pi_1(G^{\red}),
$$
factors through $\red_\#^*:\RPic(\bg{G^{\red}}^\delta)\to \RPic(\bg{G}^\delta)$. By Corollary \ref{C:inj}, we can choose $d\in \pi_1(T)$ in such a way that  $\iota_\#^*$ is injective, which then implies that $\red_\#^*$ is also injective. 

\end{proof}

As a special case of Theorem \ref{T:BunG-nr}, we get the following 

\begin{cor}\label{C:BunG-nr}
Let $C$ be a (smooth, projective and connected) curve over $k=\ov k$. Then, for any $\delta\in \pi_1(G)\cong \pi_1(G^{\red})$, the homomorphism
$$
\red_\#^*: \Pic(\mathrm{Bun}_{G^{\red}}^\delta(C/k)) \xrightarrow{\cong}\Pic(\mathrm{Bun}_G^\delta(C/k))
$$
is an isomorphism. 
\end{cor}

\begin{rmk}It is challenging to extend to the non-reductive case the presentations for the relative Picard group $\RPic(\bg{G})$ given in \S\ref{Sec:tori} and \S\ref{Sec:non-ab-rd}. Since any linear algebraic group contains a maximal torus, the lattices $\Lambda(-)$ and $\Lambda^*(-)$ are well-defined in this setting. 
From this, it follows immediately that if $G$ is a solvable group, the Picard group admits a presentation as in the case of the tori in \S\ref{Sec:tori}. When $G$ is not solvable, we need a definition of Weil group. For a general linear algebraic group, the Weil group is defined as the quotient 
$$
\scr W_G:=\scr N(T)/\scr C(T)
$$
of the normalizer of the a maximal torus $T\subset G$ by the centralizer of $T$. This definition should be the right candidate for the generalizations of the results in \S\ref{Sec:non-ab-rd}.
\end{rmk}

\vspace{0,3cm}

\noindent {\bf Acknowledgments.} We are  grateful to Giulio Codogni, Martina Lanini, Johan Martens and Brent Pym for useful conversations. 

The first author acknowledges PRIN2017 CUP  E84I19000500006, and the MIUR Excellence Department Project awarded to the Department of Mathematics, University of Rome Tor Vergata, CUP E83C18000100006. 

The second author is a member of the CMUC (Centro de Matem\'atica da Universidade de Coimbra), where part of this work was carried over. 
The second author is member of the project EXPL/MAT-PUR/1162/2021 funded by FCT (Portugal), and of the projects PRIN-2017 ``Advances in Moduli Theory and Birational Classification"
and PRIN-2020 ``Curves, Ricci flat Varieties and their Interactions" funded by MIUR (Italy).

\bibliographystyle{alpha}
\bibliography{BiblioVec}
\end{document}